\documentclass[11pt]{article}

\usepackage{amsmath,amssymb,amsfonts,amsthm,enumitem,pstricks,float,xy,mathrsfs}
\xyoption{all}

\numberwithin{equation}{section}
\newtheorem{thm}{Theorem}[section]
\newtheorem{prp}[thm]{Proposition} 
\newtheorem{lmm}[thm]{Lemma}  
\newtheorem{crl}[thm]{Corollary}
\newtheorem*{prp*}{Proposition}
\theoremstyle{definition}
\newtheorem{dfn}[thm]{Definition}
\newtheorem{eg}[thm]{Example}
\newtheorem{exer}[thm]{Exercise}
\newtheorem{rmk}[thm]{Remark}

\def\BE#1{\begin{equation}\label{#1}}
\def\EE{\end{equation}}
\def\eref#1{(\ref{#1})}
\def\ov#1{\overline{#1}}
\def\ti#1{\widetilde{#1}}
\def\wh#1{\widehat{#1}}
\def\wt#1{\widetilde{#1}}
\def\sm#1{\begin{small}#1\end{small}}
\def\lr#1{\langle#1\rangle}
\def\blr#1{\big\langle#1\big\rangle}

\def\sf#1{\textsf{#1}}

\def\lra{\longrightarrow}
\def\Lra{\Longrightarrow}

\def\fa{\mathfrak a}
\def\C{\mathbb{C}}

\def\fc{\mathfrak c}
\def\ff{\mathfrak f}
\def\fI{\mathfrak i}
\def\fj{\mathfrak j}
\def\fm{\mathfrak m}
\def\cM{\mathcal M}
\def\fM{\mathfrak M}
\def\P{\mathbb{P}}
\def\cP{\mathcal P}
\def\Q{\mathbb Q}
\def\R{\mathbb R}
\def\cR{\mathcal R}
\def\bR{\mathbf R}
\def\bS{\mathbf S}
\def\bT{\mathbf T}
\def\cU{\mathcal U}
\def\cW{\mathcal W}
\def\fX{\mathfrak X}
\def\Z{{\mathbb Z}}

\def\al{\alpha}
\def\be{\beta}
\def\ga{\gamma}
\def\de{\delta}
\def\ep{\epsilon}

\def\la{\lambda}
\def\na{\nabla}
\def\om{\omega}
\def\si{\sigma}
\def\th{\theta}
\def\ve{\varepsilon}
\def\vph{\varphi}
\def\ze{\zeta}

\def\De{\Delta}
\def\Ga{\Gamma}
\def\La{\Lambda}
\def\Om{\Omega}
\def\Si{\Sigma}
\def\Th{\Theta}

\def\eset{\emptyset}
\def\i{\infty}
\def\prt{\partial}
\def\dbar{\bar\partial}
\def\w{\wedge}
\def\hb{\hbar}
\def\bu{\bullet}

\def\Aut{\textnormal{Aut}}
\def\diam{\textnormal{diam}}
\def\End{\textnormal{End}}
\def\ev{\textnormal{ev}}
\def\GL{\textnormal{GL}}

\def\id{\textnormal{id}}
\def\Id{\textnormal{Id}}
\def\Im{\textnormal{Im}}

\def\nd{\textnormal{d}}
\def\nD{\textnormal{D}}
\def\ne{\textnormal{e}}

\def\ord{\textnormal{ord}}

\def\PSL{\textnormal{PSL}}
\def\Sym{\textnormal{Sym}}

\def\vir{\textnormal{vir}}

\begin{document}

\title{Notes on $J$-Holomorphic Maps}
\author{Aleksey Zinger\thanks{Partially supported 
by NSF grants 0846978 and 1500875}}

\date{\today}

\maketitle

\begin{abstract}
\noindent
These notes present a systematic treatment of local properties of $J$-holomorphic maps
and of Gromov's convergence for sequences of such maps,
specifying the assumptions needed for all statements.
In particular, only one auxiliary statement depends on the manifold being symplectic. 
The content of these notes roughly corresponds to Chapters~2 and~4
of McDuff-Salamon's book on the subject.
\end{abstract}

\tableofcontents

\section{Introduction}
\label{intro_sec}

\noindent
Gromov's introduction~\cite{Gr} of pseudoholomorphic curves techniques into symplectic topology
has revolutionized this field and led to its numerous connections with algebraic geometry.
The ideas put forward in~\cite{Gr} have been further elucidated and developed
in \cite{Pansu,Ye,MS94,RT,RT2,LT} and in many other works.
The most comprehensive introduction to the subject of pseudoholomorphic curves
is without a doubt the monumental book~\cite{MS12}.
Chapters~2 and~4 of this book concern two of the three fundamental building blocks 
of this subject, the local structure of $J$-holomorphic maps and 
Gromov's convergence for sequences of $J$-holomorphic maps.
The present notes contain an alternative systematic exposition of these two topics
with generally sharper specification of the assumptions needed for each statement.
Chapter~3 and Sections~6.2 and~6.3 in~\cite{MS12} concern the third fundamental building block
of the subject, transversality for $J$-holomorphic maps.
A more streamlined and general treatment of this topic is the concern of~\cite{CmplRT}.\\

\noindent
The present notes build on the lecture notes on $J$-holomorphic maps
written for the class the author taught at Stony Brook University in Spring~2014.
The lectures themselves were based on the hand-written notes he made while
studying~\cite{MS94} back in graduate school and 
were also influenced by the more thorough exposition of the same topics in~\cite{MS12}.
The author would like to thank D.~McDuff and D.~Salamon for the time and care taken
in preparing and updating these books,
the students in the Spring~2014 class for their participation that
guided the preparation of the original version of the present notes, and
X.~Chen for thoughtful comments during the revision process.

\subsection{Stable maps}
\label{StabMaps_subs}

\noindent
A (smooth) \sf{Riemann surface} (without boundary) is a pair $(\Si,\fj)$ consisting
of a smooth two-dimensional manifold~$\Si$ (without boundary) and 
a complex structure~$\fj$ in the fibers of~$T\Si$.
A~\sf{nodal Riemann surface} is a pair $(\Si,\fj)$ obtained from 
a Riemann surface~$(\wt\Si,\fj)$ by identifying pairs of 
distinct points of~$\wt\Si$ in a discrete subset~$S_{\Si}$
(with no point identified with more than one other point);
see the left-hand sides of Figures~\ref{st_fig1} and~\ref{st_fig2}.
The pair $(\wt\Si,\fj)$ is called the \sf{normalization} of~$(\Si,\fj)$;
the images of the points of~$S_{\Si}$ in~$\Si$ are called the \sf{nodes} of~$\Si$.
We denote their complement  in~$\Si$ by~$\Si^*$.
An \sf{irreducible component} of~$(\Si,\fj)$ is the image of 
a topological component of~$\wt\Si$ in~$\Si$.
Let
$$\fa(\Si)= \frac{2-\chi(\wt\Si)+|S_{\Si}|}{2}\,,$$
where $\chi(\wt\Si)$ is the Euler characteristic of $\wt\Si$, be the 
(\sf{arithmetic}) \sf{genus} of~$\Si$.
An \sf{equivalence} between Riemann surfaces $(\Si,\fj)$ and $(\Si',\fj')$
is a homeomorphism $h\!:\Si\!\lra\!\Si'$ induced by a biholomorphic map~$\wt{h}$
from $(\wt\Si,\fj)$ to $(\wt\Si',\fj')$.
We denote by $\Aut(\Si,\fj)$ the group of automorphisms, 
i.e.~self-equivalences, of a Riemann surface $(\Si,\fj)$.\\ 

\noindent
Let $(X,J)$ be an almost complex manifold.
If $(\Si,\fj)$ is a Riemann surface, a smooth map $u\!:\Si\!\lra\!X$
is called \sf{$J$-holomorphic map} if
$$\nd u\!\circ\!\fj=J\!\circ\!\nd u\!:T\Si\lra u^*TX.$$
A \sf{$J$-holomorphic map} from a nodal Riemann surface $(\Si,\fj)$ 
is a tuple
$(\Si,\fj,u)$, where $u\!:\Si\!\lra\!X$ is a continuous map induced by 
a $J$-holomorphic map $\wt{u}\!:\wt\Si\!\lra\!X$;
see Figures~\ref{st_fig1} and~\ref{st_fig2}.
An \sf{equivalence} between $J$-holomorphic maps $(\Si,\fj,u)$ and $(\Si',\fj',u')$
is an equivalence 
$$h\!: (\Si,\fj)\lra(\Si',\fj')$$
between the underlying Riemann surfaces such that $u\!=\!u'\!\circ\!h$.
We denote by $\Aut(\Si,\fj,u)$ the group of automorphisms, 
i.e.~self-equivalences, of a $J$-holomorphic map $(\Si,\fj,u)$.
A $J$-holomorphic map $(\Si,\fj,u)$ is called \sf{stable} if 
$(\Si,\fj)$ is compact and  $\Aut(\Si,\fj,u)$ is a finite group.\\

\noindent
The Riemann surface $(\Si,\fj)$ on the left-hand side of Figure~\ref{st_fig1}
is obtained by identifying the marked points of two copies 
of a smooth elliptic curve $(\Si_0,\fj_0,z_1^*)$,
i.e.~a torus with a complex structure and a marked point.
The Riemann surface $(\Si_0,\fj_0)$ with the marked point~$z_1^*$ 
is biholomorphic to $\C/\La$ with the marked point~$0$  
for some lattice $\La\!\subset\!\C$
and thus has an automorphism of order~2 that preserves~$z_1^*$
(it is induced by the map $z\!\lra\!-z$ on~$\C$).
This is the only non-trivial automorphism of $(\Si_0,\fj_0)$ preserving~$z_1^*$ 
if~$\fj_0$ is generic;
in special cases, the group of such automorphisms is either $\Z_4$ or~$\Z_6$.
Each automorphism of $(\Si_0,\fj_0)$ preserving~$z_1^*$ gives rise to an automorphism
of $(\Si,\fj)$ fixing one of the irreducible components.
There is also an automorphism of $(\Si,\fj)$ which interchanges the two irreducible components
of~$\Si$.
Since it does not commute with the automorphisms preserving one of the components,
$\Aut(\Si,\fj)\!\approx\!D_4$ in most cases and contains~$D_4$ in the special cases.
If $u\!:\Si\!\lra\!\Si_0$ is the identity on each irreducible component,
$(\Si,\fj,u)$ is a stable $J$-holomorphic map; the interchange of the two 
irreducible components is then the only non-trivial automorphism of $(\Si,\fj,u)$.
The $J$-holomorphic maps  $u\!:\Si\!\lra\!\Si_0$ obtained by sending either or both
irreducible components of~$\Si$ to~$z_1^*$ instead are also stable, but have different
automorphism groups.
If $(\Si_0,\fj_0)$ were taken to be the Riemann sphere~$\P^1$, the  $J$-holomorphic map 
$u\!:\Si\!\lra\!\Si_0$ restricting to
the identity on each copy of~$\Si_0$ would still be stable.
However, a map $u\!:\Si\!\lra\!\Si_0$ sending either copy of~$\Si_0$ to~$z_1^*$
would not be stable, since the group of automorphisms of~$\P^1$
fixing a point is a complex two-dimensional submanifold of~$\PSL_2$.\\

\begin{figure}
\begin{pspicture}(-4.5,-5.5)(10,-2)
\psset{unit=.3cm}
\psellipse[linewidth=.08](0,-14)(5,2)
\pscircle*(5,-14){.25}
\psarc[linewidth=.08](-2.5,-17){3.16}{65}{115}
\psarc[linewidth=.08](-2.5,-11){3.16}{245}{295}
\psellipse[linewidth=.08](10,-14)(5,2)
\psarc[linewidth=.08](12.5,-17){3.16}{65}{115}
\psarc[linewidth=.08](12.5,-11){3.16}{245}{295}
\psellipse[linewidth=.08](27,-14)(5,2)
\psarc[linewidth=.08](24.5,-17){3.16}{65}{115}
\psarc[linewidth=.08](24.5,-11){3.16}{245}{295}
\psline[linewidth=.12]{->}(16.5,-14)(21.5,-14)
\rput(19,-13.4){$u$}
\rput(5,-17.5){$(\Si,\fj)\!=\!(\Si_0,\fj_0)\!\vee\!(\Si_0,\fj_0)$}
\rput(27,-17.5){$(\Si,\fj)\!=\!(\Si_0,\fj_0)$}
\rput(4,-14){\sm{$z_1^*$}}\rput(6,-14){\sm{$z_1^*$}}
\end{pspicture}
\caption{A stable $J$-holomorphic map}
\label{st_fig1}
\end{figure}
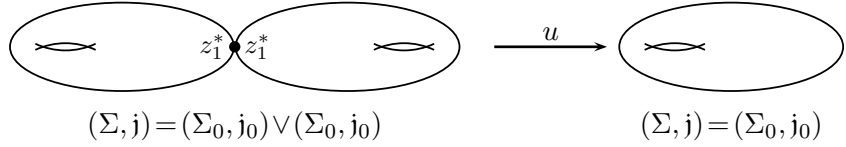

\noindent
Let $(\Si,\fj)$ be a compact connected Riemann surface of genus~$g$.
If $g\!\ge\!2$, then $\Aut(\Si,\fj)$ is a finite group.
If  $g\!=\!1$, then $\Aut(\Si,\fj)$ is an infinite group, but its subgroup
fixing any point is finite.
If  $g\!=\!0$, then the subgroup of  $\Aut(\Si,\fj)$ fixing any pair of points
is infinite, but the subgroup fixing any triple of points is trivial.
If in addition $(X,J)$ is an almost complex manifold
and $u\!:\Si\!\lra\!X$ is a non-constant $J$-holomorphic map, 
then the subgroup of $\Aut(\Si,\fj)$ consisting of the automorphisms
such that $u\!=\!u\!\circ\!h$ is finite;
this is an immediate consequence of Corollary~\ref{FHS_crl2}.
If $(\Si,\fj)$ is a compact nodal Riemann surface,
a $J$-holomorphic map $(\Si,\fj,u)$ is thus stable if and only~if
\begin{enumerate}[label=$\bullet$,leftmargin=*]

\item every  genus~1 topological component of the normalization~$\wt\Si$ 
of~$\Si$ such that $u$ restricts to a constant map on its image in~$\Si$
contains at least 1~element of~$S_{\Si}$ and

\item every genus~0 topological component of~$\wt\Si$ 
such that $u$ restricts to a constant map on its image in~$\Si$
contains at least 3 elements of~$S_{\Si}$.

\end{enumerate}

\subsection{Gromov's topology}
\label{GrTopol_subs}

\noindent
Given a Riemann surface $(\Si,\fj)$,
a Riemannian metric~$g$ on a smooth manifold~$X$ determines 
the \sf{energy} $E_g(f)$ for every smooth map $f\!:\Si\!\lra\!X$;
see \eref{Egfdfn_e} and~\eref{Egfdfn2_e}.
The fundamental insight in~\cite{Gr} that laid the foundations
for the pseudoholomorphic curves techniques in symplectic topology
and for the moduli spaces of stable maps and related curve-parametrizing objects
in algebraic geometry
is that a sequence of stable $J$-holomorphic maps $(\Si_i,\fj_i,u_i)$ 
into a compact almost complex manifold~$(X,J)$ with 
\BE{Ebndcond_e}
\liminf_{i\lra\i}\!\Big(\big|\pi_0(\Si_i)\big|\!+\!\fa(\Si_i)\!+\!E_g(u_i)\!\Big)<\i\EE
has a subsequence converging in a suitable sense to another 
stable $J$-holomorphic map.\\

\noindent
The notion of \sf{Gromov's convergence} of a sequence 
of stable $J$-holomorphic maps $(\Si_i,\fj_i,u_i)$
to another stable $J$-holomorphic map $(\Si_{\i},\fj_{\i},u_{\i})$ comes down~to
\begin{enumerate}[label=(GC\arabic*),leftmargin=*]

\item $|\pi_0(\Si_i)|\!=\!|\pi_0(\Si_{\i})|$ and $\fa(\Si_i)\!=\!\fa(\Si_{\i})$
for all~$i$ large,

\item $(\Si_{\i},\fj_{\i})$ is at least as singular as $(\Si_i,\fj_i)$
for all~$i$ large,

\item\label{GCenergy_it} 
the energy is preserved, i.e.~$E_g(u_i)\!\lra\!E_g(u_{\i})$ as $i\!\lra\!\i$, and

\item\label{GCmap_it} 
$u_i$ converges to~$u_{\i}$ uniformly in the $C^{\i}$-topology on compact 
subsets of~$\Si_{\i}^*$.

\end{enumerate}
Most applications of the pseudoholomorphic curves techniques in symplectic topology
involve \hbox{$J$-holomorphic} maps from the Riemann sphere~$\P^1$.
This is a special case of the situation when the complex structures~$\fj_i$ on
the domains~$\Si_i$ of~$u_i$ are fixed.
The condition~\ref{GCmap_it} can then be formally stated in a way  
clearly indicative of the rescaling procedure of~\cite{Gr}.

\begin{dfn}[Gromov's Compactness I]\label{GromConv_dfn1}
Let $(X,J)$ be an almost complex manifold with  Riemannian metric~$g$ and
$(\Si,\fj)$ be a compact Riemann surface. 
A sequence $(\Si,\fj,u_i)$ of stable \hbox{$J$-holomorphic} maps 
\sf{converges} to a stable $J$-holomorphic map $(\Si_{\i},\fj_{\i},u_{\i})$
if
\begin{enumerate}[label=(\arabic*),ref=\arabic*,leftmargin=*]

\item $(\Si_{\i},\fj_{\i})$ is obtained from~$(\Si,\fj)$ by identifying 
a point on each of $\ell$~trees of Riemann spheres~$\P^1$, for some $\ell\!\in\!\Z^{\ge0}$,
with distinct points $z_1^*,\ldots,z_{\ell}^*\!\in\!\Si$,

\item\label{GC1energy_it} $E_g(u_{\i})=\lim\limits_{i\lra\i}E_g(u_i)$,

\item there exist $h_i\!\in\!\Aut(\Si,\fj)$ with $i\!\in\!\Z^+$ such that 
$u_i\!\circ\!h_i$ converges to~$u_{\i}$ uniformly in the $C^{\i}$-topology on compact 
subsets of~$\Si\!-\!\{z_1^*,\ldots,z_{\ell}^*\}$,

\item\label{psi_it2} for each $z_1^*,\ldots,z_{\ell}^*\!\in\!\Si\!\subset\!\Si_{\i}$
and all $i\!\in\!\Z^+$ sufficiently large,
there exist a neighborhood $U_j\!\subset\!\Si$ of~$z_j^*$, 
an open subset $U_{j;i}\!\subset\!\C$, and
a biholomorphic map $\psi_{j;i}\!:U_{j;i}\!\lra\!U_j$   such~that 
\begin{enumerate}[label=(\ref{psi_it2}\alph*),leftmargin=*]

\item $U_i\!\subset\!U_{i+1}$ and 
$\C=\bigcup_{i=1}^{\i}U_{j;i}$ for every $j\!=\!1,\ldots,\ell$,

\item $u_i\!\circ\!h_i\!\circ\!\psi_{j;i}$ 
converges to~$u_{\i}$ uniformly in the $C^{\i}$-topology 
on compact subsets of the complement of the nodes $\i,w_{j;1}^*,\ldots,w_{j;k_j}^*$
in the sphere~$\P_j^1$ attached  at $z_j^*\!\in\!\Si$,

\item condition (\ref{psi_it2}) applies with
$\Si$, $(z_1^*,\ldots,z_{\ell}^*)$, and $u_i\!\circ\!h_i$
replaced by~$\P^1$, $(w_{j;1}^*,\ldots,w_{j;k_j}^*)$, and
$u_i\!\circ\!h_i\!\circ\!\psi_{j;i}$, respectively,
for each $j\!=\!1,\ldots,\ell$.

\end{enumerate}
\end{enumerate}
\end{dfn}

\begin{figure}
\begin{pspicture}(-6,-5.8)(-6,-2)
\psset{unit=.3cm}
\psellipse[linewidth=.08](0,-14)(5,2)\pscircle[linewidth=.08](7,-14){2} 
 \pscircle[linewidth=.08](0,-10){2}\pscircle*(0,-12){.25}
\pscircle[linewidth=.08](7,-10){2}\pscircle[linewidth=.08](7,-18){2}
\pscircle*(5,-14){.25}\pscircle*(7,-12){.25}\pscircle*(7,-16){.25}
\psarc[linewidth=.08](-2.5,-17){3.16}{65}{115}
\psarc[linewidth=.08](-2.5,-11){3.16}{245}{295}
\psline[linewidth=.12]{->}(12,-14)(18,-14)
\rput(-7.5,-14){$\Si_{\i}=$}\rput(20,-14){$X$}\rput(15,-13){$u_{\i}$}
\rput(0,-17){$\Si$}
\rput(0,-13){\sm{$z_2^*$}}\rput(0,-11){\sm{$\i$}}
\rput(4,-14){\sm{$z_1^*$}}\rput(6,-14){\sm{$\i$}}
\rput(7.5,-13){\sm{$w_{1;1}^*$}}\rput(7.1,-15){\sm{$w_{1;2}^*$}}
\rput(7,-11){\sm{$\i$}}\rput(7,-17){\sm{$\i$}}
\end{pspicture}
\caption{Gromov's limit of a sequence of $J$-holomorphic maps $u_i\!:\Si\!\lra\!X$}
\label{st_fig2}
\end{figure}

\vspace{.1in}

\noindent
An example of a possible limiting map with $\ell\!=\!2$ trees of spheres is shown
in Figure~\ref{st_fig2}.
The recursive condition~\eref{psi_it2} in Definition~\ref{GromConv_dfn1} 
is equivalent to the {\it Rescaling} axiom in \cite[Definition~5.2.1]{MS12}
on sequences of automorphisms~$\phi_{\al}^i$ of~$\P^1$; they correspond to compositions
of the maps~$\psi_{j;i}$ associated with different irreducible components
of~$\Si_{\i}$.
The single energy condition~\eref{GC1energy_it} in Definition~\ref{GromConv_dfn1} 
is replaced 
in \cite[Definition~5.2.1]{MS12} by multiple conditions of the {\it Energy} axiom.
These multiple conditions are equivalent to~\eref{GC1energy_it} 
if the other three axioms in \cite[Definition~5.2.1]{MS12} are satisfied.

\begin{thm}[Gromov's Compactness I]\label{GromConv_thm1}
Let $(X,J)$ be a compact almost complex manifold with Riemannian metric~$g$,
$(\Si,\fj)$ be a compact Riemann surface, and $u_i\!:\Si\!\lra\!X$ be a sequence
of non-constant $J$-holomorphic maps.
If $\liminf E_g(u_i)\!<\!\i$, then the sequence $(\Si,\fj,u_i)$ contains
a subsequence
converging to some stable $J$-holomorphic map $(\Si_{\i},\fj_{\i},u_{\i})$
in the sense of Definition~\ref{GromConv_dfn1}.
\end{thm}

\noindent
This theorem is established in Section~\ref{Conver_subs2} by assembling together 
a number of geometric statements obtained earlier in these notes.
In Section~\ref{PnEg_subs}, we relate the convergence notion of 
Definition~\ref{GromConv_dfn1} in the case of holomorphic maps from $\C\P^1$
to~$\C\P^n$, which can always be represented by $(n\!+\!1)$-tuples 
of homogeneous polynomials in two variables, 
to the behavior of the linear factors of the associated polynomials.\\

\noindent
The convergence notion of Definition~\ref{GromConv_dfn1} can be equivalently  reformulated 
in terms of deformations of the limiting domain~$(\Si_{\i},\fj_{\i})$
so that  it readily extends to sequences of stable $J$-holomorphic maps 
with varying complex structures~$\fj_i$ on the domains~$\Si_i$.
This was formally done in the algebraic geometry category by~\cite{FPa},
several years after this perspective had been introduced into the field informally,
and adapted to the almost complex category by~\cite{LT}.
We summarize this perspective below.\\

\noindent
Let $(\Si,\fj)$ be a nodal Riemann surface. 
A \sf{flat family of deformations} of $(\Si,\fj)$ is a holomorphic map
$\pi\!:\cU\!\lra\!\De$, where $\cU$ is a complex manifold and $\De\!\subset\!\C^N$
is a neighborhood of~$0$, such~that
\begin{enumerate}[label=$\bullet$,leftmargin=*]

\item $\pi^{-1}(\la)$ is a nodal Riemann surface for each $\la\!\in\!\C^n$ and
$\pi^{-1}(0)\!=\!(\Si,\fj)$,

\item  $\pi$ is a submersion outside of the nodes of the fibers of~$\pi$,

\item for every $\la^*\!\equiv\!(\la_1^*,\ldots,\la_N^*)\!\in\!\De$ and 
every node $z^*\!\in\!\pi^{-1}(\la^*)$, 
there exist $i\!\in\!\{1,\ldots,N\}$ with $\la_i\!=\!0$,
neighborhoods $\De_{\la^*}$ of~$\la^*$ in~$\De$ and $\cU_{z^*}$ of~$z^*$ in~$\cU$, 
and a holomorphic~map
$$\Psi\!:\cU_{z^*}\lra \big\{\big((\la_1,\ldots,\la_N),x,y\big)\!\in\!\De_{\la^*}\!\times\!\C^2\!:
xy\!=\!\la_i\big\}$$
such that $\Psi$ is a homeomorphism onto a neighborhood of $(\la^*,0,0)$
and the composition of~$\Psi$ with the projection to~$\De_{\la^*}$ equals~$\pi|_{\cU_{z^*}}$.
\end{enumerate} 

\vspace{.1in}

\noindent
If $\pi\!:\cU\!\lra\!\De$ is a flat family of deformations of $(\Si,\fj)$ and $\Si$ is compact,
there exists a neighborhood $\cU^*\!\subset\!\cU$ of $\Si^*\!\subset\!\pi^{-1}(0)$
such~that 
$$\pi|_{\cU^*}\!:\cU^*\lra \De_0\!\equiv\!\pi(\cU^*)\subset\De$$
is a trivializable $\Si^*$-fiber bundle in the smooth category.
For each $\la\!\in\!\De_0$, let 
$$\psi_{\la}\!:\Si^*\lra \pi^{-1}(\la)\!\cap\!\cU^*$$
be the corresponding smooth identification.
If $\la_i\!\in\!\De$ is a sequence converging to $0\!\in\!\De$
and \hbox{$u_i\!:\pi^{-1}(\la_i)\!\lra\!X$} is a sequence of continuous maps
that are smooth on the complements of the nodes of~$\pi^{-1}(\la_i)$,
we say that the sequence~$u_i$ \sf{converges to} a smooth map $u\!:\!\Si^*\!\lra\!X$
{\sf u.c.s.} if the sequence of maps
$$u_i\!\circ\!\psi_{\la_i}\!:\Si^*\lra X$$
converges uniformly in the $C^{\i}$-topology on compact  subsets of~$\Si^*$.
This notion is independent of the choices of~$\cU^*$ and 
trivialization of~$\pi|_{\cU^*}$.

\begin{figure}
\begin{pspicture}(-7,-7)(-6,-2)
\psset{unit=.3cm}
\pscircle*(-10,-12){.25}\pscircle*(-10,-16){.25}
\psline[linewidth=.08](-10,-10)(-10,-18)
\psline[linewidth=.08](-11,-14)(-8,-8)\psline[linewidth=.08](-9,-14)(-12,-20)
\psline[linewidth=.08](-14,-8)(-14,-20)\psline[linewidth=.08](-17,-8)(-17,-20)
\psline[linewidth=.08](-6,-8)(-6,-20)\psline[linewidth=.08](-3,-8)(-3,-20)
\psline[linewidth=.05,linestyle=dashed](-18.5,-11)(-1.5,-11)
\psline[linewidth=.05,linestyle=dashed](-18.5,-13)(-1.5,-13)
\psline[linewidth=.05,linestyle=dashed](-18.5,-11)(-18.5,-13)
\psline[linewidth=.05,linestyle=dashed](-1.5,-11)(-1.5,-13)
\psline[linewidth=.08](-19,-22)(-1,-22)
\pscircle*(-10,-22){.25}\pscircle*(-6,-22){.25}\pscircle*(-3,-22){.25}
\rput(-10,-23){$0$}\rput(-5.8,-23){$\la_{20}$}\rput(-2.8,-23){$\la_{10}$}
\rput(-20,-22){$\De$}\rput(-20,-14){$\cU$}\rput(-19.3,-18){$\pi$}
\psline[linewidth=.05]{->}(-20,-15)(-20,-21)
\rput(0.6,-12){$B_{\de}(z^*)$}\rput(25.6,-12){$B_{\de}(z^*)$}
\rput(-8.9,-11.9){\sm{$z^*$}}\rput(10,-11.9){\sm{$z^*$}}
\psline[linewidth=.04]{->}(9.4,-12)(8.1,-12)
\pscircle[linewidth=.08](7,-14){2} 
\pscircle[linewidth=.08](7,-10){2}\pscircle[linewidth=.08](7,-18){2}
\pscircle*(7,-12){.25}\pscircle*(7,-16){.25}
\psline[linewidth=.05,linestyle=dashed](4.5,-11)(23.5,-11)
\psline[linewidth=.05,linestyle=dashed](4.5,-13)(23.5,-13)
\psline[linewidth=.05,linestyle=dashed](4.5,-11)(4.5,-13)
\rput(7.2,-21){$\Si_{\i}$}
\psarc[linewidth=.08](14,-14){2}{-60}{60}
\psarc[linewidth=.08](14,-14){2}{120}{240} 
\psarc[linewidth=.08](14,-10){2}{-60}{240}
\psarc[linewidth=.08](14,-18){2}{120}{60}
\psarc[linewidth=.08](15.15,-12){.31}{120}{240} 
\psarc[linewidth=.08](12.85,-12){.31}{-60}{60} 
\psarc[linewidth=.08](15.15,-16){.31}{120}{240} 
\psarc[linewidth=.08](12.85,-16){.31}{-60}{60} 
\rput(14.4,-21){$\Si_{\la_{20}}$}
\psarc[linewidth=.08](21,-14){2}{-45}{45}
\psarc[linewidth=.08](21,-14){2}{135}{225} 
\psarc[linewidth=.08](21,-10){2}{-45}{225}
\psarc[linewidth=.08](21,-18){2}{135}{45}
\psarc[linewidth=.08](23,-12){.82}{135}{225} 
\psarc[linewidth=.08](19,-12){.82}{-45}{45} 
\psarc[linewidth=.08](23,-16){.82}{135}{225} 
\psarc[linewidth=.08](19,-16){.82}{-45}{45} 
\psline[linewidth=.05,linestyle=dashed](23.5,-11)(23.5,-13)
\rput(21.4,-21){$\Si_{\la_{10}}$}
\end{pspicture}
\caption{A complex-geometric presentation of a flat family of deformations of 
$(\Si_{\i},\fj_{\i})\!=\!\pi^{-1}(0)$ and a differential-geometric presentation
of the domains of the maps~$u_i$ in Definition~\ref{GromConv_dfn2}.}
\label{st_fig3}
\end{figure}
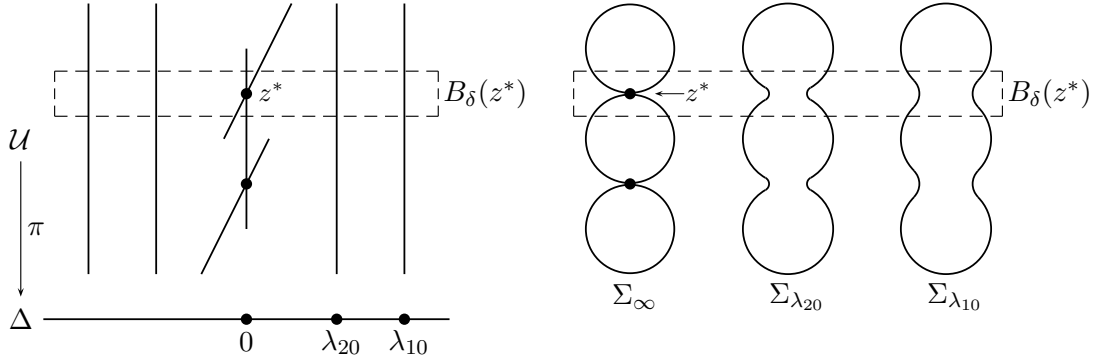

\begin{dfn}[Gromov's Convergence II]\label{GromConv_dfn2}
Let $(X,J)$ be an almost complex manifold with  Riemannian metric~$g$.
A sequence $(\Si_i,\fj_i,u_i)$ of stable \hbox{$J$-holomorphic} maps 
\sf{converges} to a stable $J$-holomorphic map $(\Si_{\i},\fj_{\i},u_{\i})$
if $E_g(u_{\i})\!\lra\!E_g(u_i)$ as $i\!\lra\!\i$ and there exist

\begin{enumerate}[label=(\alph*),leftmargin=*]

\item  a flat family of deformations $\pi\!:\cU\!\lra\!\De$ of $(\Si_{\i},\fj_{\i})$,

\item a sequence $\la_i\!\in\!\De$ converging to $0\!\in\!\De$, and

\item equivalences $h_i\!:\pi^{-1}(\la_i)\!\lra\!(\Si_i,\fj_i)$

\end{enumerate}
such that $u_i\!\circ\!h_i$ converges to~$u_{\i}|_{\Si_{\i}^*}$ u.c.s.
\end{dfn}

\noindent
By the compactness of~$\Si_{\i}$, the notion of convergence of Definition~\ref{GromConv_dfn2}
is independent of the choice of metric~$g$ on~$X$.
It is illustrated in Figure~\ref{st_fig3}.
If the Riemann surfaces $(\Si_i,\fj_i)$ are smooth,
the limiting Riemann surface $(\Si_{\i},\fj_{\i})$ is obtained by pinching
some disjoint embedded circles in the smooth two-dimensional manifold~$\Si$
underlying these Riemann surfaces.\\

\noindent
If \hbox{$(\Si_i,\fj_i)\!=\!(\Si,\fj)$} for all~$i$ as in Definition~\ref{GromConv_dfn1},
only contractible circles are pinched to produce~$\Si_{\i}$;
it then consists of~$\Si$ with trees of spheres attached.
The family $\pi\!:\cU\!\lra\!\De$ is obtained by starting with the family
$$\pi_0\!:\cU_0\!\equiv\!\C\!\times\!\Si\lra\C,$$
then blowing up~$\cU_0$ at a point of $\{0\}\!\times\!\Si$ 
to obtain a family $\pi_1\!:\cU_1\!\lra\!\C$
with the central fiber \hbox{$\Si_1\!\equiv\!\pi_1^{-1}(0)$} consisting of $\Si$ with~$\P^1$ attached,
then blowing up a smooth point of~$\Si_1$, and so~on.
The number of blowups involved is precisely the number of nodes of~$\Si_{\i}$,
i.e.~four in the case of Figure~\ref{st_fig2} and two in the case of Figure~\ref{st_fig3}.
The pinched annuli on the right-hand side of Figure~\ref{st_fig3} correspond 
to $\phi_{\al}(B_{\de}(z_{\al\be}))\!\cup\!\phi_{\be}(B_{\de}(z_{\be\al}))$
in the notation of \cite[Chapters~4,5]{MS12}.\\

\noindent
With the setup of Definition~\ref{GromConv_dfn2}, let $B_{\de}(z^*)\!\subset\!\cU$
denote the ball of radius $\de\!\in\!\R^+$ around a point $z^*\!\in\!\cU$
with respect to some metric on~$\cU$.
Then,
\BE{nobubbles_e}\lim_{\de\lra0}\lim_{i\lra\i}\diam_g\big(
u_i\big(h_i(\pi^{-1}(\la_i)\!\cap\!B_{\de}(z^*))\big)\!\big)=0
\qquad\forall~z^*\!\in\!\Si_{\i}\,.\EE
This is immediate from the last condition in Definition~\ref{GromConv_dfn2} 
if $z^*\!\in\!\Si_{\i}^*$.
If $z^*\!\in\!\Si_{\i}\!-\!\Si_{\i}^*$ is a node of~$\Si_{\i}$, 
\eref{nobubbles_e} is a consequence of both convergence
conditions of Definition~\ref{GromConv_dfn2}  and the maps~$u_i$ being $J$-holomorphic.
It is a reflection of the fact that \sf{bubbling} or any other kind of erratic $C^0$-behavior
of a sequence of $J$-holomorphic maps requires a nonzero amount of energy in the limit,
but the two convergence conditions of Definition~\ref{GromConv_dfn2} ensure that 
all limiting energy is absorbed by~$u|_{\Si_{\i}^*}$ and thus none is left for bubbling
around the nodes of~$\Si_{\i}$.
An immediate implication of~\eref{nobubbles_e} is that 
$u_i(h_i(\pi^{-1}(\la_i)\!\cap\!B_{\de}(z^*)))$ is contained in a geodesic ball
around $u_{\i}(z^*)$ in~$X$.
Thus,
$$u_{i*}\big[\Si_i\big]=u_{\i*}[\Si_{\i}]\in H_2(X;\Z)$$
for all $i\!\in\!\Z^+$ sufficiently large.
If $\Si_{\i}$ is a tree of spheres (and thus so is each~$\Si_i$),
then $u_i$ with $i$ sufficiently large lies in the equivalence class
in $\pi_2(X)$ determined by~$u_{\i}$ for the same reason.

\begin{thm}[Gromov's Compactness II]\label{GromConv_thm2}
Let $(X,J)$ be a compact almost complex manifold with Riemannian metric~$g$
and $(\Si_i,\fj_i,u_i)$ be a sequence of stable $J$-holomorphic maps 
into a compact almost complex manifold~$(X,J)$.
If it satisfies~\eref{Ebndcond_e},
then it contains a subsequence 
converging to some stable $J$-holomorphic map $(\Si_{\i},\fj_{\i},u_{\i})$
in the sense of Definition~\ref{GromConv_dfn2}.
\end{thm}

\noindent
This theorem is obtained by combining the compactness of the Deligne-Mumford moduli spaces
$\ov\cM_{1,1}$ of stable (possibly) nodal elliptic curves and $\ov\cM_g$ of stable nodal genus 
$g\!\ge\!2$ curves with the proof of Theorem~\ref{GromConv_thm1} in Section~\ref{Conver_subs2}.
One first establishes Theorem~\ref{GromConv_thm2} under the assumption that 
each $(\Si_i,\fj_i)$ is a smooth connected Riemann surface of genus $g\!\ge\!1$
(the $g\!=\!0$ case is treated by Theorem~\ref{GromConv_thm1}).
If $g\!=\!1$, we add a marked point to each domain $(\Si_i,\fj_i)$ and take 
a subsequence converging in $\ov\cM_{1,1}$ to the equivalence class
of some stable nodal elliptic curve $(\Si_{\i}',\fj_{\i}',z_{\i}')$.
If $g\!\ge\!2$, we take a subsequence of $(\Si_i,\fj_i)$ converging in $\ov\cM_g$
to the equivalence class
of some stable nodal genus~$g$  curve $(\Si_{\i}',\fj_{\i}')$.
This ensures the existence of a flat family of deformations 
$\pi'\!:\cU'\!\lra\!\De'$ of $(\Si_{\i}',\fj_{\i}')$,
of a sequence $\la_i'\!\in\!\De'$ converging to $0\!\in\!\De'$, and
of equivalences $h_i\!:\pi'^{-1}(\la_i')\!\lra\!(\Si_i,\fj_i)$.
The associated neighborhood $\cU'^*$ of $\Si_{\i}'^*$ in~$\cU'$ can be chosen
so that $\pi'^{-1}(\la)\!-\!\cU'^*$ consists of finitely many circles
for every $\la'\!\in\!\De'$ sufficiently small.
The complement of the image of the associated identifications
$$\psi_{\la}'\!:\Si_{\i}'^*\lra \pi'^{-1}(\la)\!\cap\!\cU'^*$$
in $\pi'^{-1}(\la)$ has the same property.\\

\noindent
One then applies the construction in the proof of Theorem~\ref{GromConv_thm1}
to the sequence of
$J$-holomorphic maps 
$$u_i\!\circ\!h_i'\!:\Si_{\i}'^*\lra X$$
to obtain a $J$-holomorphic map $\wt{u}_{\i}'$ from the normalization 
$\wt\Si_{\i}'$ of $\Si_{\i}$ and finitely $J$-holomorphic maps from
trees of~$\P^1$. 
Each of these trees will have one or two special points that are associated 
with points of~$\wt\Si_{\i}'$ (the latter happens if bubbling occurs at
a preimage of a node of~$\Si_{\i}'$ in~$\wt\Si_{\i}'$). 
Identifying these trees with the corresponding points of~$\wt\Si_{\i}'$ as 
in the proof of Theorem~\ref{GromConv_thm1}, 
we obtain a $J$-holomorphic map $(\Si_{\i},\fj_{\i},u_{\i})$
satisfying the requirements of Definition~\ref{GromConv_dfn2}.
It is necessarily stable if $g\!\ge\!2$, or $\Si_{\i}'$ is smooth, or
$\Si_{\i}$ contains a separating node.
Otherwise, the identifications~$h_i'$ may first need to be reparametrized to ensure
that either the limiting map $\wt{u}_{\i}'$ is not constant or the sequence
$u_i\!\circ\!h_i$ produces a bubble at least one smooth point of~$\wt\Si_{\i}'$.\\

\noindent
A \sf{$k$-marked Riemann surface} is a tuple $(\Si,\fj,z_1,\ldots,z_k)$
such that $(\Si,\fj)$ is a Riemann surface and $z_1,\ldots,z_k\!\in\!\Si^*$ 
are distinct points.
If $(X,J)$ is an almost complex manifold, 
a \sf{$k$-marked \hbox{$J$-holomorphic} map} into~$X$ is a tuple $(\Si,\fj,z_1,\ldots,z_k,u)$,
where $(\Si,\fj,z_1,\ldots,z_k)$ is $k$-marked Riemann surface
and  $(\Si,\fj,u)$ is a $J$-holomorphic map into~$X$.
The \sf{degree} of such a map is the homology~class
$$A=u_*[\Si]\in H_2(X;\Z)\,.$$
The notions of equivalence, stability, and convergence as in Definition~\ref{GromConv_dfn2}
and the above convergence argument for smooth domains $(\Si_i,\fj_i)$
readily extend to $k$-marked $J$-holomorphic maps.
The general case of Theorem~\ref{GromConv_thm2}, including its extension to stable marked maps, 
is then obtained by 
\begin{enumerate}[label=$\bullet$,leftmargin=*]

\item passing to a subsequence of $(\Si_i,\fj_i,u_i)$ with the same
topological structure of the domain,

\item viewing it as a sequence of tuples of $J$-holomorphic maps with smooth domains
with an additional marked point for each preimage of the nodes in the normalization, and 

\item applying the conclusion of the above argument to each component of the tuple.

\end{enumerate}

\subsection{Moduli spaces}
\label{ModSp_subs}

\noindent
The natural extension of Definition~\ref{GromConv_dfn2} to marked $J$-holomorphic maps
topologizes
the moduli space $\ov\fM_{g,k}(X,A;J)$ of equivalence classes
of stable  degree~$A$ $k$-marked genus~$g$ $J$-holomorphic maps into $X$ 
for each $A\!\in\!H_2(X;\Z)$.
The evaluation~maps
$$\ev_i\!:\ov\fM_{g,k}(X,A;J)\lra X, \quad
(\Si,\fj,z_1,\ldots,z_k,u)\lra u(z_i),$$
are continuous with respect to this topology.
If $2g\!+\!k\!\ge\!3$, there is a continuous map
$$\ff\!:\ov\fM_{g,k}(X,A;J)\lra \ov\cM_{g,k}$$
to the Deligne-Mumford moduli space of stable $k$-marked genus~$g$ nodal curves
obtained by forgetting the map~$u$ and then contracting the unstable components 
of the domain.\\

\noindent
There is a continuous map
\BE{ffkdfn_e}\ff_{k+1}\!:\ov\fM_{g,k+1}(X,A;J)\lra\ov\fM_{g,k}(X,A;J)\EE
obtained by forgetting the last marked point~$z_{k+1}$ and then contracting 
the components of the domain to stabilize the resulting $k$-marked
$J$-holomorphic map.
For each $i\!=\!1,\ldots,k$, this fibration has a natural continuous section
$$s_i\!: \ov\fM_{g,k}(X,A;J)\lra \ov\fM_{g,k+1}(X,A;J)$$
described as follows.
For a $k$-marked nodal Riemann surface $(\Si,\fj,z_1,\ldots,z_k)$, 
let $(\Si',\fj',z_1,\ldots,z_{k+1})$ be the $(k\!+\!1)$-marked nodal Riemann surface
so that $(\Si',\fj')$ consists of $(\Si,\fj)$ with $\P^1$ attached at~$z_i$,
$z_1',z_i'\!\in\!\P^1$, and $z_j'=\!z_j\!\in\!\Si$ for all $j\!=\!1,\ldots,k$
different from~$k$; see Figure~\ref{st_fig4}.
We define
$$s_i\big([\Si,\fj,z_1,\ldots,z_k,u]\big]=\big[\Si',\fj',z_1',\ldots,z_{k+1}',u'\big],$$
with $(\Si',\fj',z_1',\ldots,z_{k+1}')$ as described and $u'$ extending $u$ 
over the extra~$\P^1$ by  the constant map with value~$u(z_i)$.
The pullback
$$L_i\lra  \ov\fM_{g,k}(X,A;J)$$
of the vertical tangent line bundle of~\eref{ffkdfn_e} by $s_i$ is called
the \sf{universal tangent line bundle} at the $i$-th marked point.
Let $\psi_i\!=\!c_1(L_i^*)$ be the \sf{$i$-th descendant class}.\\

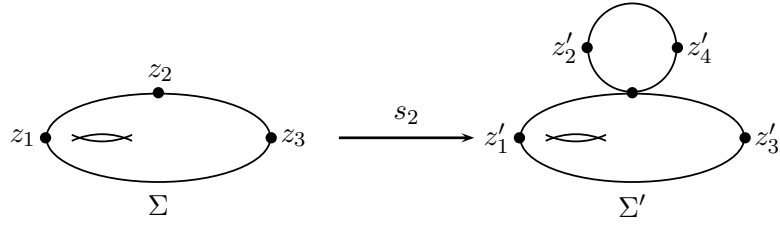
\begin{figure}
\begin{pspicture}(-5.5,-5.2)(0,-2)
\psset{unit=.3cm}
\psellipse[linewidth=.08](0,-14)(5,2)
\pscircle*(0,-12){.25}\pscircle*(5,-14){.25}\pscircle*(-5,-14){.25}
\psarc[linewidth=.08](-2.5,-17){3.16}{65}{115}
\psarc[linewidth=.08](-2.5,-11){3.16}{245}{295}
\psline[linewidth=.12]{->}(8,-14)(14,-14)
\rput(11,-13){$s_2$}\rput(0,-17){$\Si$}
\rput(0.1,-11.1){$z_2$}\rput(-6,-14){$z_1$}\rput(6,-14){$z_3$}
\psellipse[linewidth=.08](21,-14)(5,2)\pscircle[linewidth=.08](21,-10){2}
\pscircle*(21,-12){.25}\pscircle*(26,-14){.25}\pscircle*(16,-14){.25}
\pscircle*(19,-10){.25}\pscircle*(23,-10){.25}
\psarc[linewidth=.08](18.5,-17){3.16}{65}{115}
\psarc[linewidth=.08](18.5,-11){3.16}{245}{295}
\rput(15,-14){$z_1'$}\rput(27,-14){$z_3'$}
\rput(18,-10){$z_2'$}\rput(24,-10){$z_4'$}
\rput(21,-17){$\Si'$}
\end{pspicture}
\caption{Section $s_2$ of the fibration~\eref{ffkdfn_e} with $k\!=\!3$}
\label{st_fig4}
\end{figure}

\noindent
A  remarkable property of Gromov's topology which lies behind most of its applications
is that the moduli space $\ov\fM_{g,k}(X,A;J)$ is Hausdorff and has a particularly 
nice deformation-obstruction theory.
In the algebraic-geometry category, the latter is known as 
a perfect \sf{two-term deformation-obstruction theory}.
In the almost complex category, this is reflected in the existence of 
an \sf{atlas of finite-dimensional approximations} in the terminology of~\cite{LT}
or of  an \sf{atlas of Kuranishi charts} in the terminology of~\cite{LT}.\\

\noindent
If $(X,J)$ is an almost complex manifold and $J$ is tamed by a symplectic form~$\om$,
then the energy $E_g(u)$ of degree~$A$ $J$-holomorphic map~$u$ with respect to
the metric~$g$ determined by~$J$ and~$\om$ is $\om(A)$;
see~\eref{Egexp_e}.
In particular, it is the same for all elements of the moduli space $\ov\fM_{g,k}(X,A;J)$.
If in addition $X$ is compact, then Theorem~\ref{GromConv_thm2} implies 
that this moduli space is also compact.
Combining this with the remarkable property of the previous paragraph,
the constructions of \cite{BF,LT0,LT,FO} endow $\ov\fM_{g,k}(X,A;J)$
with a \sf{virtual fundamental class}.
It depends only on~$\om$, in a suitable sense, and not an almost complex structure~$J$
tamed by~$\om$.
This class in turn gives rise to \sf{Gromov-Witten invariants} of~$(X,\om)$:
$$\blr{\tau_{a_1}\!\al_1,\ldots,\tau_{a_k}\!\al_k}_{g,A}^X
\equiv \blr{\big(\psi_1^{a_1}\!\ev_1^*\al_1)\ldots\big(\psi_k^{a_k}\!\ev_k^*\al_k),
\big[\ov\fM_{g,k}(X,A;J)\big]^{\vir}}\in\Q$$
for all $a_i\!\in\!\Z^{\ge0}$  and $\al_i\!\in\!H^*(X;\Q)$.

\section{Preliminaries}
\label{prelim_sec}

\noindent
An outline of these notes with an informal description of the key statements 
appears in Section~\ref{Overview_subs};
Figure~\ref{connect_fig} indicates primary connections between these statements.
Sections~\ref{NotTerm_subs} introduces the most frequently used notation and terminology
and makes some basic observations.

\subsection{Overview of the main statements}
\label{Overview_subs}

\noindent
The main technical statement of Section~\ref{LocProp_sec} of these notes and 
of Chapter~2 in~\cite{MS12}
is the \sf{Carleman Similarity Principle}; see Proposition~\ref{FHS_prp}.
It yields a number of geometric conclusions about the local behavior
of a $J$-holomorphic map $u\!:\Si\!\lra\!X$ from a Riemann surface~$(\Si,\fj)$ into
an almost complex manifold~$(X,J)$.
For example, for every $z\!\in\!\Si$ contained in a component of~$\Si$ on which~$u$ is not constant,
the $\ell$-th derivative of~$u$ at~$z$ in a chart around~$u(z)$ does not vanish for
some $\ell\!\in\!\Z^+$; see Corollary~\ref{FHS_crl1}.
We denote by $\ord_zu\!\in\!\Z^+$ the minimum of such integers~$\ell$ and
call it the \sf{order of~$u$ at~$z$};
it is independent of the choice of a chart around~$u(z)$.
If $u$ is constant on the component of~$\Si$ on containing~$z$, we set $\ord_zu\!=\!0$.
A point $z\!\in\!u$ is \sf{singular}, i.e.~$\nd_zu\!=\!0$, if and only if 
$\ord_zu\!\neq\!1$.\\

\noindent
If $u$ is not constant on every connected component of~$\Si$,
the singular points of~$u$ and the preimages of a point $x\!\in\!X$ 
are discrete subsets of~$\Si$; see Corollary~\ref{FHS_crl2}. 
In the case $\Si$ is compact, the second statement of Corollary~\ref{FHS_crl2} implies~that 
\BE{ordxudfn_e}\ord_xu\!\equiv\!\sum_{z\in u^{-1}(x)}\!\!\!\!\!\!\ord_zu
~~\in\Z^{\ge0}
\qquad\forall\,x\!\in\!X;\EE
we call this number the \sf{order of~$u$ at~$x$}.
If $x\!\not\in\!\Im(u)$, then $\ord_xu\!=\!0$.
By Corollary~\ref{FHS_crl3}, the number~\eref{ordxudfn_e} is seen by the behavior of 
the energy~\eref{Egfdfn_e} of~$u$ and its restrictions to open subsets of~$\Si$.
This observation underpins the \sf{Monotonicity Lemma} for $J$-holomorphic maps,
which bounds below the energy required to ``escape" from a small ball in~$X$;
see Proposition~\ref{MonotLmm_prp}.\\

\noindent
The main technical statement of Section~\ref{EBnd_sec} of these notes and 
of Chapter~4 in~\cite{MS12} is the \sf{Mean Value Inequality}.
It bounds the pointwise differentials $\nd_zu$ of a $J$-holomorphic map~$u$ 
from~$(\Si,\fj)$ into~$(X,J)$ of sufficiently small energy~$E_g(u)$ by~$E_g(u)$,
i.e.~by the $L^2$-norm of~$\nd u$, from above and immediately yields a bound
on the energy of non-constant $J$-holomorphic maps from~$S^2$ into~$(X,J)$ from below; 
see Proposition~\ref{PtBnd_prp} and Corollary~\ref{LowEner_crl}, respectively.
The Mean Value Inequality also implies that the energy of a $J$-holomorphic map~$u$ 
from a cylinder $[-R,R]\!\times\!S^1$ carried by \hbox{$[-R\!+\!T,R\!-\!T]\!\times\!S^1$}
and the diameter of the image of this middle segment decay at least exponentially 
with~$T$, provided the overall energy of~$u$ is sufficiently small.
As shown in the proof of Proposition~\ref{EnerPres_prp}, 
this technical implication ensures that \sf{the energy is preserved} under Gromov's convergence
and the resulting \sf{bubbles connect}.\\

\begin{figure}
$$\xymatrix{\hbox{\ref{FHS_prp}} \ar[dd]\ar[dr]&&&&
\hbox{\ref{PtBnd_prp}}\ar[dl]\ar[ddl]\ar[dr]\ar[dd]\\
&\hbox{\ref{FHS_crl2}}\ar[r] &\hbox{\ref{GlobStr_subs}}& 
\hbox{\ref{JHolomreg_prp}} \ar[l]\ar[dl]& &\hbox{\ref{LowEner_crl}}\ar[dl]\ar[ddl]\\
\hbox{\ref{MonotLmm_prp}}\ar[rr]\ar@/_1pc/[rrr] &&   
\hbox{\ref{RemSing_prp}}\ar@/_1pc/[rr]|<<<{\hole} &\hbox{\ref{CylEner_crl}}\ar[r]
&\hbox{\ref{EnerPres1_lmm}-\ref{EnerPres3_lmm}}
\ar[d]\ar[dr]|<<<<<{\hole}\\
&&&&\hbox{\ref{EnerPres_prp}}\ar[r]& \hbox{\ref{GromConv_thm1}} 
}$$ 

\vspace{.2in}

\begin{center}\begin{tabular}{lll}
\ref{FHS_prp} Carleman Similarity Principle& \hspace{.5in}&
\ref{MonotLmm_prp} Monotonicity Lemma\\
\ref{PtBnd_prp} Mean Value Inequality&&
\ref{LowEner_crl} Lower Energy Bound\\
\ref{GlobStr_subs} Global structure of $J$-holomorphic maps&&
\ref{JHolomreg_prp} Regularity of $J$-holomorphic maps\\
\ref{CylEner_crl} Bounds on long cylinders&&
\ref{RemSing_prp} Removal of Singularity\\
\ref{EnerPres1_lmm}-\ref{EnerPres3_lmm} Bubbling&&
\ref{EnerPres_prp} Gromov's convergence
\end{tabular}\end{center}
\caption{Connections between the main statements leading to Theorem~\ref{GromConv_thm1}}
\label{connect_fig}
\end{figure}

\noindent
Another important implication of Proposition~\ref{PtBnd_prp} is that a continuous map
from a Riemann surface $(\Si,\fj)$ into an almost complex manifold~$(X,J)$
which is holomorphic outside of a discrete collection of points 
and has bounded energy is in fact holomorphic on all of~$\Si$;
see Proposition~\ref{JHolomreg_prp}.
This conclusion plays a central role in the proof of Lemma~\ref{EnerPres3_lmm}.
Theorem~\ref{GromConv_thm1} is deduced from Lemma~\ref{EnerPres3_lmm} and 
Proposition~\ref{EnerPres_prp} in Section~\ref{Conver_subs2}.\\

\noindent
Combined with Proposition~\ref{FHS_prp} and some of its corollaries, 
Proposition~\ref{PtBnd_prp}
 implies that every non-constant $J$-holomorphic map from a compact 
Riemann surface $(\Si,\fj)$ factors through 
a \sf{somewhere injective} $J$-holomorphic map from a compact 
Riemann surface $(\Si',\fj')$; see Proposition~\ref{SimplMap_prp}. 
The proof of this statement with $X$ compact
appears in Chapter~2 of~\cite{MS12},
but uses the Removal Singularities Theorem proved
in Chapter~4 of~\cite{MS12}.

\subsection{Notation and terminology}
\label{NotTerm_subs}

\noindent
Let $(\Si,\fj)$ be a Riemann surface, $V$ be a vector bundle over~$\Si$, and
$$\mu,\eta\in\Ga\big(\Si;T^*\Si\!\otimes_{\R}\!V\big)
\qquad\hbox{and}\qquad
g\in\Ga\big(\Si;T^*\Si^{\otimes2}\!\otimes_{\R}\!V\big).$$
For a local coordinate $z\!=\!s\!+\!\fI t$, define
\BE{muetadfn_e}\begin{split}
g(\mu\!\otimes_{\fj}\!\eta)&=
\big(g\big(\mu(\prt_s),\eta(\prt_s)\!\big)\!+\!g\big(\mu(\prt_t),\eta(\prt_t)\!\big)\!\big)
\nd s\!\w\!\nd t\,,\\
g(\mu\!\w_{\fj}\!\eta)&=\big(g\big(\mu(\prt_s),\eta(\prt_t)\!\big)\!-\!
g\big(\mu(\prt_t),\eta(\prt_s)\!\big)\!\big)
\nd s\!\w\!\nd t\,.
\end{split}\EE
By a direct computation, the 2-forms $g(\mu\!\otimes_{\fj}\!\eta)$ and $g(\mu\!\w_{\fj}\!\eta)$
are independent of the choice of local coordinate \hbox{$z\!=\!s\!+\!\fI t$}.
Thus, \eref{muetadfn_e} determines global 2-forms on~$\Si$
(which depend on the choice of~$\fj$).\\

\noindent
We denote by~$\fI$ the standard complex structure on~$\C$ and 
by~$J_{\C^n}$ the standard complex structures on~$\C^n$ and~$T\C^n$.
For an almost complex structure~$J$ and a 2-form~$\om$ on a manifold~$X$, we define
a 2-tensor and a 2-form on~$X$ by
\BE{omJdfn_e}\begin{split}
g_J(v,v')&=\frac12\big(\om(v,Jv')-\om(Jv,v')\big),\\
\om_J(v,v')&=\frac12\big(\om(Jv,Jv')-\om(v,v')\big)
\end{split}
\qquad\forall\,v,v'\!\in\!T_xX,~x\!\in\!X.\EE 
We note that 
\BE{omEner_e}g_J(v,v)+g_J(v',v')=2\om(v,v')+g_J(v\!+\!Jv',v\!+\!Jv')
+2\om_J(v,v')
\quad\forall\,v,v'\!\in\!T_xX,~x\!\in\!X.\EE
The 2-form $\om$ \sf{tames} $J$ if $g(v,v)\!>\!0$ for all $v\!\in\!TX$ nonzero;
in such a case,  $\om$ is nondegenerate and $g_J$ is a metric.
The almost complex structure~$J$ is \sf{$\om$-compatible} if $\om$ tames~$J$ and 
$\om_J\!=\!0$.\\

\noindent
Let $X$ be a manifold, $(\Si,\fj)$ be a Riemann surface, 
and $f\!:\Si\!\lra\!X$ be a smooth map.
We denote the pullbacks of a 2-tensor~$g$ and a 2-form~$\om$ on~$X$
to the vector bundle $f^*TX$ over~$\Si$ also by~$g$ and~$\om$.
If $g$ is a Riemannian metric on~$X$ and $U\!\subset\!\Si$
is an open subset, let 
\BE{Egfdfn_e}E_g(f)\equiv\frac12\int_{\Si}g(\nd f\!\otimes_{\fj}\!\nd f)\in[0,\i]
\qquad\hbox{and}\qquad
E_g(f;U)\equiv E_g\big(f|_U)\EE
be the \sf{energy} of~$f$ and of its restriction to~$U$.
By the first equation in~\eref{muetadfn_e}, 
\BE{Egfdfn2_e} E_g(f)=\frac12\int_{\Si}|\nd f|_{g_{\Si},g}^2\EE
is the square of the $L^2$-norm of $\nd f$ with respect to the metric~$g$ on~$X$
and a metric~$g_{\Si}$ compatible with~$\fj$.
In particular, the right-hand side of~\eref{Egfdfn2_e} depends on the metric~$g$ on~$X$
and on the complex structure~$\fj$ on~$\Si$, but {\it not} 
the metric~$g_{\Si}$ on~$\Si$ compatible with~$\fj$.\\

\noindent
Let $J$ be an almost complex structure on a manifold~$X$ and 
$(\Si,\fj)$ be a Riemann surface.
For a smooth map $f:\Si\!\lra\!X$, define
$$\dbar_Jf=\frac{1}{2}\big(\nd f\!+\!J\!\circ\!\nd f\!\circ\!\fj\big)
\in\Ga\big(\Si;(T^*\Si,\fj)^{0,1}\!\otimes_{\C}\!f^*(TX,J)\big)\,.$$
If $\om$ is a 2-form on~$X$ taming~$J$ and \hbox{$u\!:\Si\!\lra\!X$} is $J$-holomorphic,
then
\BE{Egexp_e}
E_{g_J}(f)=\int_{\Si}\big(f^*\om\!+\!2g_J(\dbar_Jf\!\otimes_{\fj}\!\dbar_Jf)\!+\!f^*\om_J\big)\EE
by~\eref{Egfdfn_e} and~\eref{omEner_e}.
If $J$ is $\om$-compatible, the last term above vanishes.
A smooth map $u\!:\Si\!\lra\!X$ is \sf{$J$-holomorphic} if $\dbar_Ju\!=\!0$.
For such a map, the last two terms in~\eref{Egexp_e} vanish.\\

\noindent
For each $R\!\in\!\R^+$, denote by $B_R\!\subset\!\C$ the open ball of 
radius~$R$ around the origin and let 
$$B_R^*=B_R\!-\!\{0\}.$$
If in addition $(X,g)$ is a Riemannian manifold and $x\!\in\!X$,
let $B_{\de}^g(x)\!\subset\!X$ be the ball of radius~$\de$ around~$x$ in~$X$
with respect to the metric~$g$.\\

\noindent
Let $(X,J)$ be an almost complex manifold and $(\Si,\fj)$ be a Riemann surface.
A smooth map \hbox{$u\!:\Si\!\lra\!X$} is called
\begin{enumerate}[label=$\bu$,leftmargin=*]

\item \sf{somewhere injective} if there exists $z\!\in\!\Si$
such that $u^{-1}(u(z))\!=\!\{z\}$ and $\nd_zu\!\neq\!0$,

\item \sf{multiply covered} if $u\!=\!u'\!\circ\!h$ 
for some smooth connected orientable surface~$\Si'$,
branched cover \hbox{$h\!:\!\Si\!\lra\!\Si'$} of degree different from $\pm1$,
and a smooth map $u'\!:\Si'\!\lra\!X$,

\item \sf{simple} if it is not multiply covered. 

\end{enumerate}
By Proposition~\ref{SimplMap_prp}, every $J$-holomorphic map from a compact Riemann
surface is simple if and only if it is somewhere injective
(the {\it if} implication is trivial).

\section{Local Properties}
\label{LocProp_sec}

\noindent
We begin by studying local properties of $J$-holomorphic maps~$u$
from Riemann surfaces~$(\Si,\fj)$ into almost complex manifolds~$(X,J)$
that resemble standard properties of holomorphic maps.
None of the statements in Section~\ref{LocProp_sec} depending on~$X$ being compact;
very few depend on~$\Si$ being compact.

\subsection{Carleman Similarity Principle}
\label{LocJMaps_subs}

\noindent
\sf{Carleman Similarity Principle}, i.e.~Proposition~\ref{FHS_prp} below, is a local description of 
solutions of a non-linear differential equation which generalizes the equation $\dbar_Ju\!=\!0$.
It states that such solutions look similar to holomorphic maps and implies
that they exhibit many local properties one would expect of holomorphic maps.

\begin{prp}[Carleman Similarity Principle, {{\cite[Theorem~2.2]{FHS}}}]\label{FHS_prp}
Suppose $n\!\in\!\Z^+$, \hbox{$p,\ep\!\in\!\R^+$} with $p\!>\!2$, 
$J\!\in\!L^p_1(B_{\ep};\End_{\R}\C^n)$, $C\!\in\!L^p(B_{\ep};\End_{\R}\C^n)$,
and $u\!\in\!L^p_1(B_{\ep};\C^n)$ are such~that 
\BE{FHScond_e} u(0)=0, \qquad
J(z)^2=-\Id_{\C^n}, ~~
u_s(z)+J(z)u_t(z)+C(z)u(z)=0 ~~\forall~z\!=\!s\!+\!\fI t\!\in\!B_{\ep}\,.\EE
Then, there exist $\de\!\in\!(0,\ep)$, 
$\Phi\!\in\!L_1^p(B_{\de};\GL_{2n}\R)$, and 
a $J_{\C^n}$-holomorphic map $\si\!:B_{\de}\!\lra\!\C^n$ such~that 
\BE{FHSconc_e} \si(0)=0, \qquad
J(z)\Phi(z)=\Phi(z)J_{\C^n},~~u(z)=\Phi(z)\si(z)
~~\forall~z\!\in\!B_{\de}\,.\EE
\end{prp}

\vspace{.2in}

\noindent
By the Sobolev Embedding Theorem \cite[Corollary~4.3]{anal}, 
the assumption $p\!>\!2$ implies that $u$ is a continuous function.
In particular, all equations in~\eref{FHScond_e} and
in \eref{FHSconc_e} make sense.
This assumption also implies that the left-hand sides of the third equation
in~\eref{FHScond_e} and of the second equation in~\eref{FHSconc_e} 
and the right-hand side of the third equations in~\eref{FHSconc_e} 
lie in~$L_1^p$. 

\begin{eg}\label{FHS_eg}
Let $\fc\!:\C\!\lra\!\C$ denote the usual conjugation.
Define
\begin{gather*}
\wh{J}(z_1,z_2)=\left(\!\!\begin{array}{cc} \fI& 0\\ -2\fI s_1\fc& \fI\end{array}\!\!\right)
=\left(\begin{array}{cc} 1& 0\\ s_1\fc& 1\end{array}\!\!\right)
\!\!J_{\C^2}\!\!
\left(\!\!\begin{array}{cc} 1& 0\\ s_1\fc& 1\end{array}\!\!\right)^{\!\!-1}\!:\C^2\lra\C^2
\quad\forall\,z_i\!=\!s_i\!+\!\fI t_i, \\
u\!:\C\lra\C^2, \qquad u(s\!+\!\fI t)=\big(z,s^2).
\end{gather*}
Thus, $\wh{J}$ is an almost complex structure on~$\C^2$ and
$u$ is a $\wh{J}$-holomorphic map, i.e.~it satisfies the last condition
in~\eref{FHScond_e} with $J(z)\!=\!\wh{J}(u(z))$ and $C(z)\!=\!0$.
The functions
$$\si\!:\C\lra\C^2, \quad\si(z)=(z,0), \qquad
\Phi\!:\C\lra\GL_4\R, \quad 
\Phi(s\!+\!\fI t)=
\left(\!\!\begin{array}{cc} 1& 0\\ s\fc\!+\!\frac{\fI st}{z}& 1\end{array}\!\!\right),$$
satisfy~\eref{FHSconc_e}.
\end{eg}

\begin{crl}\label{FHS_crl1}
Let $n$, $p$, $\ep$, $J$, $C$, and $u$ be as in Proposition~\ref{FHS_prp}.
If in addition $J_0\!=\!J_{\C^n}$ and $u$ does not vanish to infinite order~0,
then there exist $\ell\!\in\!\Z^+$ and $\al\!\in\!\C^n\!-\!0$ such~that
$$\lim_{z\lra0}\frac{u(z)-\al z^{\ell}}{z^{\ell}}=0\,.$$  
\end{crl}

\begin{proof}
This follows from~\eref{FHSconc_e}
and from the existence of such~$\ell$ and~$\al$ for~$\si$.
\end{proof}

\begin{crl}\label{FHS_crl2}
Suppose $(X,J)$ is an almost complex manifold, $(\Si,\fj)$ is a Riemann surface,
and $u\!:\Si\!\lra\!X$ is a  $J$-holomorphic map.
If $u$ is not constant on every connected component of~$\Si$, then the~subset 
$$u^{-1}\big(\{u(z)\!:\,z\!\in\!\Si,~\nd_zu\!=\!0\}\big)\subset\Si$$
is discrete. If in addition $x\!\in\!X$, the subset $u^{-1}(x)\!\subset\!\Si$ 
is also discrete.
\end{crl}

\begin{proof}
The first and third equations in~\eref{FHSconc_e} immediately imply the second claim
(but not the first, since~$\Phi$ may not be in~$C^1$).
The first claim follows from Corollary~\ref{FHS_crl1} and Taylor's formula for~$u$
(as well as from Corollary~\ref{FHS_crl}).
\end{proof}

\noindent
Before establishing the full statement of Proposition~\ref{FHS_prp}, 
we consider a special case.

\begin{lmm}\label{FHS_lmm}
Suppose $n\!\in\!\Z^+$ and $p,\ep\!\in\!\R^+$ are as in Proposition~\ref{FHS_prp},
$A\!\in\!L^p(B_{\ep};\End_{\C}\C^n)$, and $u\!\in\!L^p_1(B_{\ep};\C^n)$
are such~that 
\BE{FHScond4_e} u(0)=0, \qquad 
u_s+J_{\C^n}u_t(z)+A(z)u(z)=0 ~~\forall~z\!=\!s\!+\!\fI t\!\in\!B_{\ep}\,.\EE
Then, there exist $\de\!\in\!(0,\ep)$,   $\Phi\!\in\!L_1^p(B_{\de};\GL_n\C)$, 
a $J_{\C^n}$-holomorphic map $\si\!:B_{\de}\!\lra\!\C^n$ such~that 
\BE{FHSconc4_e} \si(0)=0, \qquad
\Phi(0)=\Id_{\C^n}, \qquad u(z)=\Phi(z)\si(z)
~~\forall~z\!\in\!B_{\de}\,.\EE
\end{lmm}

\begin{proof}
For each $\de\!\in\![0,\ep]$, we define 
\begin{gather*}
A_{\de}\in L^p(S^2;\End_{\C}\C^n)
\qquad\hbox{by}\quad
A_{\de}(z)=\begin{cases}A(z),&\hbox{if}~z\!\in\!B_{\de};\\
0,&\hbox{otherwise};\end{cases}\\
D_{\de}: L_1^p(S^2;\End_{\C}\C^n)\lra L^p(S^2;(T^*S^2)^{0,1}\!\otimes_{\C}\!\End_{\C}\C^n)
\qquad\hbox{by}\quad
D_{\de}\Th=\big(\Th_s\!+\!J_{\C^n}\Th_t\!+\!A_{\de}\Th\big)\nd\bar{z}\,.
\end{gather*}
Since the cokernel of $D_0\!=\!2\dbar$ is isomorphic $H^1(S^2;\C)\!\otimes_{\C}\!\End_{\C}\C^n$,
$D_0$ is surjective and the homomorphism 
$$\ti{D}_0\!:L_1^p(S^2;\End_{\C}\C^n)\lra L^p(S^2;(T^*S^2)^{0,1}\!\otimes_{\C}\!\End_{\C}\C^n)
\oplus \End_{\C}\C^n,
\qquad \Th\lra \big(D_0\Th,\Th(0)\big),$$
is an isomorphism. 
Since 
$$\big\|D_{\de}\Th-D_0\Th\big\|_{L^p} \le \|A_{\de}\|_{L^p}\|\Th\|_{C^0}
\le C\|A_{\de}\|_{L^p}\|\Th\|_{L^p_1}
\qquad\forall~\Th\in L_1^p(S^2;\End_{\C}\C^n)$$
and $\|A_{\de}\|_{L^p}\!\lra\!0$ as $\de\!\lra\!0$, the homomorphism 
$$\ti{D}_{\de}\!:L_1^p(S^2;\End_{\C}\C^n)\lra 
L^p(S^2;(T^*S^2)^{0,1}\!\otimes_{\C}\!\End_{\C}\C^n)\oplus \End_{\C}\C^n,
\qquad \Th\lra \big(D_{\de}\Th,\Th(0)\big),$$
is also an isomorphism for $\de\!>\!0$ sufficient small.
Let $\Th_{\de}\!=\!D_{\de}^{-1}(0,\Id_{\C^n})$.
Since $D_{\de}$ is an isomorphism,
$$\big\|\Th_{\de}\!-\!\Id_{\C^n}\big\|_{C^0}
\le C \big\|\Th_{\de}\!-\!\Id_{\C^n}\big\|_{L^p_1}
\le C'\big\|D_{\de}(\Th_{\de}\!-\!\Id_{\C^n})\big\|_{L^p}
= C'\big\|A_{\de}\big\|_{L^p}\,.$$
Since $\|A_{\de}\|_{L^p}\!\lra\!0$ as $\de\!\lra\!0$, 
$\Th_{\de}\!\in\!L_1^p(B_{\de};\GL_n\C)$.
By \eref{FHScond4_e} and $D_{\de}\Th_{\de}\!=\!0$, 
the function \hbox{$\si\!\equiv\!\Th_{\de}^{-1}u$} satisfies
$$\si(0)=0,\qquad \si_s\!+\!J_{\C^n}\si_t\!=\!0\quad\forall~z\in B_{\de},$$ 
i.e.~$\si$ is $J_{\C^n}$-holomorphic, as required.
\end{proof}

\begin{proof}[{\bf\emph{Proof of Proposition~\ref{FHS_prp}}}]
(1) Since $B_{\ep}$ is contractible, 
the complex vector bundles $u^*(T\C^n,J_{\C^n})$ and $u^*(T\C^n,J)$ over $B_{\ep}$
are isomorphic.
Thus, there exists
$$\Psi\in L_1^p(B_{\ep};\GL_{2n}\R) \qquad\hbox{s.t.}\quad
J(z)\Psi(z)=\Psi(z)J_{\C^n}~~~\forall~z\!\in\!B_{\ep}\,.$$
Let $v\!=\!\Psi^{-1}u$. By the assumptions on $u$, $v\!\in\!L^p_1(B_{\ep};\C^n)$ and
\begin{gather}\label{FHScond_e2} 
v(0)=0, \qquad
v_s(z)+J_{\C^n}v_t(z)+\ti{C}(z)v(z)=0 \quad\forall~z\!=\!s\!+\!\fI t\!\in\!B_{\ep},\\
\hbox{where}\quad
\ti{C}=\Psi^{-1}\cdot\big(\Psi_s+J\Psi_t+C\Psi)
\in L^p(B_{\ep};\End_{\R}\C^n)\,.\notag
\end{gather}
Thus, we have reduced the problem to the case $J\!=\!J_{\C^n}$.\\

\noindent
(2) Let $\ti{C}^{\pm}\!=\!\frac12(\ti{C}\mp J_{\C^n}\ti{C}J_{\C^n})$ be the $\C$-linear
and $\C$-antilinear parts of~$\ti{C}$, i.e.~$\ti{C}^{\pm}J_{\C^n}\!=\pm J_{\C^n}\ti{C}^{\pm}$.
With $\lr{\cdot,\cdot}$ denoting the Hermitian inner-product on~$\C^n$
which is $\C$-antilinear in the second input, define
$$D\in L^{\i}(B_{\ep};\End_{\R}\C^n), \quad
D(z)w=\begin{cases}|v(z)|^{-2}\lr{v(z),w}v(z),&\hbox{if}~v(z)\!\neq\!0;\\
0,&\hbox{otherwise};
\end{cases} \quad
A=\ti{C}^+ +\ti{C}^-D\,.$$
Since $DJ_{\C^n}\!=\!-J_{\C^n}D$ and $Dv\!=\!v$, $A\in L^p(B_{\ep};\End_{\C}\C^n)$ 
and $Av\!=\!\ti{C}v$.
Thus, by~\eref{FHScond_e2}, 
$$v_s+J_{\C^n}v_t+Av=0\,.$$
The claim now follows from Lemma~\ref{FHS_lmm}.
\end{proof}

\begin{crl}\label{FHS_crl}
Suppose $n\!\in\!\Z^+$, $\ep\!\in\!\R^+$, 
$J$ is a smooth almost complex structure on~$\C^n$ with $J_0\!=\!J_{\C^n}$, 
and $u\!:B_{\ep}\!\lra\!\C^n$ is a $J$-holomorphic map with $u(0)\!=\!0$.
Then, there exist $\de\!\in\!(0,\ep)$, $C\!\in\!\R^+$,
$\Phi\!\in\!C^0(B_{\de};\GL_{2n}\R)$, and 
a $J_{\C^n}$-holomorphic map $\si\!:B_{\de}\!\lra\!\C^n$ such~that 
$\Phi$ is smooth on~$B_{\de}^*$, 
$$ \si(0)=0, ~~ \Phi(0)=\Id_{\C^n}, ~~
J(u(z))\Phi(z)=\Phi(z)J_{\C^n},~~u(z)=\Phi(z)\si(z), ~~
\big|\nd_z\Phi\big|\le C~\forall\,z\!\in\!B_{\de}^*\,.$$
\end{crl}

\begin{proof} We can assume that $u$ is not identically~0 on some neighborhood of $0\!\in\!B_{\ep}$.
Similarly to~(1) in the proof of Proposition~\ref{FHS_prp}, there exists
$$\Psi\in C^{\i}(\C^n;\GL_{2n}\R) \qquad\hbox{s.t.}\qquad
\Psi(0)=\Id_{\C^n}, \quad
J(x)\Psi(x)=\Psi(x)J_{\C^n}~~~\forall~x\!\in\!\C^n\,.$$
Let $v(z)\!=\!\Psi(u(z))^{-1}u(z)$.
By Corollary~\ref{FHS_crl1}, we can choose complex linear coordinates on~$\C^n$ so~that 
$$v(z)=\big(f(z),g(z)\big)h(z) \in \C\!\oplus\!\C^{n-1} \qquad\forall~z\!\in\!B_{\ep'}$$
for some $\ep'\!\in\!(0,\ep)$, holomorphic function $h$ on~$B_{\ep'}$ with $h(0)\!=\!0$, 
and continuous functions $f$ and $g$ on~$B_{\ep'}$ with $f(0)\!=\!1$ and $g(0)\!=\!0$.
By Lemma~\ref{C1bnd_lmm} below applied with~$f$ above and with each component of~$g$ separately, 
there exists $\de\!\in\!(0,\ep')$ so that 
the function 
$$\Phi\!:B_{\de}\lra\GL_{2n}\R, \qquad 
\Phi(z)=\Psi\big(u(z)\big)
\left(\!\!\begin{array}{cc} f(z)& 0\\ g(z)& 1\end{array}\!\!\right),$$
is continuous on $B_{\de}$ and smooth on $B_{\de}\!-\!0$ with
$|\nd_z\Phi|$  uniformly bounded  on $B_{\de}\!-\!0$.
Taking \hbox{$\si(z)\!=\!(h(z),0)$}, we conclude the proof.
\end{proof}

\begin{lmm}\label{C1bnd_lmm}
Suppose $\ep\!\in\!\R^+$, and 
$f,h\!:B_{\ep}\!\lra\!\C$ are continuous functions such~that
$h$ is holomorphic, $h(z)\!\!\neq\!0$ for some $z\!\in\!B_{\ep}$, and
the function
\BE{C1bnd_e}B_{\ep}\lra\C, \qquad z\lra f(z)h(z),\EE
is smooth. 
Then there exist $\de\!\in\!(0,\ep)$ and $C\!\in\!\R^+$ such that 
$f$ is differentiable on $B_{\ep}\!-\!0$ and
\BE{C1bnd_e2}\big|\nd_zf\big|\le C \qquad\forall~z\!\in\!B_{\de}\!-\!0\,.\EE
\end{lmm}

\begin{proof} After a holomorphic change of coordinate on $B_{2\de}\!\subset\!B_{\ep}$,
we can assume that $h(z)\!=\!z^{\ell}$ for some $\ell\!\in\!\Z^{\ge0}$.
Define
$$g\!: B_{2\de}\lra\C, \qquad g(z)=f(z)z^{\ell}-f(0)z^{\ell}\,.$$
By Taylor's Theorem and the smoothness of the function~\eref{C1bnd_e}, 
there exists $C\!>\!0$ such that the smooth function~$g$ satisfies
$$\big|g(z)\big|\le C|z|^{\ell+1} \qquad \forall~z\!\in\!B_{\de}\,.$$
Dividing~$g$ by $z^{\ell}$, we thus obtain~\eref{C1bnd_e2}.
\end{proof}

\begin{rmk}\label{FHS_rmk}
Corollary~\ref{FHS_crl} refines the conclusion of Proposition~\ref{FHS_prp}
for $J$-holomorphic maps.
In contrast to the output $(\Phi,\si)$ of Proposition~\ref{FHS_prp},
the output of Corollary~\ref{FHS_crl} does not depend continuously 
on the input~$u$ with respect to the $L^p_1$-norms.
This makes Corollary~\ref{FHS_crl} less suitable for applications in settings 
involving families of $J$-holomorphic maps.
\end{rmk}

\subsection{Local structure of $J$-holomorphic maps}
\label{LocJMaps_subs2}

\noindent
We now obtain three corollaries from Proposition~\ref{FHS_prp}.
They underpin important geometric statements established later
in these notes, such as Propositions~\ref{MonotLmm_prp} and~\ref{SimplMap_prp}
and Lemma~\ref{EnerPres3_lmm}.

\begin{crl}[Unique Continuation]\label{FHS_crl2a}
Suppose $(X,J)$ is an almost complex manifold, $(\Si,\fj)$ 
is a connected Riemann surface, and
$$u,u'\!:(\Si,\fj)\lra (X,J)$$
are $J$-holomorphic maps.
If $u_0$ and $u_0'$ agree to infinite order at $z_0\!\in\!\Si$, then $u'\!=\!u'$.
\end{crl}

\begin{proof}
Since the subset of the points of~$\Si$ at which $u$ and $u'$ agree is closed to infinite order, 
it is enough to show that $u\!=\!u'$ on some neighborhood of~$z_0$.
By the continuity of~$u$, we can assume that $X\!=\!\C^n$, 
$\Si\!=\!B_1$, $z_0\!=\!0$, and $u(0),u'(0)\!=\!0$.
Let
$$w\!=\!u'\!-\!u:B_{\ep}\lra \C^n\,.$$
Since $J$ is $C^1$,
\BE{FHS2a_e3}J(x\!+\!y)=J(x)+\!\!\int_0^1\frac{\nd J(x\!+\!ty)}{\nd t}\nd t
=J(x)+\sum_{i=1}^ny_i\!\!
\int_0^1\frac{\prt J}{\prt y_i}\bigg|_{x+ty}\nd t\,.\EE
Since $u$ and $u'$ are $J$-holomorphic, \eref{FHS2a_e3} implies that 
\begin{gather*}
\prt_s w + J\big(u(z)\big)\prt_t w+C(z)w(z)=0, \qquad\hbox{where}\quad
C\in L^p\big(B_1;\End_{\R}\C^n\big),\\
C(z)y=\sum_{i=1}^ny_i
\bigg(\int_0^1\frac{\prt J}{\prt y_i}\bigg|_{v(z)+tw(z)}\nd t\bigg)\prt_tw|_z\,.
\end{gather*}
By Proposition~\ref{FHS_prp}, 
there thus exist $\de\!\in\!(0,1)$, \hbox{$\Phi\!\in\!L^p_1(B_{\de};\GL_{2n}\R)$},
and holomorphic map \hbox{$\wt{w}\!:B_{\de}\!\lra\!\C^n$} such~that 
$$w(z)=\Phi(z)\wt{w}(z) \qquad\forall~z\in B_{\de}\,.$$
Since $w$ vanishes to infinite order at~$0$, it follows that
$\wt{w}(z)\!=\!0$ for all $z\!\in\!B_{\de}$
(otherwise, $w$ would satisfy the conclusion of Corollary~\ref{FHS_crl1})
and thus $w(z)\!=\!0$ for all $z\!\in\!B_{\de}$.
\end{proof}

\begin{crl}\label{FHS_crl2b}
Suppose $(X,J)$ is an almost complex manifold, 
$$u,u'\!:(\Si,\fj),(\Si',\fj')\lra (X,J)$$
are $J$-holomorphic maps,
$z_0\!\in\!\Si$ is such that $\nd_{z_0}u\!\neq\!0$,
and $z_0'\!\in\!\Si'$ is such that $u'(z_0')\!=\!u(z_0)$.
If there exist sequences $z_i\!\in\!\Si\!-\!z_0$ and $z_i'\!\in\!\Si'\!-\!z_0'$
such~that 
$$\lim_{i\lra\i}z_i=z_0\,, \qquad \lim_{i\lra\i}z_i'=z_0'\,, 
\quad\hbox{and}\quad u(z_i)=u'(z_i)~~\forall~i\!\in\!\Z^+\,,$$
then there exists a holomorphic map $\si\!:U'\!\lra\!\Si$ from a neighborhood
of~$z_0'$ in $\Si'$ such that $\si(z_0')\!=\!z_0$ and $u'|_{U'}\!=\!u\!\circ\!\si$.
\end{crl}

\begin{proof}
It can be assumed that $(\Si,\fj,z_0),(\Si',\fj',z_0')\!=\!(B_1,\fj_0,0)$,
where $B_1\!\subset\!\C$ is the unit ball with the standard complex structure.
Since $\nd_{z_0}u\!\neq\!0$ and $u$ is $J$-holomorphic, 
$u$ is an embedding near $0\!\in\!B_1$ and so is a slice in a coordinate system.
Thus, we can assume that 
$$u,u'\!\equiv\!(v,w)\!: (B_1,0)\lra (\C\!\times\!\C^{n-1},0),
\qquad u(z)=(z,0)\in\C\!\times\!\C^{n-1}\,,$$
and $u,u'$ are $J$-holomorphic with respect to some almost complex structure
$$J(x,y)=\left(\begin{array}{cc}J_{11}(x,y)& J_{12}(x,y)\\ 
J_{21}(x,y)& J_{22}(x,y)\end{array}\right)\!: \C\!\times\!\C^{n-1}\lra\C\!\times\!\C^{n-1}\,,
\qquad (x,y)\in\C\!\times\!\C^{n-1}\,.$$
Since $J$ is $C^1$,
\BE{Jexpand_e}J_{ij}(x,y)=J_{ij}(x,0)+\!\!\int_0^1\frac{\nd J_{ij}(x,ty)}{\nd t}\nd t
=J_{ij}(x,0)+\sum_{i=1}^{n-1}y_i\!\!
\int_0^1\frac{\prt J_{ij}}{\prt y_i}\bigg|_{(x,ty)}\nd t\,.\EE
Since $u$ is $J$-holomorphic,
\BE{J0cond_e} J_{21}(x,0)=0,\quad J_{22}(x,0)^2=-\Id \qquad\forall~x\in B_1\subset\C.\EE
Since $u'$ is $J$-holomorphic,
$$ \prt_s w + J_{22}\big(v(z),w(z)\big)\prt_t w
+J_{21}\big(v(z),w(z)\big)\prt_t v=0.$$
Combining this with~\eref{Jexpand_e} and the first equation in~\eref{J0cond_e},
we find that 
\begin{gather*}
\prt_s w + J_{22}\big(v(z),0\big)\prt_t w+C(z)w(z)=0, \qquad\hbox{where}\quad
C\in L^p\big(B_1;\End_{\R}\C^{n-1}\big),\\
C(z)y=\sum_{i=1}^{n-1}y_i\Bigg(
\bigg(\int_0^1\frac{\prt J_{22}}{\prt y_i}\bigg|_{(v(z),tw(z))}\nd t\bigg)\prt_tw|_z
+\bigg(\int_0^1\frac{\prt J_{21}}{\prt y_i}\bigg|_{(v(z),tw(z))}\nd t\bigg)\prt_tv|_z\Bigg)\,.
\end{gather*}
By Proposition~\ref{FHS_prp} and the second identity in~\eref{J0cond_e}, 
there thus exist $\de\!\in\!(0,1)$, \hbox{$\Phi\!\in\!L^p_1(B_{\de};\GL_{2n-2}\R)$},
and holomorphic map $\wt{w}\!:B_{\de}\!\lra\!\C^{n-1}$ such~that 
$$w(z)=\Phi(z)\wt{w}(z) \qquad\forall~z\in B_{\de}\,.$$
Since \hbox{$u'(z_i')\!=\!u(z_i)$}, $\wt{w}(z_i')\!=\!0$ for all $i\!\in\!\Z^+$.
Since $z_i'\!\lra\!0$ and $z_i'\!\neq\!0$, it follows that $w\!=\!0$.
This implies the claim with $U'\!=\!B_{\de}$ and $\si\!=\!v$.
\end{proof}

\begin{crl}\label{FHS_crl3}
Let $(X,J)$ be an almost complex manifold with a Riemannian metric~$g$
and $x\!\in\!X$ be such that $g$ is compatible with~$J$ at~$x$.
If $u\!:\Si\!\lra\!X$ is a $J$-holomorphic map 
from a compact Riemann surface with boundary, then
$$\lim_{\de\lra0}\frac{E_g(u;u^{-1}(B_{\de}^g(x)))}{\pi\de^2}=\ord_xu\,.$$
\end{crl}

\begin{proof} 
By the continuity of~$u$, we can assume that $X\!=\!\C^n$, 
$J$ agrees with the standard complex structure~$J_{\C^n}$ at the origin,
$g$ agrees with the standard metric~$g_{\C^n}$ at the origin, 
$\Si\!=\!\ov{B_R}$ for some $R\!\in\!\R^+$,
and $u(0)\!=\!0$.
In particular, there exists $C\!\ge\!1$ such that 
\BE{FHScrl3_e0} \big|J_x-J_{\C^n}\big|\le C|x|,  \quad \big|g_x-g_{\C^n}\big|\le C|x|
\qquad\forall~x\!\in\!\C^n~\hbox{s.t.}~|x|\le1,\EE
where $|\cdot|$ denotes the usual norm of~$x$ 
(i.e.~the distance to the origin with respect to~$g_{\C^n}$).\\

\noindent
Let $\ell\!\equiv\!\ord_0u$ and $\al\!\in\!\C^{n-1}\!-\!0$ be
as in Corollary~\ref{FHS_crl1}, where $0\!\in\!B_R$ is 
the origin in the domain of~$u$.
Thus, there exist $\ep\!\in\!(0,1)$ and $C\!\in\!\R^+$ such~that 
\BE{FHScrl3_e1} u(z)=\al\cdot\big(z^{\ell}\!+\!f(z)\big),~~
\big|f(z)\big|\le C|z|^{\ell+1} \qquad\forall~z\!\in\!B_{\ep}\,.\EE
Let $z\!=\!s\!+\!\fI t$ as before.
By \eref{FHScrl3_e1}, there exists $C\!\in\!\R^+$ such~that
\BE{FHScrl3_e2} u_s(z)=\al\cdot\big(\ell z^{\ell-1}\!+\!f_s(z)\big),~~
u_s(z)=\al\cdot\big(\ell \fI z^{\ell-1}\!+\!f_t(z)\big),~~
\big|f_s(z)\big|,\big|f_t(z)\big|\le C|z|^{\ell} \quad\forall~z\!\in\!B_{\ep}\,.\EE
We can also assume that the three constants~$C$ in~\eref{FHScrl3_e0}, \eref{FHScrl3_e1}, 
and~\eref{FHScrl3_e2} are the same, $C\!\ge\!1$, 
$$C_{\al}\ep \equiv (C\!+\!C|\al|\!+\!C^2|\al|\big)\ep\le 1\,,$$
and $|u(z)|\!\le\!1$ for all $z\!\in\!B_{\ep}$.
By~\eref{FHScrl3_e0}-\eref{FHScrl3_e2}, 
\BE{FHScrl3_e5}\begin{split}
\bigg|\frac{|u(z)|_g}{|\al||z|^{\ell}}-1\bigg|,
\bigg|\frac{|u_s(z)|_g}{|\al|\ell|z|^{\ell-1}}-1\bigg|,
\bigg|\frac{|u_t(z)|_g}{|\al|\ell|z|^{\ell-1}}-1\bigg|
&\le C|z|+C|\al||z|^{\ell}+C^2|\al||z|^{\ell+1}\\
&\le C_{\al}|z|\quad\forall~z\in B_{\ep},
\end{split}\EE
where $|\cdot|_g$ denotes the distance to the origin in~$\C^n$ with respect to
the metric~$g$ and the corresponding norm on~$T\C^n$.\\

\noindent
Given $r\!\in\!(0,1)$, let $\de_r\!\in\!(0,\ep)$ be such that 
\BE{FHScrl3_e7}
C_{\al}\bigg(\frac{2\de_r}{(1\!-\!r)|\al|}\bigg)^{1/\ell}\le r\,.\EE
For any $\de\!\in\![0,\de_r]$, \eref{FHScrl3_e5} and~\eref{FHScrl3_e7} give
\begin{alignat*}{2}  
|z|&\le \bigg(\frac{\de}{(1\!+\!r)|\al|}\bigg)^{1/\ell} &\qquad&\Lra\qquad
u(z)\in B_{\de}^g(0)\,,\\
u(z)&\in B_{\de}^g(0) &\qquad&\Lra\qquad 
|z|\le \bigg(\frac{\de}{(1\!-\!r)|\al|}\bigg)^{1/\ell}\,,\\
|z|&\le \bigg(\frac{\de}{(1\!-\!r)|\al|}\bigg)^{1/\ell} &\qquad&\Lra\qquad
1\!-\!r \le \frac{|u_s(z)|_g}{|\al|\ell|z|^{\ell-1}},\frac{|u_t(z)|_g}{|\al|\ell|z|^{\ell-1}}
\le 1\!+\!r.
\end{alignat*}  
Combining these, we obtain
\begin{equation*}\begin{split}
\int_{|z|\le\left(\frac{\de}{(1+r)|\al|}\right)^{\!\!\frac1\ell}}
(1\!-\!r)^2\big(|\al|\ell|z|^{\ell-1}\big)^2    
&\le \frac12\int_{u^{-1}(B_{\de}^g(0))}\!\!\big(|u_s|_g^2\!+\!|u_t|_g^2\big)\\
&\le \int_{|z|\le\left(\frac{\de}{(1-r)|\al|}\right)^{\!\!\frac1\ell}}
(1\!+\!r)^2\big(|\al|\ell|z|^{\ell-1}\big)^2  \,.
\end{split}\end{equation*}
Evaluating the outer integrals, we find that 
$$\bigg(\frac{1\!-\!r}{1\!+\!r}\bigg)^{\!2}\ell\pi\de^2
\le E_g\big(u;u^{-1}(B_{\de}^g(0))\big)
\le \bigg(\frac{1\!+\!r}{1\!-\!r}\bigg)^{\!2}\ell\pi\de^2\,.$$
These inequalities hold for all $r\!\in\!(0,1)$ and $\de\!\in\!(0,\de_r)$;
the claim is obtained by sending $r\!\lra\!0$.
\end{proof}

\subsection{The Monotonicity Lemma}
\label{MonotLmm_subs}

\noindent
Proposition~\ref{MonotLmm_prp} below is a key step in the continuity part 
of the proof of the \sf{Removal of Singularity} Proposition~\ref{RemSing_prp}.
The precise nature of the lower energy bound on the right hand-side of~\eref{MonThm_e} 
does not matter, as long as it is positive for~$\de\!>\!0$.

\begin{prp}[Monotonicity Lemma]\label{MonotLmm_prp}
If $(X,J)$ is an almost complex manifold and $g$ is a Riemannian metric on~$X$ compatible with~$J$, 
there exists a continuous function $C_{g,J}\!:X\!\lra\!\R^+$
with the following property.
If $(\Si,\fj)$ is a compact Riemann surface with boundary,
\hbox{$u\!:\Si\!\lra\!X$} is a $J$-holomorphic map, 
$x\!\in\!X$, and $\de\!\in\!\R^+$ are such that 
$u(\prt\Si)\!\cap\!B_{\de}^g(x)\!=\!\eset$, then 
\BE{MonThm_e} E_g(u)\ge \big(\ord_xu\big)\frac{\pi\de^2}{{(1\!+\!C_{g,J}(x)\de)^4}}\,.\EE
If $\om(\cdot,\cdot)\equiv\!g(J\cdot,\cdot)$ is a symplectic form on~$X$, 
then the above fraction can be replaced by the product $\pi\de^2\ne^{-C_{g,J}(x)\de^2}$.
\end{prp}

\noindent
According to this proposition, ``completely getting out" of the ball $B_{\de}(x)$ 
via a $J$-holomorphic map requires an energy bounded below by a little less than~$\pi\de^2$.
Thus, the $L^2_1$-norm of a $J$-holomorphic map~$u$ exerts some control over 
the $C^0$-norm of~$u$. 
If $p\!>\!2$, the $L^p_1$-norm of any smooth map~$f$ from a two-dimensional manifold controls 
the $C^0$-norm of~$f$. 
However, this is not the case of the $L^2_1$-norm, as 
illustrated by the example of \cite[Lemma~10.4.1]{MS12}:
the function
$$f_{\ep}\!:\R^2\lra[0,1], \qquad
f_{\ep}(z)=\begin{cases}1,&\hbox{if}~|z|\le\ep;\\
\frac{\ln|z|}{\ln\ep},&\hbox{if}~\ep\le|z|\le1;\\
0,&\hbox{if}~|z|\ge1;
\end{cases}$$
with any $\ep\!\in\!(0,1)$ is continuous and satisfies
$$\int_{\R^2}|\nd f_{\ep}|_g^2=-\frac{2\pi}{\ln\ep}\,.$$ 
It is arbitrarily close in the $L^2_1$-norm to a smooth function~$\ti{f}_{\ep}$.
Thus, it is possible to ``completely get out" of $B_{\de}^g(x)$   
using a smooth function with arbitrarily small energy
($\ti{f}_{\de}$ does this for the ball $B_1(1)$ in~$\R$).\\

\noindent
By~\eref{Egexp_e}, the holomorphic maps are the local minima of the functional
$$C^{\i}(\Si;X)\lra\R, \qquad f\lra E_g(f)-\int_{\Si}f^*\om_J\,,$$
for every compact Riemann surface $(\Si,\fj)$ without boundary.
This fact underlines  Lemma~\ref{CritEner_lmm},
the key ingredient in the proof of the Monotonicity Lemma.
Lemma~\ref{CritEner_lmm} implies that the ratio of $E_g(u;u^{-1}(B_{\de}^g(x)))$
and the fraction on the right-hand side~\eref{MonThm_e} 
is a non-decreasing function of~$\de$, as long as \hbox{$u(\prt\Si)\!\cap\!B_{\de}^g(x)\!=\!\eset$}.
By Corollary~\ref{FHS_crl3}, this ratio approaches $\ord_xu$ as~$\de$ approaches~0.
These two statements imply Proposition~\ref{MonotLmm_prp}.\\

\noindent
We first make some general Riemannian geometry observations.
Let $(X,g)$ be a Riemannian manifold. Denote by $\exp\!:\cW_g\!\lra\!X$,
the exponential map from a neighborhood of~$X$ in~$TX$ with respect to
the Levi-Civita connection~$\na$ of~$g$.
For each $v\!\in\!TX$, we denote~by
$$\ga_v\!:[0,1]\lra X, \qquad \ga_v(\tau)=\exp_x(\tau v),$$  
the geodesic with $\ga_v'(0)\!=\!v$.
Let
$$r_g\!:X\lra\R^+ \qquad\hbox{and}\qquad d_g\!:X\!\times\!X\lra\R^{\ge0}$$
be the injectivity radius of~$\exp$ and the distance function.
For each $x\!\in\!X$, define
$$\ze_x\in\Ga\big(B_{r_g(x)}^g(x);TX\big) \quad\hbox{by}\quad
\exp_y\!\big(\ze_x(y)\!\big)=x, ~g\big(\!\ze_x(y),\ze_x(y)\!\big)<r_g(x)^2
~~\forall\,y\!\in\!B_{r_g(x)}^g(x).$$
 
\begin{lmm}\label{distder_lmm}
Let $(X,g)$ be a Riemannian manifold and $x\!\in\!X$.
If  $\al\!:(-\ep,\ep)\!\lra\!X$ is a smooth curve such that 
$\al(0)\!\in\!B_{r_g(x)}^g(x)$, 
then 
$$\frac12\frac{\nd}{\nd\tau}d_g\big(x,\al(\tau)\big)^2\bigg|_{\tau=0}
=-g\big(\al'(0),\ze_x(\al(0))\big)\,.$$
\end{lmm}

\begin{proof}
If $\be(\tau)\!=\!\exp_x^{-1}\al(\tau)$, then
$$\frac12\frac{\nd}{\nd\tau}d_g\big(x,\al(\tau)\big)^2\bigg|_{\tau=0}
=\frac12\frac{\nd}{\nd\tau}|\be(\tau)|^2\bigg|_{\tau=0}
=g\big(\be'(0),\be(0)\big)\,.$$
By Gauss's Lemma,
$$g\big(\be'(0),\be(0)\big)
=g\big(\{\nd_{\be(0)}\exp_x\}(\be'(0)),\{\nd_{\be(0)}\exp_x\}(\be(0))\!\big)
=g\big(\al'(0),-\ze_x(\al(0))\big)\,.$$
This establishes the claim.
\end{proof}

\begin{lmm}\label{JacDer_lmm}
If $(X,g)$ is a Riemannian manifold, there exists 
a continuous function \hbox{$C_g\!:X\!\lra\!\R^+$}
with the following property.
If $x\!\in\!X$, $v\!\in\!T_xX$ with $|v|_g\!<\!\frac12r_g(x)$, and 
$\tau\!\lra\!J(\tau)$ is a Jacobi vector field along the geodesic~$\ga_v$
with $J(0)\!=\!0$, then 
$$\big|J'(1)-J(1)\big|_g\le C_g(x)|v|_g^2\big|J(1)\big|_g\,.$$
\end{lmm}

\begin{proof}
Let $R_g$ be the Riemann curvature tensor of~$g$ and
$f(\tau)\!=\!|\tau J'(\tau)\!-\!J(\tau)|_g$.
Then, $f(0)\!=\!0$ and
\begin{equation*}\begin{split}
f(\tau)f'(\tau)=\frac12\frac{\nd}{\nd\tau}f(\tau)^2
=g\big(\tau J''(\tau),\tau J'(\tau)\!-\!J(\tau)\!\big)
&=\tau g\big(R(\ga'(\tau),J(\tau))\ga'(\tau),\tau J'(\tau)\!-\!J(\tau)\big)\\
&\le C_g(x)|v|_g^2|J(\tau)|_g\tau f(\tau).
\end{split}\end{equation*}
If $C_g$ is sufficiently large, then $|J(\tau)|_g\!\le\!C_g(x)|J(1)|_g$.
Thus, 
$$f(\tau)f'(\tau)\le C_g(x)|v|_g^2|J_v(\tau)|_g\tau f(\tau)
\le C_g(x)^2|v|_g^2|J(1)|_g\tau f(\tau), \quad
f'(\tau)\le C_g(x)^2|v|_g^2|J(1)|_g\tau.$$
The claim follows from the last inequality.
\end{proof}

\begin{crl}\label{JacDer_crl}
If $(X,g)$ is a Riemannian manifold, there exists a continuous function \hbox{$C_g\!:X\!\lra\!\R^+$}
with the following property.
If $x\!\in\!X$, then
$$\big|\na_w\ze_x|_y+w\big|_g\le C_g(x)d_g(x,y)^2|w|_g \qquad 
\forall~w\!\in\!T_yX,~y\!\in\!B_{r_g(x)/2}^g(x).$$
\end{crl}

\begin{proof}
Let $\tau\!\lra\!u(s,\tau)$ be a family of geodesics such that 
$$u(s,0)=x,\qquad u(0,1)=y, \qquad \frac\nd{\nd s}u(s,1)\bigg|_{s=0}=w.$$
Since $\tau\!\lra\!u(s,\tau)$ is a geodesic,
\begin{gather*}
\frac{\nd}{\nd\tau}u(s,\tau)\bigg|_{\tau=1}
=\big\{\!\nd_{u_{\tau}(s,0)}\exp_x\!\big\}\big(u_{\tau}(s,0)\big)
=-\ze_x\big(u(s,1)\big),\\
\frac{\nD}{\nd\tau}\frac{\nd u(s,\tau)}{\nd s}\bigg|_{(s,\tau)=(0,1)}
=\frac{\nD}{\nd s}\frac{\nd u(s,\tau)}{\nd\tau}\bigg|_{(s,\tau)=(0,1)}
=-\na_w\ze_x|_y\,.
\end{gather*}
Furthermore, $J(\tau)\!\equiv\!\frac{\nd}{\nd s}u(s,\tau)\big|_{s=0}$ is a Jacobi vector field 
along the geodesic $\tau\!\lra\!u(0,\tau)$ with 
$$J(0)=\!0,\quad J(1)=w, \quad 
J'(1)=\frac{\nD}{\nd\tau}\frac{\nd u(s,\tau)}{\nd s}\bigg|_{(s,\tau)=(0,1)}=-
\na_w\ze_x|_y\,.$$
Thus, the claim follows from Lemma~\ref{JacDer_lmm}.
\end{proof}

\begin{lmm}\label{CritEner_lmm}
Suppose $(X,\om)$ is a symplectic manifold, 
$J$ is an almost complex structure on~$X$ tamed by~$\om$,
and $\na$ is the Levi-Civita connection of the metric~$g_J$.
If $(\Si,\fj)$ is a compact Riemann surface with boundary and 
$u\!:\Si\!\lra\!X$ is a $J$-holomorphic map, then
$$\int_{\Si}g_J\big(\nd u\!\otimes_{\fj}\!\na\xi\big)=
\int_{\Si}\big(u^*\{\na_{\xi}\om_J\}\!+\!
\om_J(\nd u\!\w_{\fj}\!\na\xi)\big)
\qquad\forall~\xi\!\in\!\Ga(\Si;u^*TX)~\hbox{s.t.}~\xi|_{\prt\Si}\!=\!0.$$
\end{lmm}

\begin{proof} For $\tau\!\in\!\R$ sufficiently close to~0, define
$$u_{\tau}\!:\Si\lra X, \qquad u_{\tau}(z)\!=\!\exp_{u(z)}(\tau\xi(z)).$$
Since $\xi|_{\prt\Si}\!=\!0$, $u_{\tau}|_{\prt\Si}\!=\!u|_{\prt\Si}$. 
Denote by $\wh\Si$ the closed oriented surface obtained by gluing two copies of~$\Si$
along the common boundary and reversing the orientation on the second copy.
Let
$$\wh{u}_{\tau}\!: \wh\Si\lra X$$
be the map restricting to~$u_{\tau}$ on the first copy of~$\Si$
and to~$u$ on the second.\\

\noindent
By~\eref{Egexp_e}, 
$$E(\tau)\equiv E_{g_J}(u_{\tau})-\int_{\Si}u_{\tau}^*\om_J-E_{g_J}(u)
=\int_{\wh\Si}\wh{u}_{\tau}^*\om+
2\int_{\Si}g_J\big(\dbar u_{\tau}\!\otimes_{\fj}\!\dbar u_{\tau}\big)
\ge0\quad\forall\tau.$$
Since $\om$ is closed and $\wh{u}_*$ represents the zero class in $H_2(X;\Z)$,
the first integral on the right-hand side above vanishes.
Thus, the function \hbox{$\tau\!\lra\!E(\tau)$}
is minimized at $\tau\!=\!0$ (when it equals 0) and~so  
\BE{Ediff_e2}\begin{split}
0=E'(0)
&=\frac{\nd}{\nd\tau}\bigg(E_{g_J}(u_{\tau})
-\int_{\Si}u_{\tau}^*\om_J\bigg)\bigg|_{\tau=0}\\
&=\frac{\nd}{\nd\tau}\bigg(
\frac12\int_{\Si}\!\!g_J(\nd u_{\tau}\!\otimes_{\fj}\!\nd u_{\tau})
-\int_{\Si}u_{\tau}^*\om_J\bigg)\bigg|_{\tau=0}\,;
\end{split}\EE
the last equality above uses the definition of $E(u_{\tau})$ in~\eref{Egfdfn_e}.\\

\noindent
Let $z\!=\!s\!+\!\fI t$ be a local coordinate on~$(\Si,\fj)$.
Since $\na$ is torsion-free, 
$$\frac{D}{\nd\tau}(u_{\tau})_s\Big|_{\tau=0}
\equiv\frac{D}{\nd\tau}\frac{\nd u_{\tau}}{\nd s}\bigg|_{\tau=0}
=\frac{D}{\nd s}\frac{\nd u_{\tau}}{\nd\tau}\bigg|_{\tau=0}
=\frac{D}{\nd s}\xi\equiv\na_s\xi, \qquad
\frac{D}{\nd\tau}(u_{\tau})_t\Big|_{\tau=0}=\na_t\xi\,.$$
Since $\na$ is also $g$-compatible, 
\begin{gather*}\begin{split}
\frac12\frac{\nd}{\nd\tau}g_J(\nd u_{\tau}\!\otimes_{\fj}\!\nd u_{\tau})
\bigg|_{\tau=0}
&=\bigg(g_J\bigg(u_s,\frac{D}{\nd\tau}(u_{\tau})_s\Big|_{\tau=0}\bigg)
+g_J\bigg(u_t,\frac{D}{\nd\tau}(u_{\tau})_t\Big|_{\tau=0}\bigg)\!\bigg)
\nd s\!\w\!\nd t\\
&=g_J(u_s,\na_s\xi)+g_J(u_t,\na_t\xi)
=g_J\big(\nd u\!\otimes_{\fj}\!\na\xi\big)\,,
\end{split}\\
\begin{split}
\frac{\nd}{\nd\tau}u_{\tau}^*\om_J\bigg|_{\tau=0}
&=\bigg(\!\!\big\{\na_{\xi}\om_J\big\}(u_s,u_t)+
\om_J\bigg(\frac{D}{\nd\tau}(u_{\tau})_s\Big|_{\tau=0},u_t\!\!\bigg)
+\om_J\bigg(\!\!u_s,\frac{D}{\nd\tau}(u_{\tau})_t\Big|_{\tau=0}\bigg)\!\bigg)
\nd s\!\w\!\nd t\\
&=u^*\{\na_{\xi}\om_J\}\!+\!\om_J\big(\nd u\!\w_{\fj}\!\na\xi\big)\,.
\end{split}\end{gather*}
Combining this with~\eref{Ediff_e2}, we obtain the claim.
\end{proof}

\begin{proof}[{\bf\emph{Proof of Proposition~\ref{MonotLmm_prp}}}]
Let $\de_g\!:X\!\lra\!\R^+$ be a continuous function such that for every $x\!\in\!X$
there exists a symplectic form~$\om_x$ on $B_{2\de_g(x)}^g(x)$
so that~$J$ is tamed by~$\om_x$ on $B_{2\de_g(x)}^g(x)$  
and compatible with~$\om_x$ at~$x$.
We assume that $2\de_g(x)\!\le\!r_g(x)$ for every $x\!\in\!X$.
It is sufficient to establish the proposition for each $x\!\in\!X$
and each $\de\!\le\!\de_g(x)$ under the assumption that 
the metric~$g$ is determined by~$J$ and~$\om_x$ on $B_{\de_g(x)}^g(x)$.\\

\noindent
Choose a $C^{\i}$-function $\eta\!:\R\!\lra\![0,1]$ such that 
\BE{etadfn_e}\eta(\tau)=\begin{cases}1,&\hbox{if}~\tau\le\frac12;\\
0,&\hbox{if}~\tau\ge1; \end{cases} \qquad
\eta'(\tau)\le0.\EE
For a compact Riemann surface with boundary $(\Si,\fj)$,
a smooth map \hbox{$u\!:\Si\!\lra\!X$},
$x\!\in\!X$, and $\de\!\in\!\R^+$, define
\begin{gather*}
\eta_{u,x,\de}\in C^{\i}(\Si;\R), \qquad
\eta_{u,x,\de}(z)=\eta\bigg(\frac{d_g(x,u(z))}{\de}\bigg),\\
E_{u,x,\eta}(\de)=\frac12\int_{\Si}
\eta_{u,x,\de}(z)g\big(\nd u\!\otimes_{\fj}\!\nd u\big)\,,\quad
E_{u,x}(\de)=E_g\big(u;u^{-1}(B_{\de}^g(x))\big).
\end{gather*}

\vspace{.2in}

\noindent
We show in the remainder of this proof that there exists a continuous function 
\hbox{$C_{g,J}\!:X\!\lra\!\R^+$} such~that 
\BE{Aetabd_e}
-\de E_{u,x,\eta}'(\de)+2E_{u,x,\eta}(\de)\le 
2C_{g,J}(x)\de E_{u,x,\eta}(\de)+C_{g,J}(x)\de^2 E_{u,x,\eta}'(\de)\EE
for every compact Riemann surface with boundary $(\Si,\fj)$,
$J$-holomorphic map \hbox{$u\!:\Si\!\lra\!X$}, and \hbox{$\de\!\in\!(0,\de_g(x))$}
such that $u(\prt\Si)\!\cap\!B_{\de}^g(x)\!=\!\eset$.
This inequality is equivalent~to
$$\bigg(E_{u,x,\eta}(\de)\bigg/\frac{\de^2}{(1\!+\!C_{g,J}(x)\de)^4}\bigg)'\ge 0.$$
By Lebesgue's Dominated Convergence Theorem, $E_{u,x,\eta}(\de)$ approaches~$E_{u,x}(\de)$
from below as~$\eta$ approaches the characteristic function~$\chi_{(-\i,1)}$ of~$(-\i,1)$.
Thus, the function
$$\de\lra E_{u,x}(\de)\bigg/\frac{\de^2}{(1\!+\!C_{g,J}(x)\de)^4}$$
is non-decreasing  as long as $u(\prt\Si)\!\cap\!B_{\de}^g(x)\!=\!\eset$.
By Corollary~\ref{FHS_crl3},
$$\lim_{\de\lra0}\bigg(E_{u,x}(\de)\!\!\bigg/\!\!\frac{\de^2}{(1\!+\!C_{g,J}(x)\de)^4}\bigg)
=\lim_{\de\lra0}\frac{E_{u,x}(\de)}{\de^2}=\big(\ord_xu\big)\pi.$$
This implies the first claim.\\

\noindent
Fix $x\!\in\!X$. We note~that
\BE{MonLem_e6}
E_{u,x,\eta}'(\de)=-\frac12\int_{\Si}\eta'\!\bigg(\!\frac{d_g(x,u(z))}{\de}\!\bigg)
\frac{d_g(x,u(z))}{\de^2}
g\big(\nd u\!\otimes_{\fj}\!\nd u\big).\EE
For a compact Riemann surface with boundary $(\Si,\fj)$,
a smooth map \hbox{$u\!:\Si\!\lra\!X$}, and \hbox{$\de\!\in\!(0,\de_g(x))$}, 
let
$$\xi_{u,x,\de}\in \Ga(\Si;u^*TX), \qquad
\xi_{u,x,\de}(z)=-\eta_{u,x,\de}(z)\ze_x\big(u(z)\big);$$
the vanishing assumption in~\eref{etadfn_e} implies that $\xi_{u,x,\de}$
is well-defined.
If $u(\prt\Si)\!\cap\!B_{\de}^g(x)\!=\!\eset$, then $\xi_{u,x,\de}|_{\prt\Si}\!=\!0$.
By Lemma~\ref{distder_lmm},
\BE{MonLem_e0}\begin{split}
\na\xi_{u,x,\de}|_z
=\eta'\bigg(\frac{d_g(x,u(z))}{\de}\bigg)\frac{1}{\de\,d_g(x,u(z))}
g\big(\nd_zu,\ze_x(u(z))\!\big)\ze_x(u(z))
-\eta_{u,x,\de}(z)\na\ze_x\!\circ\!\nd_zu.
\end{split}\EE
Along with Corollary~\ref{JacDer_crl}, \eref{MonLem_e6}, 
and the last assumption in~\eref{etadfn_e}, this implies~that 
\BE{MonLem_e8}
\int_{\Si}\!d_g(x,u(z))\big|g(\nd u\!\otimes_{\fj}\!\na\xi_{u,x,\de})\big|
\le 2\de^2 E_{u,x,\eta}'(\de)+2\big(1\!+\!C_g(x)\de^2\big)\de E_{u,x,\eta}(\de).\EE

\vspace{.1in}

\noindent
By the $\om_x$-compatibility assumption on~$J$ at~$x$, there exists 
a continuous function \hbox{$C\!:X\!\lra\!\R^+$} such~that 
$$\int_{\Si}\big|(\om_x)_J(\nd u\!\w_{\fj}\!\na\xi_{u,x,\de})\big|
\le C(x)\!\!\int_{\Si}d_g\big(x,u(z)\!\big)\big|g(\nd u\!\otimes_{\fj}\!\na\xi_{u,x,\de})\big|$$
for all $u$ and $\de$ as above.
Along with this, Lemma~\ref{CritEner_lmm} implies that 
there exists a continuous function \hbox{$C\!:X\!\lra\!\R^+$} such~that 
$$\bigg|\int_{\Si}g\big(\nd u\!\otimes_{\fj}\!\na\xi_{u,x,\de}\big)\bigg|
\le C(x)\!\!\int_{\Si}\!\big(g\big(\nd u\!\otimes_{\fj}\!\nd u\big)|\xi_{u,x,\de}|
\!+\!d_g(x,u(z))\big|g(\nd u\!\otimes_{\fj}\!\na\xi_{u,x,\de})\big|\big)$$
for every compact Riemann surface with boundary $(\Si,\fj)$,
$J$-holomorphic map \hbox{$u\!:\Si\!\lra\!X$}, and $\de\!\in\!(0,\de_g(x))$ 
such that $u(\prt\Si)\!\cap\!B_{\de}^g(x)\!=\!\eset$.
Combining this with~\eref{MonLem_e8}, we conclude that 
there exists a continuous function \hbox{$C\!:X\!\lra\!\R^+$} such~that
\BE{MonLem_e0b}
\bigg|\int_{\Si}g\big(\nd u\!\otimes_{\fj}\!\na\xi_{u,x,\de}\big)\bigg|
\le C(x)\big(\de E_{u,x,\eta}(\de)\!+\!\de^2E_{u,x,\eta}'(\de)\big)\EE
for all $u$ and $\de$ as above.\\

\noindent
Suppose $(\Si,\fj)$ is a compact Riemann surface with boundary,
$u\!:\Si\!\lra\!X$ is a smooth map, and $\de\!\in\!(0,\de_g(x))$. 
Let $z\!=\!s\!+\!\fI t$ be a coordinate on~$(\Si,\fj)$.
By~\eref{MonLem_e0},
\BE{MonLem_e1}\begin{split}
g\big(u_s,\na_s\xi_{u,x,\de}\big)
=\eta'\bigg(\!\frac{d_g(x,u(z))}{\de}\!\bigg)\frac{1}{\de\,d_g(x,u(z))}g\big(u_s,\ze_x(u(z))\big)^2&\\
+\eta_{u,x,\de}(z)g\big(u_s,\na_s(-\ze_x)|_z\big)&.
\end{split}\EE
By Corollary~\ref{JacDer_crl}, 
\BE{MonLem_e2}
|u_s|^2\le g\big(u_s,\na_s(-\ze_x)|_z\big) +C_g(x)d_g(x,u(z))^2|u_s|^2
\quad\forall\,z\!\in\!u^{-1}\big(B_{\de_g(x)}^g(x)\big)\,.\EE
If $u$ is $J$-holomorphic, then $|u_s|\!=\!|u_t|$, $\lr{u_s,u_t}\!=\!0$, and
\BE{MonLem_e3}
\frac12\big(|u_s|^2\!+\!|u_t|^2\big)d_g(x,u(z))^2=|u_s|^2|\ze_x(u(z))|^2
\ge g\big(u_s,\ze_x(u(z))\big)^2+g\big(u_t,\ze_x(u(z))\big)^2\,.\EE
Since $\eta'\!\le\!0$,  \eref{MonLem_e1}-\eref{MonLem_e3} give
\BE{MonLem_e4}\begin{split}
&\frac12\eta'\bigg(\!\frac{d_g(x,u(z))}{\de}\!\bigg)\frac{d_g(x,u(z))}{\de}\big(|u_s|^2\!+\!|u_t|^2\big)
+\eta_{u,x,\de}(z)\big(|u_s|^2\!+\!|u_t|^2\big)\\
&\qquad\le g\big(u_s,\na_s\xi_{u,x,\de}\big)+g\big(u_t,\na_t\xi_{u,x,\de}\big)+
C_g(x)\eta_{u,x,\de}(z)d_g(x,u(z))^2\big(|u_s|^2\!+\!|u_t|^2\big).
\end{split}\EE
Along with~\eref{MonLem_e6}, this implies that 
\BE{MonLem_e9}-\de E_{u,x,\eta}'(\de)+2E_{u,x,\eta}(\de)\le 
\int_{\Si}g\big(\nd u\!\otimes_{\fj}\!\na\xi_{u,x,\de}\big)
+2C_g(x)\de^2E_{u,x,\eta}(\de)\EE
for every compact Riemann surface with boundary $(\Si,\fj)$,
$J$-holomorphic map \hbox{$u\!:\Si\!\lra\!X$}, and $\de\!\in\!(0,\de_g(x))$.
Combining this inequality with~\eref{MonLem_e0b}, we obtain~\eref{Aetabd_e}.\\

\noindent
Suppose $\om\!\equiv\!g(J\cdot,\cdot)$ is a symplectic form on~$X$.
By Lemma~\ref{CritEner_lmm}, the left-hand side of~\eref{MonLem_e0b} then vanishes.
From~\eref{MonLem_e9}, we thus obtain
$$-\de E_{u,x,\eta}'(\de)+2E_{u,x,\eta}(\de)\le 2C_{g,J}(x)\de^2 E_{u,x,\eta}(\de)\,.$$
The reasoning below~\eref{Aetabd_e} now yields the second claim.
\end{proof}

\section{Mean Value Inequality and applications}
\label{EBnd_sec}

\noindent
We now move to properties of $J$-holomorphic maps~$u$
from Riemann surfaces~$(\Si,\fj)$ into almost complex manifolds~$(X,J)$
that are of a more global nature.
They generally concern the distribution of the energy of such a map over its domain
and are consequences of the \sf{Mean Value Inequality} for $J$-holomorphic maps.
These fairly technical properties lead to geometric conclusions
such as Propositions~\ref{GlobStr_subs} and~\ref{RemSing_prp}.

\subsection{Statement and proof}
\label{MeanValIn_subs}

\noindent
According to Cauchy's Integral Formula, a holomorphic map $u\!:B_R\!\lra\!\C^n$ satisfies
$$u'(0)=\frac{1}{2\pi\fI}\oint_{|z|=r}\frac{u(z)}{z^2}\nd z
\qquad\forall\,r\!\in\!(0,R).$$
This immediately implies that a bounded holomorphic function defined on all of~$\C$ is constant.
The Mean Value Inequality of Proposition~\ref{PtBnd_prp}
bounds the norms of the differentials of 
$J$-holomorphic maps of sufficiently small energy away from the boundary of the domain ``uniformly"
by their $L^2$-norms.
In general, one would not expect the value of a function to be bounded 
by its integral.
The Mean Value Inequality implies that a $J$-holomorphic map
which is defined on all of~$\C$ and has sufficiently small energy is in fact constant;
see Corollary~\ref{LowEner_crl}.

\begin{prp}[Mean Value Inequality]\label{PtBnd_prp}
If $(X,J)$ is an almost complex manifold and $g$ is a Riemannian metric on~$X$
compatible with~$J$,  there exists a continuous function 
\hbox{$\hb_{J,g}\!:X\!\times\!\R\!\lra\!\R^+$} with the following property. 
If \hbox{$u\!:B_R\!\lra\!X$} is a $J$-holomorphic map such that 
$$u(B_R)\subset B_r^g(x) \qquad\hbox{and}\qquad
E_g(u)<\hb_{J,g}(x,r)$$
for some $x\!\in\!X$ and $r\!\in\!\R$, then
\BE{PtBnd_e} \big|\nd_0u\big|_g^2<\frac{16}{\pi R^2} E_g(u)\,.\EE
\end{prp}

\begin{proof}
Let $\phi(z)\!=\!\frac12|\nd_zu|_g^2$. 
By Lemma~\ref{LapEng_lmm} below, $\De\phi\!\ge\!-A_{J,g}\phi^2$ with 
\hbox{$A_{J,g}\!:X\!\times\!\R\!\lra\!\R^+$} determined by $(X,J,g)$. 
The claim with $\hb_{J,g}\!=\!\pi/8A_{J,g}$ thus follows from Proposition~\ref{Lapl_prp}.
\end{proof}

\begin{crl}[Lower Energy Bound]\label{LowEner_crl}
If $(X,J)$ is a compact almost complex manifold and $g$ is a Riemannian metric on~$X$, 
then there exists $\hbar_{J,g}\!\in\!\R^+$ such that 
$E_g(u)\!\ge\!\hbar_{J,g}$ for every non-constant $J$-holomorphic map
$u\!:S^2\!\lra\!X$.
\end{crl}

\begin{proof}
By the compactness of $X$, we can assume that $g$ is compatible with~$J$.
Let $\hbar_{J,g}\!>\!0$ be the minimal value of the function~$\hbar_{J,g}$
in the statement of Proposition~\ref{PtBnd_prp} on the compact space 
$X\!\times\![0,\diam_g(X)]$.
If $u\!:S^2\!\lra\!X$ is $J$-holomorphic map with $E_g(u)\!<\!\hbar_{J,g}$,
$$\big|\nd_zu\big|_g^2<\frac{16}{\pi R^2} E_g\big(u;B_R(z)\big)
\le \frac{16}{\pi R^2}E_g(u)\qquad\forall~z\!\in\!\C,~R\!\in\!\R^+$$
by Proposition~\ref{PtBnd_prp}, since $B_R(z)\!\subset\!\C$ as Riemann surfaces.
Thus, $\nd_zu\!=\!0$ for all $z\!\in\!\C$, and so $u$ is constant.
\end{proof}

\noindent
If $\phi\!:U\!\lra\!\R$ is a $C^2$-function on an open subset of~$\R^2$, let
$$\De\phi=\frac{\prt^2\phi}{\prt s^2}+\frac{\prt^2\phi}{\prt t^2}
\equiv \phi_{ss}+\phi_{tt}$$
denote the Laplacian of~$\phi$.

\begin{exer}\label{LaplPol_exer}
Show that in the polar coordinates $(r,\th)$ on $\R^2$,
\BE{LaplPol_e}\De\phi=\phi_{rr}+r^{-1}\phi_r+r^{-2}\phi_{\th\th}\,.\EE
\end{exer}

\begin{lmm}\label{Lapl_lmm1}
If $\phi\!:\ov{B_R}\!\lra\!\R$ is $C^2$, then
\BE{Lapl_e1}
2\pi R\,\phi(0)=-R\int_{(r,\th)\in B_R}\!\!(\ln R\!-\!\ln r)\De\phi
+\int_{\prt B_R}\phi\,.\EE 
\end{lmm}

\begin{proof} By Stokes' Theorem applied to $\phi\nd\th$ on $\ov{B_R}\!-\!B_{\de}$,
\begin{equation*}\begin{split}
\oint_{\prt B_R}\!\!\!\phi\,\nd\th-\oint_{\prt B_{\de}}\!\!\!\phi\,\nd\th
&=\int_{\ov{B_R-B_{\de}}}\!\phi_r\,\nd r\!\w\!\nd\th
=\int_0^{2\pi}\!\!\!\!\int_{\de}^R(r\phi_r)r^{-1}\,\nd r\nd\th\\
&=\int_0^{2\pi}\!\!(\ln R\!-\!\ln\de)\de\,\phi_r(\de,\th)\nd\th
+\int_0^{2\pi}\!\!\!\!\int_{\de}^R\!\!(\ln R\!-\!\ln r)(\phi_{rr}+r^{-1}\phi_r)r\,\nd r\nd\th\,;
\end{split}\end{equation*}
the last equality above is obtained by applying integration by parts to the functions
$\ln r\!-\!\ln R$ and~$r\phi_r$.
Sending $\de\!\lra\!0$ and using~\eref{LaplPol_e}, we obtain
$$\frac1R\int_{\prt B_R}\phi - 2\pi\,\phi(0)=0+ 
\int_{(r,\th)\in B_R}\!\!(\ln R\!-\!\ln r)\De\phi\,,$$
which is equivalent to~\eref{Lapl_e1}.
\end{proof}

\begin{crl}\label{Lapl_crl}
If $\phi\!:\ov{B_R}\!\lra\!\R$ is $C^2$ and $\De\phi\!\ge\!-C$ for some $C\!\in\!\R^+$, 
then
\BE{Lapl_e2}
\phi(0)\le \frac18 CR^2+\frac1{\pi R^2}\int_{B_R}\phi\,.\EE 
\end{crl}

\begin{proof}
By \eref{Lapl_e1}, 
$$2\pi r\,\phi(0)\le Cr\int_0^{2\pi}\!\!\!\int_0^r(\ln r\!-\!\ln\rho)
\rho\,\nd\rho\,\nd\th+\int_{\prt B_r}\!\!\phi
=Cr\cdot 2\pi\cdot\frac{r^2}{4}+\int_{\prt B_r}\!\!\phi
\qquad\forall\,r\!\in\!(0,R).$$
Integrating the above in $r\!\in\!(0,R)$, we obtain
$$2\pi\phi(0)\cdot\frac{R^2}{2} \le 2\pi C\cdot\frac{R^4}{16}
+\int_{B_R}\phi.$$
This inequality is equivalent to~\eref{Lapl_e2}.
\end{proof}

\begin{prp}\label{Lapl_prp}
If $\phi\!:B_R\!\lra\!\R^{\ge0}$ is $C^2$ and 
there exists $A\!\in\!\R^+$ such that 
$\De\phi\!\ge\!-A\phi^2$ and $\displaystyle\int_{B_R}\phi<\frac{\pi}{8A}$, then
\BE{Lapl_e3}
\phi(0)\le \frac8{\pi R^2}\int_{B_R}\phi\,.\EE 
\end{prp}

\begin{proof}
Replacing $A$ by $\ti{A}\!=\!R^2A$ and  $\phi$ by 
$$\ti\phi\!: B_1\lra\R, \qquad\ti\phi(z)=\phi(Rz),$$
we can assume that $R\!=\!1$, as well as that $\phi$ is defined on~$\ov{B_1}$.\\

\noindent
(1) Define
$$f\!:[0,1)\lra\R^{\ge0} \qquad\hbox{by}\qquad 
f(r)=(1\!-\!r)^2\max_{\ov{B_r}}\phi\,.$$
In particular, $f(0)\!=\!\phi(0)$ and $f(1)\!=\!0$.
Choose $r^*\!\in\![0,1)$ and $z^*\!\in\!B_{r^*}$ such~that 
$$f(r^*)=\sup f \qquad\hbox{and}\qquad 
\phi(z^*)=\sup_{B_{r^*}}\phi\equiv c^*\,.$$
Let $\de\!=\!\frac12(1\!-\!r^*)\!>\!0$; see Figure~\ref{Lapl_fig}. Thus,
$$\sup_{B_{\de}(z^*)}\phi \le  \sup_{B_{r^*+\de}}\phi
=\frac{f(r^*\!+\!\de)}{(1\!-\!(r^*\!+\!\de))^2}
\le \frac{f(r^*)}{\frac14(1\!-\!r^*)^2}
=4\phi(z^*)=4c^*\,.$$
In particular, $\De\phi\ge -A\phi^2\ge -16Ac^{*2}$ on $B_{\de}(z^*)$.\\

\noindent
(2) Using Corollary~\ref{Lapl_crl}, we thus find that 
\BE{cbnd_e} c^*=\phi(z^*)\le\frac18\cdot 16Ac^{*2}\cdot\rho^2
+\frac{1}{\pi\rho^2}\int_{B_{\rho}(z^*)}\!\!\!\phi
\le 2Ac^{*2}\rho^2+\frac{1}{\pi\rho^2}\int_{B_1}\!\!\!\phi
\qquad\forall~\rho\!\in\![0,\de]\,.\EE
If $2Ac^*\de^2\le\frac12$, the $\rho\!=\!\de$ case of the above inequality gives
$$\frac12 c^*\le \frac{1}{\pi\de^2}\int_{B_1}\phi \,,
\qquad
\phi(0)=f(0)\le f(r^*)=4c^*\cdot \de^2
\le \frac8\pi \int_{B_1}\phi\,,$$
as claimed.
If $2Ac^*\de^2\ge\frac12$, $\rho\!\equiv\!(4Ac^*)^{-\frac12}\le\de$ and 
\eref{cbnd_e} gives
$$c^*\le 2Ac^{*2}\cdot\frac{1}{4Ac^*}+\frac{4Ac^*}{\pi}\int_{B_1}\phi\,.$$
Thus, $\displaystyle\frac{\pi}{8A}\le\int_{B_1}\!\!\phi$,
contrary to the assumption.
\end{proof}

\begin{figure}
\begin{pspicture}(38,-.3)(9,2.7)
\psset{unit=.3cm}
\pscircle[linewidth=.04](56,4){3}\pscircle[linewidth=.06](56,4){5}
\pscircle*(61,4){.2}\rput(61.7,4){1}
\pscircle*(59,4){.2}\rput(58.2,4.2){$r^*$}
\psline[linewidth=.03,linestyle=dashed](61,4)(59,4)
\rput(60,4.7){\sm{$2\de$}}
\pscircle*(56,1){.2}\rput(55.7,2.7){\sm{$B_{\de}(z^*)$}}\pscircle(56,1){1}
\end{pspicture}
\caption{Setup for the proof of Proposition~\ref{Lapl_prp}}
\label{Lapl_fig}
\end{figure}
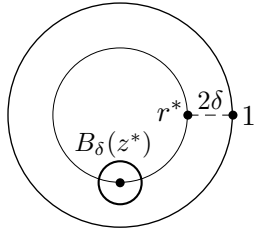

\begin{lmm}\label{LapEng_lmm}
If $(X,J)$ is an almost complex manifold and $g$ is a Riemannian metric on~$X$
compatible with~$J$,  there exists a continuous function 
\hbox{$A_{J,g}\!:X\!\times\!\R\!\lra\!\R^+$} with the following property. 
If $\Om\!\subset\!\C$ is an open subset, $u\!:\Om\!\lra\!X$ is a $J$-holomorphic map,
and $u(\Om)\!\subset\!B_r^g(x)$ for some $x\!\in\!X$ and $r\!\in\!\R$, then
the function $\phi(z)\!\equiv\!\frac12|\nd_zu|^2_g$ satisfies
$\De\phi\ge-A_{J,g}(x,r)\phi^2$.
\end{lmm}

\begin{proof}
Let $z\!=\!s\!+\!\fI t$ be the standard coordinate on~$\C$. 
Denote by~$u_s$ and~$u_t$ the $s$ and $t$-partials of~$u$, respectively.
Since $u$ is $J$-holomorphic, i.e.~$u_s\!=\!-Ju_t$, and $g$ is $J$-compatible,
i.e.~$g(J\cdot,J\cdot)\!=\!g(\cdot,\cdot)$, $|u_s|_g^2\!=\!|u_t|_g^2$.
Since the Levi-Civita connection~$\na$ of~$g$ is $g$-compatible and torsion-free,
\BE{LapEng_e1}
\frac12 \frac{\nd^2}{\nd^2 t}|u_s|_g^2= |\na_tu_s|_g^2+\blr{\na_t\na_tu_s,u_t}_{\!g}
=|\na_tu_s|_g^2+\blr{\na_t\na_su_t,u_s}_{\!g}\,.\EE
Similarly,
\BE{LapEng_e2}
\frac12 \frac{\nd^2}{\nd^2 s}|u_t|_g^2
=\big|\na_su_t\big|_g^2+\blr{\na_s\na_t u_s,u_t}_{\!g}.\EE
Since $u_s\!=\!-Ju_t$, 
\BE{LapEng_e3}\begin{split}
\lr{\na_s\na_t u_s,u_t}_{\!g}&=-\blr{\na_s\na_t (Ju_t),u_t}_{\!g}\\
&=-\blr{J\na_s\na_t u_t,u_t}_{\!g} -\blr{(\na_sJ)\na_t u_t,u_t}_{\!g}
-\blr{\na_s((\na_tJ)u_t),u_t}_{\!g}\\
&=-\blr{\na_s\na_t u_t,u_s}_{\!g} -\blr{(\na_sJ)\na_t u_t,u_t}_{\!g}
-\blr{\na_s((\na_tJ)u_t),u_t}_{\!g}\,.
\end{split}\EE
Putting \eref{LapEng_e1}-\eref{LapEng_e3}, we find that 
\BE{LapEng_e4}\begin{split}
\frac12\De\phi = \big|\na_tu_s\big|_g^2+\big|\na_su_t\big|_g^2+ 
\blr{R_g(u_t,u_s)u_t,u_s}_{\!g}
-\blr{(\na_sJ)\na_t u_t,u_t}_{\!g} -\blr{\na_s((\na_tJ)u_t),u_t}_{\!g}\,,
\end{split}\EE
where $R_g$ is the curvature tensor of the connection~$\na$.
Since $u(\Om)\!\subset\!B_r^g(x)$,
\BE{LapEng_e5}\begin{aligned}
\big|\lr{R_g(u_t,u_s)u_t,u_s}_g\big|&\le C_g(x,r)|u_s|_g^2|u_t|_g^2\,,\\
\big|\lr{(\na_sJ)\na_t u_t,u_t}_g\big|&\le  C_{J,g}(x,r)|u_s|_g|u_t|_g\big|\na_t(Ju_s)\big|_g 
\le C_{J,g}(x,r)|u_s|_g|u_t|_g\big(|u_s|_g|u_t|_g\!+\!|\na_tu_s|_g\big)\\
&\le (C_{J,g}(x,r)\!+\!C_{J,g}(x,r)^2)|u_s|_g^2|u_t|_g^2+|\na_tu_s|_g^2\,,\\
\big|\lr{\na_s((\na_tJ)u_t),u_t}_g\big|
&\le C_{J,g}(x,r)|u_t|_g^2\big(|u_s|_g|u_t|_g\!+\!|\na_su_t|_g\big)\\
&\le C_{J,g}(x,r)|u_s|_g|u_t|_g^3+C_{J,g}(x,r)^2|u_t|_g^4+|\na_su_t|_g^2.
\end{aligned}\EE
Combining~\eref{LapEng_e4} and~\eref{LapEng_e5}, we find that 
$$\frac12\De\phi \ge -C(x,r)\big(|u_s|_g^2|u_t|_g^2\!+\!|u_s|_g|u_t|_g^3\!+\!|u_t|_g^4\big)
\ge -3C(x,r)\phi^2,$$
as claimed.
\end{proof}

\subsection{Regularity of $J$-holomorphic maps}
\label{Regul_subs}

\noindent
By Cauchy's Integral Formula, a continuous extension of
a holomorphic map \hbox{$u\!:B_{\R}^*\!\lra\!\C^n$} over the origin
is necessarily holomorphic.
By Proposition~\ref{JHolomreg_prp} below, the same is the case
for a \hbox{$J$-holomorphic} map $u\!:B_{\R}^*\!\lra\!X$
of bounded energy.

\begin{prp}\label{JHolomreg_prp}
Let $(X,J)$ be an almost complex manifold and $g$ be a Riemannian metric on~$X$.
If $R\!\in\!\R^+$ and $u\!:B_R\!\lra\!X$ is a continuous map such that 
$u|_{B_R^*}$ is a $J$-holomorphic map and $E_g(u;B_R^*)\!<\!\i$, 
then $u$ is smooth and $J$-holomorphic on~$B_R$.
\end{prp}

\noindent
For a smooth loop $\ga\!:S^1\!\lra\!X$, define
$$\ga'(\th)=\frac{\nd}{\nd\th}\ga\big(\ne^{\fI\th}\big)\in T_{\ga(\ne^{\fI\th})}X
\qquad\hbox{and}\qquad
\ell_g(\ga)=\int_0^{2\pi}\!\big|\ga'(\th)\big|_g \nd\th\in\R^{\ge0}$$
to be the \sf{velocity} of~$\ga$ and the \sf{length} of~$\ga$, respectively.

\begin{lmm}[Isoperimetric Inequality]\label{JHolomreg_lmm}
Let $(X,J,g)$, $R$, and $u$ be as in Proposition~\ref{JHolomreg_prp} and 
$$\ga_r\!:S^1\lra X,\quad \ga_r\big(\ne^{\fI\th}\big)=u\big(r\ne^{\fI\th}\big)
\qquad\forall\,r\!\in\!(0,R).$$
There exist $\de\!\in\!(0,R)$ and $C\!\in\!\R^+$ such that 
\BE{JHolomreg1_e}E_g\big(u;B_r^*\big)\le C\ell_g(\ga_r)^2 \qquad\forall\,r\!\in\!(0,\de).\EE
\end{lmm}

\begin{proof} Let $\exp$ be as above the statement of Lemma~\ref{distder_lmm},
$\de_g$ and $\om_x$
be as in the first two sentences in the proof of Proposition~\ref{MonotLmm_prp},
$$x_0=u(0), \quad \de_0=\de_g(x_0), \quad \om_0=\om_{x_0}, \qquad
E\!:(0,R)\lra \R, ~~ E(r)\!=\!E_g(u;B_r^*).$$
We can assume that the metric~$g$ is determined by~$J$ and~$\om_0$ 
on $B_{\de_0}^g(x_0)$.\\

\noindent
For a smooth loop $\ga\!:S^1\!\lra\!B_{\de_0}^g(x_0)$, define
\begin{gather*}
\xi_{\ga}\!: S^1\lra T_{x_0}X \qquad\hbox{by}\quad
\exp_{x_0}\xi_{\ga}\!\big(\ne^{\fI\th}\big)=\ga\big(\ne^{\fI\th}\big),
~~\big|\xi_{\ga}\!(\ne^{\fI\th})\big|<\de_0,\\
f_{\ga}\!: B_1\lra X, \qquad f_{\ga}\big(r\ne^{\fI\th}\big)=
\exp_{x_0}\!\!\big(r\xi_{\ga}(\ne^{\fI\th})\!\big).
\end{gather*}
In particular,
\begin{gather*}
\big|\prt_rf_{\ga}(\rho\ne^{\fI\th})\big|_g
=\big|\xi_{\ga}(\ne^{\fI\th})\big|_g\le\ell_g(\ga)/2,\quad
\big|r^{-1}\prt_{\th}f_{\ga}(r\ne^{\fI\th})\big|_g
=\big|\nd_{r\xi_{\ga}(\ne^{\fI\th})}(\xi_{\ga}'(\th))\big|_g\le C\big|\ga'(\th)\big|_g
\end{gather*}
for some $C\!\in\!\R^+$ determined by~$x_0$.
Thus,
\BE{JHolomreg1_e3}\begin{split}
\bigg|\int_{B_1}\!f_{\ga}^*\om_0\bigg|
&\le C\!\!\int_0^{2\pi}\!\!\!\!\int_0^1
\big|\prt_rf_{\ga}(\rho\ne^{\fI\th})\big|_g
\big|r^{-1}\prt_{\th}f_{\ga}(r\ne^{\fI\th})\big|_g r\,\nd r\nd\th\\
&\le C'\ell_g(\ga)\!\!\int_0^{2\pi}\!\!\!\!\int_0^r\!\!\big|\ga'(\th)\big|_gr\,\nd r\nd\th
=\frac12 C'\ell_g(\ga)^2
\end{split}\EE
for some $C,C'\!\in\!\R^+$ determined by~$x_0$ and~$\om_0$.\\

\begin{figure}
\begin{pspicture}(0,-.7)(9,3.2)
\psset{unit=.3cm}
\pscircle[linewidth=.06](11,4){4}\rput(15.8,4){\sm{$\ga_{\rho}$}}
\rput(11.2,4){$f_{\ga_{\rho}}$}
\pscircle[linewidth=.04](26,4){2}\pscircle[linewidth=.06](26,4){6}
\rput(26.4,0.4){$u|_{B_r-B_{\rho}}$}
\rput(28.8,4){\sm{$\ga_{\rho}$}}\rput(32.8,4){\sm{$\ga_r$}}
\pscircle[linewidth=.06](41,4){4}\rput(45.8,4){\sm{$\ga_r$}}
\rput(41.2,4){$f_{\ga_r}$}
\end{pspicture}
\caption{The maps from an annulus and two  disks 
glued together to form the map \hbox{$F_{\rho;r}\!:S^2\!\lra\!X$}
in the proof of Lemma~\ref{JHolomreg_lmm}}
\label{JHolomreg_fig}
\end{figure}
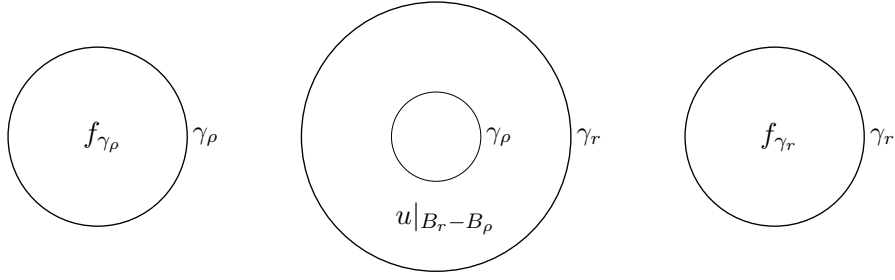

\noindent
By Proposition~\ref{PtBnd_prp} and the finiteness assumption on $E(u;B_R^*)$, 
there exists $\de\!\in\!(0,R/2)$ such~that
\begin{gather}
\label{JHolomreg1_e6}
\big|\ga_r'(\th)\big|_g^2\equiv 
\big|\prt_{\th}u(r\ne^{\fI\th})\big|_g^2=r^2\big|\prt_ru(\ne^{\fI\th})\big|_g^2
\le \frac{32}{\pi}E(2r) \qquad\forall~r\!\in\!(0,\de),\\
\label{JHolomreg1_e7}
\ell_g(\ga_r)^2=128\pi E(2r)
\qquad\forall~r\!\in\!(0,\de).
\end{gather}
By the continuity of~$u$, we can assume that  $u(B_{2\de})\!\subset\!B_{\de_0}^g(x_0)$.
For $r\!\in\!(0,\de)$ and $\rho\!\in\!(0,r)$, define
$$F_{\rho;r}\!:S^2\lra X$$ 
to be the map obtained from $u|_{B_r-B_{\rho}}$ by attaching disks to the boundary components 
$\prt B_r$ and $\prt B_{\rho}$ and letting $F_{\rho;r}$ be given by 
$f_{\ga_r}$ and~$f_{\ga_{\rho}}$ on these two disks, respectively;
see Figure~\ref{JHolomreg_fig}.
Since $F_{\rho;r}$ is  homotopic to a constant~map and $\om_0$ is closed,
$$0=\int_{S^2}\!F_{\rho;r}^{\,*}\om_0
=E_g\big(u;B_r\!-\!B_{\rho}\big)+\int_{B_1}\!\!\!f_{\ga_{\rho}}^{\,*}\om_0
-\int_{B_1}\!\!\!f_{\ga_r}^{\,*}\om_0\,.$$
Combining this with~\eref{JHolomreg1_e3} and~\eref{JHolomreg1_e7}, we obtain
\BE{JHolomreg1_e11}E_g\big(u;B_r\!-\!B_{\rho}\big)\le C\ell_g(\ga_r)^2+CE(2\rho)\EE
for some $C\!\in\!\R^+$ independent of $r$ and $\rho$ as above.
Since $E_g(u;B_R^*)\!<\!0$, $E(2\rho)\!\lra\!0$ as \hbox{$\rho\!\lra\!0$}.
Taking the limit of~\eref{JHolomreg1_e11} as $\rho\!\lra\!0$, 
we thus obtain~\eref{JHolomreg1_e}.
\end{proof}

\begin{crl}\label{JHolomreg_crl}
If $(X,J,g)$, $R$, and $u$ are as in Proposition~\ref{JHolomreg_prp},
there exist $\de\!\in\!(0,R)$ and $\mu,C\!\in\!\R^+$ such~that 
\BE{JHolomreg2_e}
\big|\nd_{r\ne^{\fI\th}}u\big|_g\le Cr^{\mu-1} 
\qquad\forall\,r\!\in\!(0,\de).\EE
\end{crl}

\begin{proof}
Let $\ga_r$, $\de$, $C$, and $E(r)$
be as in the statement and proof of Lemma~\ref{JHolomreg_lmm}.
Thus,
\begin{equation*}\begin{split}
E(r)\equiv\frac12\int_0^{2\pi}\!\!\!\!\int_0^r\!\!\!
\big|\nd_{\rho\ne^{\fI\th}}u\big|_g^2 \rho\nd\rho\nd\th
&\le C\ell_g(\ga_r)^2= \frac12 Cr^2
\bigg(\int_0^{2\pi}\!\!\!\big|\nd_{r\ne^{\fI\th}}u\big|_g\nd\th\!\!\bigg)^{\!\!2}\\
&\le C\pi r^2\!\!\!\int_0^{2\pi}\!\!\!\big|\nd_{r\ne^{\fI\th}}u\big|_g^2\nd\th
=2C\pi r E'(r) \qquad\forall\,r\!\in\!(0,\de).
\end{split}\end{equation*}
This implies that 
$$\big(r^{-1/2C\pi}E(r)\big)'\ge0, ~~ 
E(r)\le \de^{-1/2C\pi}E(\de)\cdot r^{1/2C\pi}\equiv C'r^{2\mu}
\qquad\forall\,r\!\in\!(0,\de).$$
Combining this with \eref{JHolomreg1_e6}, we obtain~\eref{JHolomreg2_e}
with $\de$ replaced by~$\de/2$.
\end{proof}

\begin{proof}[{\bf\emph{Proof of Proposition~\ref{JHolomreg_prp}}}]
With $\mu$ as in Corollary~\ref{JHolomreg_crl}, let $p\!\in\!\R^+$ be such that 
$p\!>\!2$ and $(1\!-\!\mu)p\!<\!2$.
In particular, 
$$u|_{B_{R/2}}\in L^p_1\big(B_{R/2};X\big), \qquad
\dbar_Ju|_{B_{R/2}}=0\in L^p\big(B_{R/2};X\big).$$
By elliptic regularity, this implies that $u$ is smooth; see 
\cite[Theorem~B.4.1]{MS12}.
By the continuity of $\dbar_Ju$, $u$ is then $J$-holomorphic on all of~$B_R$.
\end{proof}

\subsection{Global structure of $J$-holomorphic maps}
\label{GlobStr_subs}

\noindent
We next combine the local statement of Proposition~\ref{FHS_prp}
and some of its implications with the regularity statement of
Proposition~\ref{JHolomreg_prp} to obtain a global description of 
$J$-holomorphic maps.

\begin{prp}\label{SimplMap_prp}
Let $(X,J)$ be an almost complex manifold, $(\Si,\fj)$ 
be a compact Riemann surface, $u\!:\Si\!\lra\!X$ be a $J$-holomorphic map.
If $u$ is simple, then $u$ is somewhere injective and all limit points
of the~set 
\BE{SimplMap_e}\big\{z\!\in\!\Si\!:\,|u^{-1}(u(z))|\!>\!1\big\}\EE
are critical points of~$u$.
\end{prp}

\noindent
Suppose $(X,J)$ is an almost complex manifold, 
$(\Si,\fj)$ is a Riemann surface, and $u\!:\Si\!\lra\!X$ is a $J$-holomorphic map.
Let 
\BE{Siudfn_e}\Si_u^*=\Si-u^{-1}\big(u\big(\{z\!\in\!\Si:\nd_zu\!=\!0\}\big)\!\big)\EE
be the preimage of the regular values of~$u$ and 
$$\cR_u^*\subset \Si_u^*\!\times\!\Si_u^*$$
be the subset of pairs $(z,z')$ such that there exists a diffeomorphism
$\vph_{z'z}\!:U_z\!\lra\!U_{z'}$ between neighborhoods of~$z$ and~$z'$ in~$\Si$
satisfying 
\BE{vphcond_e}\vph_{z'z}(z)=z' \qquad\hbox{and}\qquad u|_{U_z}=u\!\circ\!\vph_{z'z}.\EE
Denote by $\cR_u\!\subset\!\Si\!\times\!\Si$ the closure of $\cR_u^*$.\\

\noindent
It is immediate that $\cR_u^*$ is an equivalence relation on~$\Si$ and 
$u(z)\!=\!u(z')$ whenever $(z,z')\!\in\!\cR_u^*$.
Thus, $\cR_u$ is also a reflexive and symmetric relation and
 $u(z)\!=\!u(z')$ whenever $(z,z')\!\in\!\cR_u$.
By Lemma~\ref{SimplMap_lmm4} below, $\cR_u$ is transitive as well.
We denote this equivalence relation by~$\sim_u$.
Let
\BE{SimplMap_e2}h_u\!:\Si\lra\Si'\!\equiv\!\Si/\!\!\sim_u \qquad\hbox{and}\qquad
u'\!:\Si'\lra X\EE
be the quotient map and the continuous map induced by $u$, respectively.
In particular,
$$u\!=\!u'\!\circ\!h_u\!:\Si\lra X.$$
In the case $\Si$ is compact, we will show that $\Si'$ inherits a Riemann surface structure~$\fj'$
from~$\fj$ so that the maps~$h_u$ and~$u'$ are $\fj'$- and $J$-holomorphic, respectively.
If the degree of~$h$ is~1, we will show that all limit points of 
the set~\eref{SimplMap_e} are critical points of~$u$.

\begin{lmm}\label{SimplMap_lmm2}
Suppose $(X,J)$ is an almost complex manifold, $R\!\in\!\R^+$, 
and $u\!:B_R\!\lra\!X$ is a non-constant $J$-holomorphic map
such that $\nd_zu\!\neq\!0$ for all $z\!\in\!B_R^*$.
Then there exist $m\!\in\!\Z^+$ and a neighborhood~$U_0$ of~$0$ in~$B_R$
such~that 
\BE{SimplMap2_e}h_u\!: U_0\!\cap\!B_R^*\lra h_u\big(U_0\!\cap\!B_R^*\big)\subset B_R'\EE 
is a covering projection of degree~$m$.
\end{lmm}

\begin{proof}
By the continuity of $u$, we can assume that $X\!=\!\C^n$, $u(0)\!=\!0$,
and $J_0\!=\!J_{\C^n}$. 
As shown in the proof of Corollary~\ref{FHS_crl3}, 
there exist $\ep\!\in\!(0,R)$ and $\de\!\in\!(0,\ep/2)$
such~that
$$U_0\equiv u^{-1}\big(u(B_{\de})\big)\!\cap\!B_{\ep}\subset B_{2\de}.$$
By Proposition~\ref{FHS_prp} and the compactness of $\ov{B_{2\de}}\!\subset\!B_R$, 
the number 
$$m(z)\equiv \big|h_u^{-1}(h_u(z))\!\cap\!U_0\big|$$
is finite for every $z\!\in\!U_0\!\cap\!B_R^*$.\\

\noindent
Suppose $z_i\!\in\!B_{\de}^*$ and $z_i'\!\in\!U_0$
are sequences such that~$z_i$ converges to some $z_0\!\in\!B_{\de}^*$
with $z_i\!\neq\!z_0$ for all~$i$
and \hbox{$h_u(z_i)\!=\!h_u(z_i')$} for all~$i$.
Passing to a subsequence, we can assume that $z_i'$ converges to some 
\hbox{$z_0'\!\in\!\ov{B_{2\de}}$}.
By the continuity of~$u$, $u(z_0')\!=\!u(z_0)$ and so $z_0'\!\in\!U_0$.
Corollary~\ref{FHS_crl2b} then implies that $h_u(z_0')\!=\!h_u(z_0)$.
Since $B_{\de}^*$ is connected, this implies that 
the number~$m(z)$ is independent of $z\!\in\!U_0\!\cap\!B_R^*$;
we denote it by~$m$.\\

\noindent
Suppose $z\!\in\!U_0\!\cap\!B_R^*$ and
$$h_u^{-1}\big(h_u(z)\big)\cap U_0=\big\{z_1,\ldots,z_m\big\}\,.$$
Let $\vph_i\!:U_1\!\lra\!U_i$ for $i\!=\!1,\ldots,m$ be diffeomorphisms between 
neighborhoods of~$z_1$ and~$z_i$ in~$U_0\!\cap\!B_R^*$ such that 
$$\vph_i(z_1)=z_i,~~ u=u\!\circ\!\vph_i~~\forall\,i,\qquad
U_i\cap U_j=\eset~~\forall\,i\!\neq\!j,$$
and $u\!:U_1\!\lra\!X$ is injective.
Then $h_u(U_1)\!\subset\!B_R'$ is an open neighborhood of $h_u(z)$, 
$$h_u^{-1}\!\big(h_u(U_1)\!\big)\cap U_0=\bigsqcup_{i=1}^mU_i\,,$$
and $h_u\!:U_i\!\lra\!h_u(U_1)$ is a homeomorphism.
Thus, \eref{SimplMap2_e} is a covering projection of degree~$m$.
\end{proof}

\begin{lmm}\label{SimplMap_lmm3}
Suppose $(X,J)$, $R$, and $u$ are as in Lemma~\ref{SimplMap_lmm2}.
Then there exists a neighborhood~$U_0$ of~$0$ in~$B_R$ such~that 
\BE{SimplMap3_e}\Psi_0\!: h_u(U_0)\lra\C, \qquad
h_u(z)=\prod_{z'\in h_u^{-1}(h_u(z))\cap U_0}\hspace{-.2in}z',\EE 
is a homeomorphism from an open neighborhood of $h_u(0)$ in $B_R'$
to an open neighborhood of~0 in~$\C$ and $\Psi_0\!\circ\!h_u|_{U_0}$ is a holomorphic map.
\end{lmm}

\begin{proof}
By Lemma~\ref{SimplMap_lmm2}, there exists a neighborhood~$U_0$ of~$0$ in~$B_R$
so that  \eref{SimplMap2_e} is a covering projection of some degree~$m\!\in\!\Z^+$.
Since the restriction of~$u$ to $B_R^*$ is a $J$-holomorphic immersion,
the diffeomorphisms~$\vph_i$ as in the proof of Lemma~\ref{SimplMap_lmm2}
are holomorphic.
Thus, the~map
$$\Psi_0\!\circ\!h_u|_{U_0\cap B_R^*}\!:U_0\!\cap\!B_R^*\lra\C, \qquad
z\lra\prod_{z'\in h_u^{-1}(h_u(z))\cap U_0}\hspace{-.2in}z'$$
is holomorphic.
Since it is also bounded, it extends to a holomorphic map
$$\wt\Psi_0\!: U_0\lra\C.$$
This extension is non-constant and vanishes at~0.\\

\noindent
After possibly shrinking~$U_0$, we can assume that there exist $k\!\in\!\Z^+$
and $C\!\in\!\R^+$ such~that 
\BE{SimplMap3_e3}  C^{-k}|z|^k\le \big|\wt\Psi_0(z)\big|\le C^k|z|^k
\qquad\forall\,z\!\in\!U_0.\EE
Since $\wt\Psi_0(z')\!=\!\wt\Psi_0(z)$ for all $z'\!\in\!h_u^{-1}(h_u(z))\!\cap\!U_0$,
it follows~that 
\begin{gather*}
C^{-2}|z|\le |z'|\le C^2|z|
\qquad\forall~z'\!\in\!h_u^{-1}(h_u(z))\!\cap\!U_0,~z\!\in\!U_0,\\
C^{-2m}|z|^m\le \big|\wt\Psi_0(z)\big|\le C^{2m}|z|^m
\qquad\forall\,z\!\in\!U_0.
\end{gather*}
Along with~\eref{SimplMap3_e3}, the last estimate implies that $k\!=\!m$
and that $\wt\Phi_0$ has a zero of order precisely~$m$ at $z\!=\!0$.
Thus, shrinking $\de$ in the proof of Lemma~\ref{SimplMap_lmm2} if necessary,
we can assume that~$\wt\Phi_0$ is $m\!:\!1$ over $\ov{U_0}\!\cap\!B_R^*$.
This implies that the map~\eref{SimplMap3_e} and its extension over the closure
of $h_u(U_0)$ in~$B_R'$ are continuous and injective.
Since the closure of $h_u(U_0)$ is compact and~$\C$ is Hausdorff,
we conclude that~\eref{SimplMap3_e} is a homeomorphism onto an open subset 
of~$\C$.
\end{proof}

\begin{lmm}\label{SimplMap_lmm4}
Suppose $(X,J)$, $(\Si,\fj)$, and~$u$ are as in Proposition~\ref{SimplMap_prp}
and $(x,y)\!\in\!\cR_u$.
For every neighborhood~$U_x$ of~$x$ in~$\Si$, the image of the projection
$$\cR_u\cap (U_x\!\times\!X)\lra X$$
to the second component contains a neighborhood~$U_y$ of~$y$ in~$\Si$.
\end{lmm}

\begin{proof}
By Corollary~\ref{FHS_crl2}, the last set in~\eref{Siudfn_e} is finite.
By the same reasoning as in the last part of the proof of  Lemma~\ref{SimplMap_lmm2},
\BE{SimplMap4_e1}h_u\!:\Si_u^*\lra h_u(\Si_u^*)\subset \Si'\EE
is a local homeomorphism. 
Since $u(z)\!=\!u(z')$ for all $(z,z')\!\in\!\cR_u^*$, 
the definition of~$\Si_u^*$ thus implies that~\eref{SimplMap4_e1} 
is a finite-degree covering projection over each topological component of~$h_u(\Si_u^*)$.
Since the complement of finitely many points in a connected Riemann surface is connected,
the degree of this covering over a point $h_u(z)$ depends only on the topological component
of~$\Si$ containing~$z$.
For any point $z\!\in\!\Si$, not necessarily in~$\Si_u^*$, we denote this degree by~$d(z)$.\\

\noindent
By Corollary~\ref{FHS_crl2}, the set 
$$S\equiv u^{-1}\big(u(x)\big)\subset \Si$$
is finite. 
Let $W\!\subset\!X$ be a neighborhood of $u(x)$ such that the topological components~$\Si_s$
of~$u^{-1}(W)$ containing the points $s\!\in\!S$ are pairwise disjoint
(if $U$ is a union of disjoint balls around the points of~$S$, then
$$W \equiv X-u(\Si\!-\!U)$$
works).
By Lemma~\ref{SimplMap_lmm2}, for each $s\!\in\!S$ there exists a neighborhood~$U_s'$
of~$s$ in~$\Si_s$ such~that 
$$h_u\!: U_s'\!-\!\{s\}\lra h_u\big(U_s'\!-\!\{s\}\big)\subset \Si'$$
is a covering projection of some degree $m_s\!\in\!\Z^+$;
we can assume that $U_x'\!\subset\!U_x$.
Along with  the compactness of~$\Si$, the former implies that
\BE{SimplMap4_e3}\begin{aligned}
&\big|h_u^{-1}\big(h_u(y')\big)\cap U_s'\big|\in \big\{0,m_s\big\}
&\qquad &\forall\,y'\!\in\!U_{s'}'\!\cap\!\Si_u^*,~s,s'\!\in\!S, \\
&\sum_{s\in S}\big|h_u^{-1}\big(h_u(y')\big)\cap U_s'\big|=d(s')
&\qquad
&\forall~y'\!\in\!U_{s'}'\!\cap\!\Si_u^*,~s'\!\in\!S.
\end{aligned}\EE
Define
$$\cP_y(S)=\big\{S'\!\subset\!S\!:\sum_{s\in S'}\!m_s\!=\!d(y)\big\}.$$

\vspace{.2in}

\noindent
Let $U_y''\!\subset\!U_y'$ be a connected neighborhood of $y$.
For each $S'\!\in\!\cP_y(S)$, define
$$U_{y;S'}''=\big\{y'\!\in\!U_y''\!\cap\!\Si_u^*\!:
\{s\!\in\!S\!:h_u^{-1}(h_u(y'))\!\cap\!U_s'\!\neq\!\eset\}\!=\!S'\big\}.$$
By~\eref{SimplMap4_e3}, these sets partition $U_y''\!\cap\!\Si_u^*$.
Since each set
$$\big\{y'\!\in\!U_y''\!\cap\!\Si_u^*\!:h_u^{-1}(h_u(y'))\!\cap\!U_s'\!\neq\!\eset\big\}$$
is open, \eref{SimplMap4_e3} also implies that each set $U_{y;S'}''$ is open.
Since the set $U_y''\!\cap\!\Si_u^*$ is connected, it follows that 
$U_y''\!\cap\!\Si_u^*\!=\!U_{y;S_y}''$ for some $S_y\!\in\!\cP_y(S)$.
Since $(x,y)\!\in\!\cR_u$, $x\!\in\!S_y$.
Thus, the image of the projection
$$\cR_u\cap (U_x'\!\times\!X)\lra X$$
to the second component contains the neighborhood~$U_y''$ of~$y$ in~$\Si$.
\end{proof}

\begin{crl}\label{SimplMap_crl}
Suppose $(X,J)$, $(\Si,\fj)$, and~$u$ are as in Proposition~\ref{SimplMap_prp}.
The quotient map~$h_u$ in~\eref{SimplMap_e2} is open and closed.
\end{crl}

\begin{proof}
The openness of $h_u$ is immediate from Lemma~\ref{SimplMap_lmm4}.
Suppose $A\!\subset\!\Si$ is a closed subset and $y_i\!\in\!h_u^{-1}(h_u(A))$
is a sequence converging to some $y\!\in\!\Si$.
Let $x_i\!\in\!A$ be such that \hbox{$h_u(x_i)\!=\!h_u(y_i)$}.
Passing to a subsequence, we can assume that the sequence $x_i$ converges to some $x\!\in\!A$.
Since $\Si\!-\!\Si_u^*$ consists of isolated points, we can also assume that 
$y_i\!\in\!\Si_u^*$ and so $(x_i,y_i)\!\in\!\cR_u^*$.
Thus, $(x,y)\!\in\!\cR_u$ and so $y\!\in\!h_u^{-1}(h_u(A))$.
We conclude that $h_u$ is a closed~map.
\end{proof}

\begin{proof}[{\bf\emph{Proof of Proposition~\ref{SimplMap_prp}}}]
Let $\Si'$, $h_u$, and $u'$ be as in~\eref{SimplMap_e2}.
By the second statement in Corollary~\ref{SimplMap_crl} and \cite[Lemma~73.3]{Mu},
$\Si'$ is a Hausdorff topological space.
Fix a Riemannian metric~$g$ on~$X$.\\

\noindent
For $(z,z')\!\in\!\cR_u^*$ with $z\!\neq\!z'$, 
the neighborhoods $U_z$ and~$U_{z'}$ as in~\eref{vphcond_e}
can be chosen so that they are disjoint and $u|_{U_z}$ is an embedding.
Since $u$ is  $J$-holomorphic, $\vph_{z'z}$ is then a biholomorphic map 
and $h_u|_{U_z}$ is a homeomorphism onto $h_u(U_z)\!\subset\!\Si'$.
Thus, the Riemann surface structure~$\fj$ on~$\Si$ determines a Riemann surface structure~$\fj'$
on~$h_u(\Si_u^*)$ so that $h_u|_{\Si_u^*}$ is a holomorphic covering projection
of~$h_u(\Si_u^*)$ and $u'|_{h_u(\Si_u^*)}$ is a $J$-holomorphic map with
\BE{SimplMap_e5}E_g\big(u';h_u(\Si_u^*)\big)\le E_g(u).\EE
By Corollary~\ref{FHS_crl2}, $\Si_u'\!-\!h_u(\Si_u^*)$ consists of finitely many points.
By the first statement in Corollary~\ref{SimplMap_crl} and by Lemma~\ref{SimplMap_lmm3},  
$\fj'$ extends over these points to a Riemann surface structure
on~$\Si'$; we denote the extension also by~$\fj'$.
Since the continuous map $h_u$ is $\fj'$-holomorphic outside of the finitely many points
of $\Si\!-\!\Si_u^*$, it is holomorphic everywhere.
Since the continuous map~$u'$ is $J$-holomorphic on~$h_u(\Si_u^*)$,
\eref{SimplMap_e5} and Proposition~\ref{JHolomreg_prp} imply that it is $J$-holomorphic everywhere.\\

\noindent
Suppose $z\!\in\!\Si$ and $z_i,z_i'\!\in\!\Si$ with $i\!\in\!\Z^+$ are such that 
$$\nd_zu\neq0, \qquad z_i\neq z_i',~u(z_i)=u(z_i')~~\forall\,i, \qquad \lim_{i\lra\i}z_i=z.$$
Passing to a subsequence, we can assume that the sequence $z_i'$ converges
to some point \hbox{$z'\!\in\!\Si$} with $u(z')\!=\!u(z)$.
Since the restriction of~$u$ to a neighborhood of~$z$ is an embedding,
$z'\!\neq\!z$.
By Corollary~\ref{FHS_crl2b}, there exists a diffeomorphism~$\vph_{z'z}$
as in~\eref{vphcond_e}.
Thus, $h_u(z)\!=\!h_u(z')$, the map $h_u$ is not injective, and $u$ is not simple.
\end{proof}

\subsection{Energy bound on long cylinders}
\label{CylEner_subs}

\noindent
Proposition~\ref{CylEner_prp} and Corollary~\ref{CylEner_crl} below 
concern $J$-holomorphic maps from long cylinders.
Their substance is that most of the energy and variation of such maps
are concentrated near the~ends.
These technical statements are used to obtain important geometric conclusions
in Sections~\ref{Conver_subs1} and~\ref{Conver_subs2}.

\begin{prp}\label{CylEner_prp}
If $(X,J)$ is an almost complex manifold and $g$ is a Riemannian metric on~$X$, then 
there exist continuous functions $\de_{J,g},\hb_{J,g},C_{J,g}\!:X\!\lra\!\R^+$
with the following properties.
If \hbox{$u\!:[-R,R]\!\times\!S^1\!\lra\!X$} is a \hbox{$J$-holomorphic} map such that
$\Im\,u\subset B_{\de_{J,g}(u(0,1))}^g(u(0,1))$, then 
\BE{CylEner_e}
E_g\big(u;[-R\!+\!T,R\!-\!T]\!\times\!S^1\big) \le 
C_{J,g}\big(u(1,0)\big)\ne^{-T}E_g(u) \qquad\forall~T\ge0\,.\EE
If in addition $E_g(u)<\hb_{J,g}\big(u(0,1)\big)$, then
\BE{CylEner_e2}
\diam_g\big(u([-R\!+\!T,R\!-\!T]\!\times\!S^1)\big)
\le C_{J,g}\big(u(1,0)\big) \ne^{-T/2}\sqrt{E_g(u)} \quad\forall~T\ge1\,.\EE
\end{prp}

\begin{crl}\label{CylEner_crl}
If $(X,J)$ is a compact almost complex manifold and $g$ is a Riemannian metric on~$X$, 
there exist $\hb_{J,g},C_{J,g}\!\in\!\R^+$ with the following property.
If $u\!:[-R,R]\!\times\!S^1\!\lra\!X$ is a \hbox{$J$-holomorphic} map such that 
$E_g(u)\!<\!\hb_{J,g}$, then 
\begin{alignat*}{2}
E_g\big(u;[-R\!+\!T,R\!-\!T]\!\times\!S^1\big) &\le C_{J,g}\ne^{-T}E_g(u)
&\qquad&\forall~T\ge1,\\
\diam_g\big(u([-R\!+\!T,R\!-\!T]\!\times\!S^1)\big)
&\le C_{J,g}\ne^{-T/2}\sqrt{E_g(u)} &\quad&\forall~T\ge2\,.
\end{alignat*}
\end{crl}

\vspace{.1in}

\noindent
As an example, the energy of the injective map
$$[-R,R]\times S^1\lra \C, \qquad (s,\th)\lra s\ne^{\fI\th}\,,$$
is the area of its image, i.e.~$\pi(\ne^{2R}\!-\!\ne^{-2R}\big)$.
Thus, the exponent $\ne^{-T}$ in~\eref{CylEner_e} can be replaced by~$\ne^{-2T}$ 
in this case.
The proof of Proposition~\ref{CylEner_prp} shows that in general the exponent
can be taken to be~$\ne^{-\mu T}$ with $\mu$ arbitrarily close to~2, 
but at the cost of increasing~$C_{J,g}$ and reducing~$\de_{J,g}$.

\begin{lmm}[Poincare Inequality]\label{poin_lmm}
If $f\!: S^1\!\lra\!\R^n$ is a smooth function such that
$\int_0^{2\pi}\!f(\th)\nd\th\!=\!0$, then
$$\int_0^{2\pi}\!\!|f(\th)|^2\nd\th\le\int_0^{2\pi}\!\!|f'(\th)|^2\nd\th.$$
\end{lmm}

\noindent
{\it Proof:} We can write 
$f(\th)=\sum\limits_{k>-\i}\limits^{k<\i}\!a_k\ne^{\fI k\th}$.
Since $\int_0^{2\pi}\!f(\th)\nd\th\!=\!0$, $a_0\!=\!0$. Thus,
$$\int_0^{2\pi}\!\!|f(\th)|^2\nd\th=
\sum_{k>-\i}^{k<\i}\!|a_k|^2\le 
\sum_{k>-\i}^{k<\i}\!|ka_k|^2=
\int_0^{2\pi}\!\!|f'(\th)|^2\nd\th.$$

\begin{proof}[{\bf\emph{Proof of Proposition~\ref{CylEner_prp}}}]
It is sufficient to establish the first statement under the assumption that 
$(X,g)$ is $\C^n$ with the standard Riemannian metric,
$J$ agrees with the standard complex structure~$J_{\C^n}$ at $0\!\in\!\C^n$,
and $u(0,1)\!=\!0$.
Let 
$$\bar\prt u=\frac12\big(u_t+J_{\C^n}u_{\th}\big)\,.$$
By our assumptions, there exist $\de',C\!>\!0$ (dependent on~$u(0,1)$) such that 
\BE{CylEnerPf_e1} \big|\bar\prt_z u\big|\le C\de \big|\nd_z u\big|
\qquad\forall~z\in u^{-1}\big(B_{\de}(0)\big),~\de\le\de'\,.\EE
Write $u\!=\!f\!+\!\fI g$, with $f,g$ taking values in~$\R^n$
and assume that $\Im\,u\!\subset\!B_{\de}(0)$.
By~\eref{omEner_e}, 
$$|\nd u|^2=4\big|\bar\prt u\big|^2+2\nd(f\!\cdot\!\nd g).$$
Combining this with~\eref{CylEnerPf_e1} and Stokes' Theorem, we obtain
\BE{CylEnerPf_e3}
\int_{[-t,t]\times S^1}\!\!|\nd u|^2
\le 4C^2\de^2\int_{[-t,t]\times S^1}\!\!|\nd u|^2
+2\int_{\{t\}\times S^1}f\!\cdot\!g_{\th}\,\nd\th
-2\int_{\{-t\}\times S^1}f\!\cdot\!g_{\th}\,\nd\th\,.\EE
Let $\ti{f}\!=\!f\!-\!\frac1{2\pi}\int_0^{2\pi}\!f\nd\th$.
By H\"older's inequality and Lemma~\ref{poin_lmm}, 
\BE{CylEnerPf_e5}\begin{split}
\bigg|\int_{\{\pm t\}\times S^1}\!f\!\cdot\!g_{\th}\,\nd\th\bigg|
&=\bigg|\int_{\{\pm t\}\times S^1}\!\ti{f}\!\cdot\!g_{\th}\,\nd\th\bigg|
\le \bigg(\int_{\{\pm t\}\times S^1}\!|\ti{f}|^2\nd\th \bigg)^{\!\frac12}
\bigg(\int_{\{\pm t\}\times S^1}\!|g_{\th}|^2\nd\th \bigg)^{\!\frac12}\\
&\le \bigg(\int_{\{\pm t\}\times S^1}\!|\ti{f}_{\th}|^2\nd\th \bigg)^{\!\frac12}
\bigg(\int_{\{\pm t\}\times S^1}\!|g_{\th}|^2\nd\th \bigg)^{\!\frac12}
\le \frac12 \int_{\{\pm t\}\times S^1}\!|u_{\th}|^2\nd\th\,.
\end{split}\EE
Since 
$$3|u_{\th}|^2=2|u_{\th}|^2+\big|u_t-2\bar\prt u\big|^2
\le 2|\nd u|^2 + 8\big|\bar\prt u\big|^2\,, $$
the inequalities \eref{CylEnerPf_e1}-\eref{CylEnerPf_e5} give
$$ \big(1\!-\!4C^2\de^2\big)\!\!\int_{[-t,t]\times S^1}\!\!|\nd u|^2
\le \frac{2}{3}\big(1\!+\!4C^2\de^2\big) 
\bigg(\int_{\{t\}\times S^1}|\nd u|^2\nd\th
+\int_{\{-t\}\times S^1}|\nd u|^2\nd\th\bigg).$$
Thus, the function 
$$\ve(T)\equiv E_g\big(u;[-R\!+\!T,R\!-\!T]\big)
\equiv \frac12\int_{[-R+T,R-T]\times S^1}\!|\nd u|^2\nd\th\nd t$$
satisfies $\ve(T)\!\le\!-\ve'(T)$ for all $T\!\in\![-R,R]$,
if $\de$ is sufficiently small (depending on~$C$).
This implies~\eref{CylEner_e}.\\

\noindent
Let $h_{J,g}(x)\!=\!(x,\de_{J,g}(x))$,
with $h_{J,g}(\cdot,\cdot)$ as in Proposition~\ref{PtBnd_prp}
and $\de_{J,g}(\cdot)$ as provided by the previous paragraph.
Suppose $u$ also satisfies the last condition in Proposition~\ref{CylEner_prp}.
By Proposition~\ref{PtBnd_prp} and~\eref{CylEner_e},
$$|\nd_{(t,\th)}u|\le 3\sqrt{E_g(u;[-|t|\!-\!1,|t|\!+\!1]\!\times\!S^1)}
\le 3\sqrt{C_{J,g}(u(0,1))}\ne^{(1+|t|-R)/2}\sqrt{E_g(u)}$$
for all $t\in[-R\!+\!1,R\!-\!1]$ and $\th\!\in\!S^1$.
Thus, for all $t_1,t_2\!\in\![-R\!+\!T,R\!-\!T]$ with $T\!\ge\!1$ and 
$\th_1,\th_2\!\in\!S^1$,
\begin{equation*}\begin{split}
d_g\big(u(t_1,\th_1),u(t_2,\th_2)\big)
&\le 3\sqrt{C_{J,g}(u(0,1))E_g(u)}
\bigg(\pi\ne^{(1+|t_1|-R)/2}+\int_{t_1}^{t_2}\!\ne^{(1+|t|-R)/2}\nd t\bigg)\\
&\le \big(3\pi\!+\!12\big) \sqrt{C_{J,g}(u(0,1))}\,\ne^{(1-T)/2}\sqrt{E_g(u)}\,.
\end{split}\end{equation*}
This establishes~\eref{CylEner_e2}. 
\end{proof}

\begin{lmm}\label{CylCoord_lmm}
If $(X,J)$ is a compact almost complex manifold and $g$ is a Riemannian metric on~$X$,
there exists a continuous function $\ep_{J,g}\!:\R^+\!\lra\!\R^+$ 
with the following property.
If  $\de\!\in\!\R^+$  and  \hbox{$u\!:(-R,R)\!\times\!S^1\!\lra\!X$} is a $J$-holomorphic map
with $E_g(u)\!<\ep_{J,g}(\de)$, then
$$\diam_g\big(u\big([-R\!+\!1,R\!-\!1]\!\times\!S^1\big)\big)\le \de\,.$$
\end{lmm}

\begin{proof}
By Proposition~\ref{MonotLmm_prp} and the compactness of~$X$, 
there exists $c_{J,g}\!\in\!\R^+$ with the following property.
If $(\Si,\fj)$ is a compact connected Riemann surface with boundary,
\hbox{$u\!:\Si\!\lra\!X$} is a non-constant $J$-holomorphic map, 
$x\!\in\!X$, and $\de\!\in\!\R^+$ are such~that 
$u(\prt\Si)\!\cap\!B_{\de}^g(x)\!=\!\eset$, then 
\BE{CylCoord_e3} E_g(u)\ge c_{J,g}\de^2\,.\EE
Let $\hbar_{J,g}\!>\!0$ be the minimal value of the function~$\hbar_{J,g}$
in the statement of Proposition~\ref{PtBnd_prp} on the compact space 
$X\!\times\![0,\diam_g(X)]$.\\

\noindent
Suppose $u\!:(-R,R)\!\times\!S^1\!\lra\!X$ is a $J$-holomorphic map with 
$E_g(u)\!<\!\hb_{J,g}$ and
$$\de_u\equiv\diam_g\big(u([-R\!+\!1,R\!-\!1]\!\times\!S^1)\big)>32\sqrt{E_g(u)}.$$ 
By the first condition on~$u$,
\begin{gather}
\notag
\big|\nd_zu\big|_g^2\le \frac{16}{\pi}
E_g(u) \qquad\forall~z\!\in\![-R\!+\!1,R\!-\!1]\!\times\!S^1,\\
\label{CylCoord_e7}
\diam_g\big(u(r\!\times\!S^1)\big)\le 8\sqrt{E_g(u)}
\quad\forall~r\!\in\![-R\!+\!1,R\!-\!1].
\end{gather}
Let $r_-,r_0,r_+\!\in\![-R\!+\!1,R\!-\!1]$ and $\th_-,\th_0,\th_+\!\in\!S^1$ be such that 
$$r_-<r_0<r_+,\quad d_g\big(u(r_0,\th_0),u(r_{\pm},\th_{\pm})\big)\ge 
\frac12\de_u\,.$$
By~\eref{CylCoord_e7}, we can apply~\eref{CylCoord_e3} with 
$$\Si=[r_-,r_+]\!\times\!S^1,\qquad x=u(r_0,\th_0), \qquad \de=\frac14\de_u,$$
and $u$ replaced by its restriction to~$\Si$.
We conclude~that 
$$E_g(u)\ge  \frac{c_{J,g}}{16}\de_u^2\,.$$
It follows that the function
$$\ep_{J,g}\!:\R^+\lra\R^+, \qquad
\ep_{J,g}(\de)=\min\bigg(\hb_{J,g},\frac{\de^2}{32^2},
\frac{c_{J,g}}{16}\de^2\bigg),$$
has the desired property.
\end{proof}

\begin{proof}[{\bf\emph{Proof of Corollary~\ref{CylEner_crl}}}]
Let $\de\!\in\!\R^+$ be the minimum of the function~$\de_{J,g}$
in Proposition~\ref{CylEner_prp}
and $\ve_{J,g}(\cdot)$ be as in Lemma~\ref{CylCoord_lmm}.
Take $C_{J,g}$ to be the maximum of the function~$C_{J,g}$ in Proposition~\ref{CylEner_prp} 
times~$\ne$ and 
$\hb_{J,g}\!\in\!\R^+$ to be smaller than the minimum of 
the function~$\hb_{J,g}$ in Proposition~\ref{CylEner_prp}
and the number~$\ve_{J,g}(\de)$.
\end{proof}

\section{Limiting behavior of $J$-holomorphic maps}
\label{GlobProp_sec}

\noindent
This section studies the limiting behavior of sequences of $J$-holomorphic maps 
from Riemann surfaces into a compact almost complex manifold~$(X,J)$.
The compactness of~$X$ plays an essential role in the statements below,
in contrast to nearly all statements in Sections~\ref{LocProp_sec} and~\ref{EBnd_sec},

\subsection{Removal of Singularity}
\label{RemSing_subs}

\noindent
By Cauchy's Integral Formula, a bounded holomorphic map $u\!:B_{\R}^*\!\lra\!\C^n$ 
extends over the origin.
By Proposition~\ref{RemSing_prp} below, the same is the case
for a $J$-holomorphic map $u\!:B_{\R}^*\!\lra\!X$ of bounded energy
if $X$ is compact.

\begin{prp}[Removal of Singularity]\label{RemSing_prp}
Let $(X,J)$ be a compact almost complex manifold
and $u\!:B_R^*\!\lra\!X$ be a $J$-holomorphic map.
If the energy $E_g(u)$ of~$u$, with respect to any metric~$g$ on~$X$, 
is finite, then $u$ extends to a $J$-holomorphic map 
$\ti{u}\!:B_R\!\lra\!X$.
\end{prp}

\noindent
A basic example of a holomorphic function $u\!:\C^*\!\lra\!\C$ that does not extend
over the origin $0\!\in\!\C$ is $z\!\lra\!1/z$.
The energy of $u|_{B_R^*}$ with respect to the standard metric on~$\C$ is 
given~by 
$$E\big(u;B_R^*\big)
=\frac12\int_{B_R}|\nd u|^2
=\int_{B_R}\frac{1}{|z|^2} 
=\int_0^{2\pi}\!\!\!\!\int_0^R\!\! r^{-1}\nd r\nd\th \not<\i.$$
The above integral would have been finite if $|\nd u|^2$ were replaced by 
$|\nd u|^{2-\ep}$ for {\it any} $\ep\!>\!0$.
This observation illustrates the crucial role played by 
the energy in the theory of $J$-holomorphic maps.\\

\noindent
By Cauchy's Integral Formula, the conclusion of Proposition~\ref{RemSing_prp} holds
if $J$ is an integrable almost complex structure  and $u(B_{\de}^*)$ 
is contained in a complex coordinate chart for some $\de\!\in\!(0,R)$.
We will use the Monotonicity Lemma to show that the latter is the case
if the energy of~$u$ is finite;
the integrability  of~$J$ turns out to be irrelevant here.

\begin{proof}[{\bf\emph{Proof of Proposition~\ref{RemSing_prp}}}]
In light of Proposition~\ref{JHolomreg_prp},  
it is to sufficient to show that $u$ extends continuously over the origin.  
Let $c_{J,g},\hbar_{J,g}\!\in\!\R^+$ be as in the proof of Lemma~\ref{CylCoord_lmm}.
We can assume that $R\!=\!1$ and $u$ is non-constant.
Define
$$v\!: \R^-\!\times\!S^1\lra X, \qquad v\big(r,\ne^{\fI\th}\big)=u\big(\ne^{r+\fI\th}).$$
This map is $J$-holomorphic and satisfies $E_g(v)\!=\!E_g(u)\!<\!\i$.\\

\noindent
Since $E_g(u)\!<\!\i$, 
\BE{RemSing_e3}\lim_{r\lra-\i}\!\!E_g\big(v;(-\i,r)\!\times\!S^1\big)
=\lim_{r\lra-\i}\!\!E_g\big(u;B_{e^r}^*\big)=0.\EE
In particular, there exists $R\!\in\!\R^-$ such that 
$$E_g\big(v;(-\i,r)\!\times\!S^1\big)<\hbar_{J,g} \qquad\forall~r\!<\!R.$$
By Proposition~\ref{PtBnd_prp} and our choice of $\hbar_{J,g}$, this implies that 
\begin{gather*}
\big|\nd_zv\big|_g^2\le \frac{16}{\pi}E_g\big(v;(-\i,r\!+\!1)\!\times\!S^1\big) 
\qquad\forall~z\!\in\!(-\i,r)\!\times\!S^1,~r\!<\!R\!-\!1,\\
\diam_g\big(v(\{r\}\!\times\!S^1)\big)\le4\sqrt{\pi}
\sqrt{E_g(v;(-\i,r\!+\!1)\!\times\!S^1)}
\quad\forall~r\!<\!R\!-\!1.
\end{gather*}
Combining the last bound with~\eref{RemSing_e3}, we obtain
$$\lim_{r\lra-\i}\diam_g\big(v(\{r\}\!\times\!S^1)\big)=0.$$
Thus, it remains to show that $\lim\limits_{r\lra-\i}\!v(r,1)$ exists.\\

\noindent 
Since $X$ is compact, every sequence in~$X$ has a convergent subsequence.
Suppose there exist
\begin{gather*}
\de\in\R^+,~~x,y\in X,~~i_k,r_k\in\R^- \qquad\hbox{s.t.}\\
d_g(x,y)>3\de,~~
r_{k+1}<i_k<r_k,~~ v\big(\{i_k\}\!\times\!S^1\big)\subset B_{\de}(x),
~~ v\big(\{r_k\}\!\times\!S^1\big)\subset B_{\de}(y).
\end{gather*}
We thus can apply~\eref{CylCoord_e3} with $\Si$, $x$, and $u$ replaced by  
$$\Si_k\equiv[r_{k+1},r_k]\!\times\!S^1,\qquad x_k\equiv u(i_k,1),
\quad\hbox{and}\quad v_k\equiv v|_{\Si_k}\,,$$
respectively.
We conclude~that
$$E_g(v)\ge \sum_kE_g\big(v;\Si_k\big) = \sum_kE_g(v_k)
\ge \sum_k c_{J,g}\de^2=\i\,.$$
However, this contradicts the assumption that $E_g(v)\!=\!E_g(u)\!<\!\i$.
\end{proof}

\begin{figure}
\begin{pspicture}(3,-1)(9,2.7)
\psset{unit=.3cm}
\psline[linewidth=.06](56,7)(20,7)
\psline[linewidth=.06](56,1)(20,1)
\psarc[linewidth=.04](59,4){4.24}{135}{225}
\psarc[linewidth=.04,linestyle=dotted](53,4){4.24}{-45}{45}
\psarc[linewidth=.04](44,4){4.24}{135}{225}
\psarc[linewidth=.04,linestyle=dotted](38,4){4.24}{-45}{45}
\psarc[linewidth=.04](34,4){4.24}{135}{225}
\psarc[linewidth=.04,linestyle=dotted](28,4){4.24}{-45}{45}
\pscircle*(56,7){.2}\rput(56,7.8){1}
\pscircle*(41,7){.2}\rput(41.2,7.8){$r_k$}
\pscircle*(31,7){.2}\rput(31.7,7.8){$r_{k+1}$}
\pscircle*(36,7){.2}\rput(36.2,7.8){$i_k$}
\rput(41,-3){$B_{\de}(y)$}\rput(31,-3){$B_{\de}(y)$}\rput(36,-3){$B_{\de}(x)$}
\rput(36.2,3.5){$\Si_k$}
\psline[linewidth=.04]{->}(41,.5)(41,-2)\psline[linewidth=.04]{->}(36,.5)(36,-2)
\psline[linewidth=.04]{->}(31,.5)(31,-2)
\rput(41.6,-.75){$v$}\rput(31.6,-.75){$v$}\rput(36.6,-.75){$v$}
\end{pspicture}
\caption{Setup for the proof of Proposition~\ref{RemSing_prp}}
\label{RemSing_fig}
\end{figure}
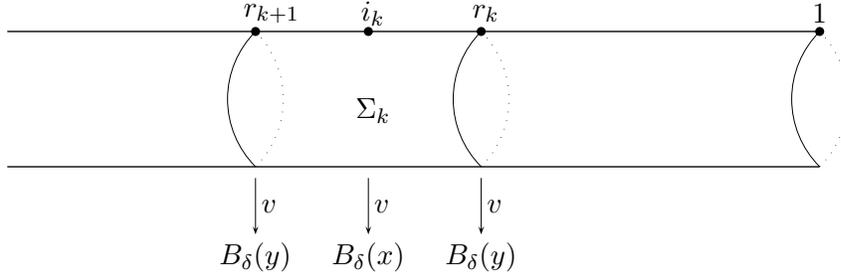

\subsection{Bubbling}
\label{Conver_subs1}

\noindent
The next three statements are used in Section~\ref{Conver_subs1} to show that 
no energy is lost under Gromov's convergence procedure,
the resulting bubbles connect, and their number is finite.

\begin{lmm}\label{EnerPres1_lmm}
Suppose $(X,J)$ is an almost complex manifold with a Riemannian metric~$g$
and \hbox{$u_i\!:B_1\!\lra\!X$} is a sequence of $J$-holomorphic maps converging
uniformly in the $C^{\i}$-topology on compact subsets of~$B_1^*$
to a $J$-holomorphic map $u\!:B_1\!\lra\!X$ such that the limit
\BE{fmdfn_e}\fm\equiv\lim_{\de\lra0}\lim_{i\lra\i}E_g(u_i;B_{\de})\EE
exists and is nonzero. 
\begin{enumerate}[label=(\arabic*),leftmargin=*]

\item\label{deEner_it} The limit $\fm(\de)\equiv\lim\limits_{i\lra\i}E_g(u_i;B_{\de})$ 
exists and is a continuous, non-decreasing function of~$\de$.

\item\label{SeqEner_it} For every sequence $z_i\!\in\!B_1$ converging to~0,
$\lim\limits_{i\lra\i}E_g(u_i;B_{\de}(z_i))\!=\!\fm(\de)$.

\item\label{dela_it} For every sequence $z_i\!\in\!B_1$ converging to~0,
$\mu\!\in\!(0,\fm)$, and $i\!\in\!\Z^+$ sufficiently large,
there exists a unique $\de_i(\mu)\!\in\!(0,1\!-\!|z_i|)$ such that 
$E_g(u_i;B_{\de_i(\mu)}(z_i))\!=\!\mu$.
Furthermore, 
\BE{EnerPres1Elim_e}
\lim_{R\lra\i}\lim_{\de\lra0}\lim_{i\lra\i}E_g\big(u_i;B_{R\de}(z_i)\!-\!B_{\de_i(\mu)}(z_i)\big)
=\fm\!-\!\mu.\EE

\end{enumerate}
\end{lmm}

\begin{proof} 
\ref{deEner_it} Since $\nd u_i$ converges uniformly to $\nd u$ on compact subsets of~$B_1^*$,
\begin{equation*}\begin{split}
\fm(\de)\equiv\lim_{i\lra\i}E_g\big(u_i;B_{\de}\big)
&=\lim_{\de'\lra0}\lim_{i\lra\i}E_g\big(u_i;B_{\de'}\big)+
\lim_{\de'\lra0}\lim_{i\lra\i}E_g\big(u_i;B_{\de}\!-\!B_{\de'}\big)\\
&=\fm+\lim_{\de'\lra0}E_g\big(u;B_{\de}\!-\!B_{\de'}\big)
=\fm+E_g(u;B_{\de}).
\end{split}\end{equation*}
Since $E_g(u;B_{\de})$ is a continuous, non-decreasing function of~$\de$, 
so is~$\fm(\de)$.\\

\noindent
\ref{SeqEner_it} For all $\de,\de'\!\in\!\R^+$ and $z_i\!\in\!B_{\de'}$,
$B_{\de-\de'}\!\subset\!B_{\de}(z_i)\!\subset\!B_{\de+\de'}$.
Thus,
$$E_g\big(u_i;B_{\de-\de'}\big)\le E_g\big(u_i;B_{\de}(z_i)\big)
\le E_g\big(u_i;B_{\de+\de'}\big)$$
for all $i\!\in\!\Z^+$ sufficiently large and
$$\lim_{\de'\lra0}\fm(\de\!-\!\de')\le  
\lim_{\de'\lra0}\lim_{i\lra\i}E_g\big(u_i;B_{\de}(z_i)\big)
\le \lim_{\de'\lra0}\fm(\de\!+\!\de')
\qquad\forall~\de'\in\R^+\,.$$
The claim now follows from~\ref{deEner_it}.\\

\noindent
\ref{dela_it} By~\ref{SeqEner_it}, \ref{deEner_it}, and~\eref{fmdfn_e},
$$\lim_{i\lra\i}E_g\big(u_i;B_{\de}(z_i)\big)=\fm(\de)\ge\fm\,.$$
Thus, there exists $i(\mu)\!\in\!\Z^+$ such that 
$$E_g(u_i;B_{\de}(z_i))>\mu \qquad\forall~i\ge i(\mu).$$ 
Since $E_g(u_i;B_{\de}(z_i))$ is a continuous, increasing function of~$\de$
which vanishes at~$\de\!=\!0$,
for every $i\!\ge\!i(\mu)$
 there exists a unique $\de_i(\mu)\!\in\!(0,\de)$
such that $E_g(u_i;B_{\de_i(\mu)}(z_i))\!=\!\mu$.\\

\noindent
By~\ref{SeqEner_it}, \ref{deEner_it}, and~\eref{fmdfn_e},
\begin{equation*}\begin{split}
\lim_{R\lra\i}\lim_{\de\lra0}\lim_{i\lra\i}E_g\big(u_i;B_{R\de}(z_i)\big)
=\lim_{R\lra\i}\lim_{\de\lra0}\fm(R\de)=\lim_{R\lra\i}\fm=\fm.
\end{split}\end{equation*}
Combining this with the definition of~$\de_i(\mu)$, we obtain~\eref{EnerPres1Elim_e}.
\end{proof}

\begin{crl}\label{EnerPres2_crl}
If $(X,J)$ is a compact almost complex manifold with a Riemannian metric~$g$, 
then there exists $\hbar_{J,g}\!\in\!\R^+$ with the following properties.
If $u_i\!:B_1\!\lra\!X$ is a sequence of $J$-holomorphic maps converging
uniformly in the $C^{\i}$-topology on compact subsets of~$B_1^*$
to a $J$-holomorphic map $u\!:B_1\!\lra\!X$ such that 
$$\lim_{i\lra\i}\max_{\ov{B_{1/2}}}\big|\nd u_i\big|_g=\i$$
and the limit~\eref{fmdfn_e} exists, then 
\begin{enumerate}[label=(\arabic*),leftmargin=*]

\item\label{fmLowBnd_it} $\fm\ge\hb_{J,g}$;

\item\label{noLoss_it} for every sequence $z_i\!\in\!B_{\de}$ converging to~0 and
$\mu\!\in\!(\fm\!-\!\hb_{J,g},\fm)$, 
the numbers \hbox{$\de_i(\mu)\!\in\!(0,1\!-\!|z_i|)$} of Lemma~\ref{EnerPres1_lmm}\ref{dela_it} 
satisfy
\begin{gather}
\label{noEloss_e1} 
\lim_{R\lra\i}\lim_{i\lra\i}E_g(u_i;B_{R\de_i(\mu)}(z_i))=\fm,\\
\label{noEloss_e2} 
\lim_{R\lra\i}\lim_{\de\lra0}
\lim_{i\lra\i} \diam_g\big(u_i(B_{\de}(z_i)\!-\!B_{R\de_i(\mu)}(z_i))\big)=0.
\end{gather}

\end{enumerate}
\end{crl}

\begin{proof}
Let $\hb_{J,g}$ be the smaller of the constants~$\hb_{J,g}$ 
in Corollaries~\ref{LowEner_crl} and~\ref{CylEner_crl}.
Let $u_i$, $u$, and~$\fm$ be as in the statement of the corollary.\\

\noindent
\ref{fmLowBnd_it} For each $i\!\in\!\Z^+$, let 
$$M_i=\max_{\ov{B_{1/2}}}\big|\nd_zu_i\big|_g\in\R^+$$
and $z_i\!\in\!\ov{B_{1/2}}$ be such that $|\nd_{z_i}u_i|_g\!=\!M_i$.
Since $M_i\!\lra\!\i$ as $i\!\lra\!\i$ and
$u_i$ converges uniformly in the $C^{\i}$-topology on compact subsets of~$B_1^*$
to~$u$, $z_i\!\lra\!0$.
For $i\!\in\!\Z^+$ such that $|z_i|\!+\!1/\sqrt{M_i}\!<\!1/2$, define
$$v_i\!:B_{\sqrt{M_i}}\lra X, \qquad v_i(z)=u_i\big(z_i\!+\!z/M_i\big).$$
Thus, $v_i$ is a $J$-holomorphic map with 
\BE{EnerPres2_e3}\sup\big|\nd v_i\big|_g=\big|\nd_0v_i\big|_g=1, \qquad 
E_g(v_i)=E_g\big(u_i;B_{1/\sqrt{M_i}}(z_i)\big)\le 
E_g\big(u_i;B_{|z_i|+1/\sqrt{M_i}}\big)\,.\EE
By the first statement in~\eref{EnerPres2_e3} and the ellipticity of the $\dbar$-operator, 
a subsequence of~$v_i$ converges  uniformly in the $C^{\i}$-topology on compact subsets of~$\C$
to a non-constant $J$-holomorphic map \hbox{$v\!:\C\!\lra\!X$}.
By the second statement in~\eref{EnerPres2_e3} and Lemma~\ref{EnerPres1_lmm}\ref{deEner_it},
\BE{EnerPres2_e5} E_g(v)
\le \limsup_{i\lra\i}E_g\big(u_i;B_{1/\sqrt{M_i}}(z_i)\big)
\le \lim_{\de\lra0}\lim_{i\lra\i}E_g\big(u_i;B_{\de}\big)=\fm.\EE
By Proposition~\ref{RemSing_prp}, $v$ thus extends to a $J$-holomorphic map 
$\wt{v}\!:\P^1\!\lra\!X$.
By Corollary~\ref{LowEner_crl}, 
$$E_g(v)=E_g(\wt{v})\ge \hb_{J,g}\,.$$
Combining this with~\eref{EnerPres2_e5}, we obtain the first claim.\\

\noindent
\ref{noLoss_it} By the first two statements in Lemma~\ref{EnerPres1_lmm}
and~\eref{fmdfn_e},
\BE{noLoss_e2} \lim_{\de\lra0}\lim_{i\lra\i}E_g\big(u_i;B_{\de}(z_i)\big)=
\lim_{\de\lra0}\fm(\de)=\fm.\EE
After passing to a subsequence of~$u_i$, we can thus assume that there exists
a sequence $\de_i\!\lra\!0$ such~that
\BE{noLoss_e3} \lim_{i\lra\i}E_g(u_i;B_{\de_i}(z_i))=\fm.\EE
Since $\de_i\!\lra\!0$, \eref{noLoss_e2} and \eref{noLoss_e3} imply that 
\BE{noLoss_e5}\lim_{R\lra\i}\lim_{i\lra\i}E_g(u_i;B_{R\de_i}(z_i))=\fm.\EE

\vspace{.1in}

\noindent
Suppose $\mu\!\in\!(\fm\!-\!\hb_{J,g},\fm)$.
By \eref{noLoss_e5} and the definition of~$\de_i(\mu)$, 
$$ \lim_{R\lra\i}\lim_{i\lra\i}E_g\big(u;B_{R\de_i}(z_i)\!-\!B_{\de_i(\mu)}(z_i)\big)
=\fm\!-\!\mu<\hb_{J,g}\,.$$
Thus, Corollary~\ref{CylEner_crl} applies with $(R,T)$ replaced by 
$(\frac12\ln(R\de_i/\de_i(\mu)),\ln R)$
and $u$ replaced by the $J$-holomorphic~map
$$v\!:  (-R,R)\!\times\!S^1\lra X,\qquad
v\big(r,\ne^{\fI\th}\big)=u\big(z_i\!+\!\sqrt{R\de_i\de_i(\mu)}\,\ne^{r+\fI\th}\big).$$
By the first statement of Corollary~\ref{CylEner_crl},
$$E_g\big(u;B_{\de_i}(z_i)\big)-E_g\big(u;B_{R\de_i(\mu)}(z_i)\big)=
E_g\big(u;B_{\de_i}(z_i)\!-\!B_{R\de_i(\mu)}(z_i)\big)\le \frac{C_{J,g}}{R}E_g(u)$$
for all $i$ sufficiently large (depending on~$R$);
see Figure~\ref{CylEner_fig}.
Combining this with \eref{noLoss_e3}, we obtain~\eref{noEloss_e1}.\\

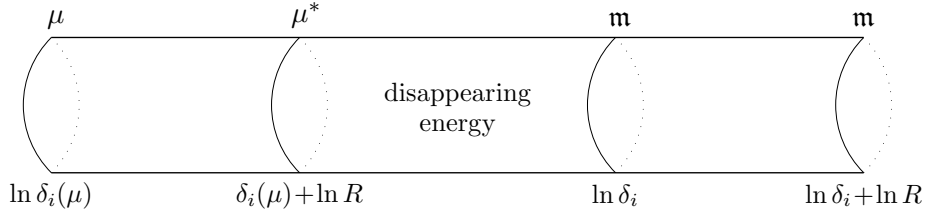
\begin{figure}
\begin{pspicture}(3,-.3)(9,2.7)
\psset{unit=.3cm}
\psline[linewidth=.06](56,7)(20,7)
\psline[linewidth=.06](56,1)(20,1)
\psarc[linewidth=.04](59,4){4.24}{135}{225}
\psarc[linewidth=.04,linestyle=dotted](53,4){4.24}{-45}{45}
\psarc[linewidth=.04](48,4){4.24}{135}{225}
\psarc[linewidth=.04,linestyle=dotted](42,4){4.24}{-45}{45}
\psarc[linewidth=.04](34,4){4.24}{135}{225}
\psarc[linewidth=.04,linestyle=dotted](28,4){4.24}{-45}{45}
\psarc[linewidth=.04](23,4){4.24}{135}{225}
\psarc[linewidth=.04,linestyle=dotted](17,4){4.24}{-45}{45}
\rput(56,7.8){$\fm$}\rput(56,0){\sm{$\ln\de_i\!+\!\ln R$}}
\rput(45.2,7.8){$\fm$}\rput(45,0){\sm{$\ln\de_i$}}
\rput(20.2,7.8){$\mu$}\rput(20,0){\sm{$\ln\de_i(\mu)$}}
\rput(31.3,8){$\mu^*$}\rput(31,0){\sm{$\de_i(\mu)\!+\!\ln R$}}
\rput(38,4.5){\small{disappearing}}\rput(38,3){\small{energy}}
\end{pspicture}
\caption{Illustration for the proof of~\eref{noEloss_e1}}
\label{CylEner_fig}
\end{figure}

\noindent
It remains to establish~\eref{noEloss_e2}.
By~\eref{EnerPres1Elim_e}, for all $R\!>\!0$ and sufficiently small $\de\!>\!0$  (depending on~$R$)
there exists $i(R,\de)\!\in\!\Z^+$ such~that
$$E_g\big(u_i;B_{R\de}(z_i)\!-\!B_{\de_i(\mu)}(z_i)\big)<\hb_{J,g}
\qquad\forall~i>i(R,\de).$$
Thus, Corollary~\ref{CylEner_crl} applies with $(R,T)$ replaced by 
$(\frac12\ln(R\de/\de_i(\mu)),\ln R)$
and $u$ replaced by the $J$-holomorphic~map
$$v\!:  (-R,R)\!\times\!S^1\lra X,\qquad
v\big(r,\ne^{\fI\th}\big)=u\big(z_i\!+\!\sqrt{R\de\de_i(\mu)}\,\ne^{r+\fI\th}\big).$$
By the second statement of Corollary~\ref{CylEner_crl}, 
$$\diam_g\big(u_i(B_{\de}(z_i)\!-\!B_{R\de_i(\mu)}(z_i))\big)
\le \frac{C_{J,g}}{\sqrt{R}}\hb_{J,g} \qquad\forall~i>i(R,\de).$$
This gives~\eref{noEloss_e2}.
\end{proof}

\begin{lmm}\label{EnerPres3_lmm}
If $(X,J)$ is a compact almost complex manifold with a Riemannian metric~$g$,
then there exists a function $N\!:\R\!\lra\!\Z$ with the following property.
If $(\Si,\fj)$ is compact Riemann surface, $S_0\!\subset\!\Si$ is a finite subset,
 and $u_i\!:U_i\!\lra\!X$ is a sequence of $J$-holomorphic maps 
from open subsets of~$\Si$ with 
\BE{EnerPres3_e}U_i\subset U_{i+1}, \qquad \Si\!-\!S_0=\bigcup_{i=1}^{\i}\!U_i,
\quad\hbox{and}\quad E\!\equiv\!\liminf_{i\lra\i}E_g(u_i)<\i,\EE
then there exist a subset $S\!\subset\!\Si$ with $|S|\!\le\!N(E)\!+\!|S_0|$ 
and a subsequence of $u_i$ converging  uniformly in the $C^{\i}$-topology 
on compact subsets of $\Si\!-\!S$
to a $J$-holomorphic map $u\!:\Si\!\lra\!X$.
\end{lmm}

\begin{proof}
Let $\hb_{J,g}$ be the minimal value of the function provided by Proposition~\ref{PtBnd_prp}.
For $E\!\in\!\R^+$, let $N(E)\!\in\!\Z^{\ge0}$ be the smallest integer
such that $E\!\le\!N(E)\hb_{J,g}$.\\

\noindent
Let $\Si$, $S_0$, $u_i$, and $E$ be as in the statement of the lemma and
$N\!=\!N(E)\!+\!|S_0|$.
Fix a Riemannian metric~$g_{\Si}$ on~$\Si$.
For $z\!\in\!\Si$ and $\de\!\in\!\Si$, 
let $B_{\de}(z)\!\subset\!\Si$ denote the ball of radius~$\de$  around~$z$.
By Proposition~\ref{PtBnd_prp}, there exists $C\!\in\!\R^+$ with the following property.
If $u\!:\Si\!\lra\!X$ is a $J$-holomorphic map, $z\!\in\!\Si$, and $\de\!\in\!\R^+$,
then
\BE{EnerPres3_e3}  E_g\big(u;B_{\de}(z)\big)<\hb_{J,g} \qquad\Lra\qquad
\big|\nd_zu\big|_g\le C/\de^2\,.\EE 

\vspace{.1in}

\noindent
For every pair $i,j\!\in\!\Z^+$, 
let $\{z_{ij}^k\}_{k=1}^N$ be a subset of points of~$\Si$ containing~$S_0$ 
such~that
\BE{EnerPres3_e5} z\in \Si_{ij}^*\equiv\Si\!-\!\bigcup_{k=1}^N\!B_{2/j}\big(z_{ij}^k\big)
 \qquad\Lra\qquad
E_g\big(u_i;B_{1/j}(z)\!\cap\!U_i\big)<\hb_{J,g}\,.\EE
By~\eref{EnerPres3_e3} and~\eref{EnerPres3_e5},
\BE{EnerPres3_e7} 
\big|\nd_zu_i\big|_g\le Cj^2 \qquad\forall\,z\!\in\!\Si_{ij}^*
~\hbox{s.t.}~B_{1/j}(z)\!\subset\!U_i\,.\EE

\vspace{.1in}

\noindent
After passing to a subsequence of~$\{u_i\}$, we can assume that 
the sequence $E_g(u_i)$ converges to~$E$ and
that the sequence 
$\{z_{ij}^k\}_{i\in\Z^+}$ converges to some $z_j^k\!\in\!\Si$ for every 
$k\!=\!1,\ldots,N$ and $j\!\in\!\Z^+$.
Along with~\eref{EnerPres3_e7} and 
the first two assumptions in~\eref{EnerPres3_e}, this implies that 
\BE{EnerPres3_e9} \limsup_{i\lra\i}\big|\nd_zu_i\big|_g\le Cj^2 
\qquad\forall\,z\in\Si_{ij}^*\,.\EE
After passing to another subsequence of~$\{u_i\}$, we can assume that the sequence 
$\{z_j^k\}_{j\in\Z^+}$ converges to some $z^k\!\in\!\Si$ for every 
$k\!=\!1,\ldots,N$.
\\

\noindent
By~\eref{EnerPres3_e9} and the ellipticity of the $\dbar$-operator, 
a subsequence of~$u_i$ converges  uniformly in the $C^{\i}$-topology on compact subsets of~$\Si_1^*$
to a $J$-holomorphic map \hbox{$v_1\!:\Si_1^*\!\lra\!X$}.
By~\eref{EnerPres3_e9} and the ellipticity of the $\dbar$-operator,
a subsequence of this subsequence in turn converges  uniformly 
in the $C^{\i}$-topology on compact subsets of~$\Si_2^*$
to a $J$-holomorphic map \hbox{$v_2\!:\Si_2^*\!\lra\!X$}.
Continuing in this way,  we obtain a subsequence of~$u_i$ 
converging  uniformly  in the $C^{\i}$-topology on compact subsets of~$\Si_j^*$
to a $J$-holomorphic map \hbox{$v_j\!:\Si_j^*\!\lra\!X$} for every $j\!\in\!\Z^+$.
The limiting maps satisfy
$$v_j|_{\Si_j\cap\Si_{j'}^*}=v_{j'}|_{\Si_j^*\cap\Si_{j'}^*}
\qquad\forall~j,j'\!\in\!\Z^+\,.$$ 
Thus, the~map
$$u\!:\Si^*\!\equiv\!\Si^*\!-\!\big\{z^k\big\}\lra X, \qquad
u(z)=v_j(z)~~\forall\,z\!\in\!\Si_j^*,
$$
is well-defined and $J$-holomorphic.\\

\noindent
By construction, the final subsequence of $u_i$ converges  uniformly 
in the $C^{\i}$-topology on compact subsets of~$\Si^*$ to~$u$.
This implies~that
$$E_g(u)\le \liminf_{i\lra\i}E_g(u_i)=E\,.$$
By Proposition~\ref{RemSing_prp}, $u$ thus extends to a $J$-holomorphic map $\Si\!\lra\!X$.
\end{proof}

\subsection{Gromov's convergence}
\label{Conver_subs2}

\noindent
We next show that a sequence of maps as in Corollary~\ref{EnerPres2_crl} gives rise 
to a continuous map from a \textsf{tree of spheres} attached at $0\!\in\!B_1$, 
i.e.~a connected union of spheres that have a distinguished, base component and
no loops;
the distinguished component will be attached at $\i\!\in\!S^2$ to $0\!\in\!B_1$.
The combinatorial structure of such a tree is described by 
a finite \textsf{rooted linearly ordered set}, i.e.~a partially ordered set~$(I,\prec)$ 
such~that 
\begin{enumerate}[label=(RS\arabic*),leftmargin=*]

\item\label{minelem_it} there is a minimal element (\textsf{root}) $i_0\!\in\!I$,
i.e.~$i_0\!\prec\!h$ for every $h\!\in\!I\!-\!\{i_0\}$, and 

\item\label{smalelem_it} 
for all $h_1,h_2,i\!\in\!I$ with $h_1,h_2\!\prec\!i$,   
either $h_1\!=\!h_2$, or $h_1\!\prec\!h_2$, or $h_2\!\prec\!h_1$.

\end{enumerate}
For each $i\!\in\!I\!-\!\{i_0\}$, let $p(i)\!\in\!I$
denote \textsf{the immediate predecessor of~$i$},
i.e.~$p(i)\!\in\!I$ such that $h\!\prec\!p(i)\!\prec\!i$ for all 
$h\!\in\!I\!-\!\{p(i)\}$ such that $h\!\prec\!i$.
Such $p(i)\!\in\!I$ exists by~\ref{minelem_it} and is unique by~\ref{smalelem_it}.
In the first diagram in Figure~\ref{glos_fig},
the vertices (dots) represent the elements of a rooted linearly ordered set~$(I,\prec)$
and the edges run from $i\!\in\!I\!-\!\{i_0\}$ down to~$p(i)$.\\

\noindent
Given a finite rooted linearly ordered set~$(I,\prec)$ with minimal element~$i_0$ and a function
\BE{zcond_e}z\!:\,I\!-\!\{i_0\}\lra\C, \quad i\lra z_i,
\quad\hbox{s.t.}\quad \big(p(i_1),z_{i_1}\big)\neq \big(p(i_2),z_{i_2}\big)
~~\forall~i_1,i_2\in I\!-\!\{i_0\},~i_1\!\neq\!i_2,\EE
let
$$\Si=\bigg(\bigsqcup_{i\in I}\{i\}\!\times\!S^2\!\bigg)\!\Big/\!\!\sim, \qquad
(i,\i)\sim\big(p(i),z_i)~~~\forall~i\!\in\!I\!-\!\{i_0\};$$
see the second diagram in Figure~\ref{glos_fig}.
Thus, the tree of spheres~$\Si$ is obtained by attaching~$\i$ in the sphere indexed by~$i$
to~$z_i$ in the sphere indexed by~$p(i)$.
The last condition in~\eref{zcond_e} insures that $\Si$ is a \textsf{nodal} Riemann surface,
i.e.~each non-smooth point (\textsf{node}) has only two local branches 
(pieces homeomorphic to~$\C$).

\begin{figure}
\begin{pspicture}(-.5,-2)(10,2.5)
\psset{unit=.4cm}
\pscircle*(10,-3){.25}\pscircle*(7,-2.5){.25}\pscircle*(13,-2.5){.25}
\psline[linewidth=.09](7,-2.5)(10,-3)\psline[linewidth=.09](13,-2.5)(10,-3)
\pscircle*(11,-.5){.25}\pscircle*(15,-.5){.25}
\psline[linewidth=.09](13,-2.5)(11,-.5)\psline[linewidth=.09](13,-2.5)(15,-.5)
\rput(10,-4){\sm{$i_0$}}\rput(13,-3.5){\sm{$p(i)$}}\rput(10.5,-.5){\sm{$i$}}
\pscircle[linewidth=.1](25,-3){2}\pscircle*(25,-5){.25}\rput(25,-4.3){\sm{$\i$}}
\pscircle*(23.59,-1.59){.25}\pscircle*(26.41,-1.59){.25}
\pscircle[linewidth=.1](22.18,-.18){2}\pscircle[linewidth=.1](27.82,-.18){2}
\pscircle*(27.82,1.82){.25}\pscircle*(29.82,-.18){.25}
\pscircle[linewidth=.1](27.82,3.82){2}\pscircle[linewidth=.1](31.82,-.18){2}
\rput(27.82,4.2){\sm{$i$}}\rput(27.82,2.5){\sm{$\i$}}
\rput(27.82,-.5){\sm{$p(i)$}}\rput(27.82,1.1){\sm{$z_i$}}
\end{pspicture}
\caption{A rooted linearly ordered set and an associated tree of spheres}
\label{glos_fig}
\end{figure}
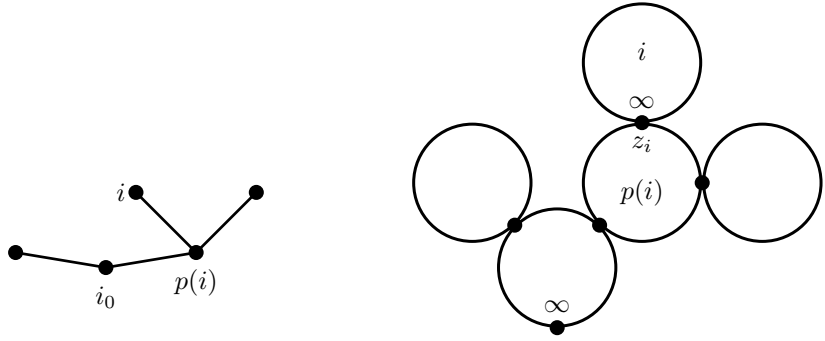

\begin{prp}\label{EnerPres_prp}
Let $(X,J)$ be a compact almost complex manifold with a Riemannian metric~$g$
and $u_i\!:B_1\!\lra\!X$ be a sequence of $J$-holomorphic maps converging
uniformly in the $C^{\i}$-topology on compact subsets of~$B_1^*$
to a $J$-holomorphic map $u\!:B_1\!\lra\!X$. If the limit
\BE{fmdfn_e2}\fm\equiv\lim_{\de\lra0}\lim_{i\lra\i}E_g(u_i;B_{\de})\EE
exists and is nonzero, then there exist
\begin{enumerate}[ref=\arabic*,label=(\arabic*),leftmargin=*]

\item a nodal Riemann surface $(\Si_{\i},\fj_{\i})$ consisting of $B_1$ with 
a tree of Riemann spheres~$\P^1$ attached at $0\!\in\!B_1$,

\item a $J$-holomorphic map $u_{\i}\!:\Si_{\i}\!\lra\!X$,
 
\item a subsequence of $\{u_i\}$ still denoted by $\{u_i\}$, and

\item\label{psi_it} a biholomorphic map $\psi_i\!:U_i\!\lra\!B_{1/2}$, 
where $U_i\!\subset\!\C$ is an open subset,  

\end{enumerate}
such that 
\begin{enumerate}[label=(\ref{psi_it}\alph*),leftmargin=*]

\item\label{Ccov_it} $E_g(u_{\i};\Si_{\i}\!-\!B_1)=\fm$, 
$U_i\!\subset\!U_{i+1}$, and $\C=\bigcup_{i=1}^{\i}U_i$,

\item\label{vcon_it} $u_i\!\circ\!\psi_i$ converges to~$u_{\i}$ uniformly 
in the $C^{\i}$-topology  on compact subsets of the complement of the nodes 
$\i,w_1^*,\ldots,w_k^*$  in the sphere~$\P_0^1$ attached at $0\!\in\!B_1$,

\item\label{stab_it}  if $u_{\i}|_{\P^1_0}$ is constant, 
$\P^1_0$ contains at least three nodes of~$\Si_{\i}$;

\item \eref{psi_it} applies with $B_1$, $(\{u_i\},0)$,  and~$\fm$
replaced by  a neighborhood of~$w_r^*$ in~$\C$, $(\{u_i\!\circ\!\psi_i\},w_r^*)$, and 
\BE{fmdfn_e3}\fm_r'\equiv\lim_{\de\lra0}\lim_{i\lra\i}
E_g\big(u_i\!\circ\!\psi_i;B_{\de}(w_r^*)\big),\EE
respectively, for each $r\!=\!1,\ldots,k$.

\end{enumerate}
\end{prp}

\begin{proof}
Let $\hb_{J,g}$ be the smaller of the numbers $\hb_{J,g}$
in Corollaries~\ref{LowEner_crl} and~\ref{EnerPres2_crl}.
In particular, $\fm\!\ge\!\hb_{J,g}$.\\

\noindent
For each $i\!\in\!\Z^+$ sufficiently large, choose $z_i\!\in\!\ov{B_{1/2}}$ so~that
\BE{maxdui_e} \max_{z\in \ov{B_{1/2}}}\big|\nd u_i\big|_g=\big|\nd_{z_i}u_i\big|_g.\EE
Since $u_i$ converges
uniformly in the $C^{\i}$-topology on compact subsets of~$B_1^*$ to~$u$,
$z_i\!\lra\!0$ as $i\!\lra\!\i$.
Thus,  $B_{1/2}(z_i)\!\subset\!B_1$
for all $i\!\in\!\Z^+$ sufficiently large.
By Lemma~\ref{EnerPres1_lmm}\ref{dela_it},
for all $i\!\in\!\Z^+$ sufficiently large there exists $\de_i\!\in\!(0,1/2)$
such~that 
\BE{dejidfn_e} E_g\big(u_i;B_{\de_i}(z_i)\big)=\fm-\frac{\hb_{J,g}}{2}\,.\EE
Define
$$\psi_i\!:U_i\!\equiv\!\big\{w\!\in\!\C\!: z_i\!+\!\de_iw\!\in\!B_{1/2}\big\}
\lra B_{1/2}
\qquad\hbox{by}\qquad \psi_i(w)=z_i\!+\!\de_iw\,.$$
Since $\de_i\!\lra\!0$, the second property in \ref{Ccov_it} holds.\\

\noindent
For each $i\!\in\!\Z^+$ sufficiently large, let
$$v_i=u_i\!\circ\!\psi_i\!:  U_i\lra X.$$
Since $u_i$ is $J$-holomorphic and $\psi_i$ is biholomorphic onto its image,
$v_i$ is a $J$-holomorphic map with \hbox{$E_g(v_i)=E_g(u_i;B_{1/2})$}.
Along with Lemma~\ref{EnerPres1_lmm}\ref{SeqEner_it}, this implies that
$$\lim_{i\lra\i}E_g(v_i)=\fm(1/2)<\i\,.$$
By Lemma~\ref{EnerPres3_lmm}, there thus exist a finite collection 
$w_1^*,\ldots,w_k^*\!\in\!\C$ of distinct points 
and a subsequence of~$\{u_i\}$, still denoted by~$\{u_i\}$, 
such that $v_i$ converges uniformly in the $C^{\i}$-topology
on compact subsets  of $\P^1\!-\!\{\i,w_1^*,\ldots,w_{\ell}^*\}$
to a $J$-holomorphic map $u\!:\P^1\!\lra\!X$.
In particular, \ref{vcon_it} holds and
$|\nd v_i|_g$ is uniformly bounded on compact subsets 
of $\P^1\!-\!\{\i,w_1^*,\ldots,w_{\ell}^*\}$.
We can also assume that the limit~\eref{fmdfn_e3}
exists for every $r\!=\!1,\ldots,k$.
We note~that
\BE{EnerBub_e}\begin{split}
E_g(v)+\!\sum_{r=1}^k\!\fm_r' 
&=\lim_{R\lra\i}\lim_{\de\lra0}\lim_{i\lra\i}\!\!
E_g\!\big(v_i,B_R\!-\!\bigcup_{r=1}^k\! B_{\de}(w_r^*)\big)
+\sum_{r=1}^k\lim_{\de\lra0}\lim_{i\lra\i}\!\!E_g\big(v_i;B_{\de}(w_r^*)\big)\\
&=\lim_{R\lra\i}\lim_{i\lra\i}\!\!E_g\big(v_i,B_R\big)
=\lim_{R\lra\i}\lim_{i\lra\i}\!\!E_g\big(u_i,B_{R\de_i}(z_i)\big)=\fm\,;
\end{split}\EE
the last equality holds by~\eref{noEloss_e1}.\\

\noindent
Let $\de_0\!\in\!\R^+$ be such that the balls $B_{\de_0}(w_r^*)$ are pairwise disjoint.
If 
$$\limsup_{i\lra\i}\max_{\ov{B_{\de_0}(w_r^*)}}\big|\nd v_i\big|<\i$$
for some~$r$, then $\{v_i\}$ converges  uniformly in the $C^{\i}$-topology
on $\ov{B_{\de_0}(w_r^*)}$ to~$v$ by the ellipticity of the $\dbar$-operator.
Thus, we can assume~that 
$$\lim_{i\lra\i}\sup_{\ov{B_{\de_0}(w_r^*)}}\big|\nd v_i\big|=\i$$
for every $r\!=\!1,\ldots,k$.
In light of  Corollary~\ref{EnerPres2_crl}\ref{fmLowBnd_it}, $\fm_r'\!\ge\!\hb_{J,g}$.\\

\begin{figure}
\begin{pspicture}(-.5,-2.5)(10,3.5)
\psset{unit=.4cm}
\pscircle[linewidth=.1](18,1){3}\pscircle*(21,1){.2}\rput(21.5,1){\sm{$1$}}
\pscircle[linewidth=.07,linestyle=dotted](18,1){7}\pscircle*(25,1){.2}
\rput(26.5,1){\sm{$\de_0/\de_i$}}
\rput(18,1){\sm{$\fm-\frac{\hb}{2}$}}\rput(18,-4){\sm{$\frac{\hb}{2}-\ep_i$}}
\end{pspicture}
\caption{The energy distribution of the rescaled map~$v_i$ in the proof of 
Proposition~\ref{EnerPres_prp}}
\label{EnerDist_fig}
\end{figure}
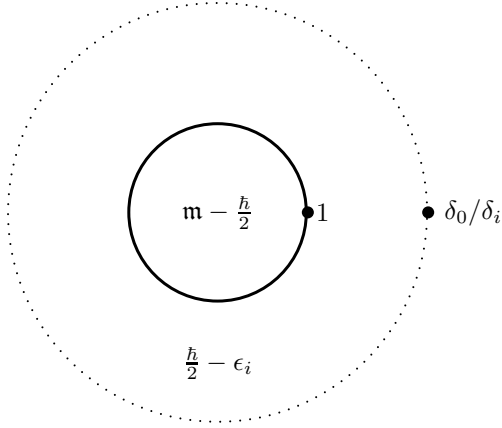

\noindent
We next show that $u(0)\!=\!v(\i)$, i.e.~that the bubble $(\P_0^1,v)$ 
connects to $(B_1,u)$ at~$z\!=\!0$.
Note~that
\begin{equation*}\begin{split}
d_g\big(u(0),v(\i)\big)
&=\lim_{R\lra\i}\lim_{\de\lra0}d_g\big(u(\de),v(R)\big)
=\lim_{R\lra\i}\lim_{\de\lra0}\lim_{i\lra\i}d_g\big(u_i(z_i\!+\!\de),v_i(R)\big)\\
&=\lim_{R\lra\i}\lim_{\de\lra0}\lim_{i\lra\i}
d_g\big(u_i(z_i\!+\!\de),u_i(z_i\!+\!R\de_i)\big)\\
&\le \lim_{R\lra\i}\lim_{\de\lra0}\lim_{i\lra\i}
\diam_g\big(u_i(B_{\de}(z_i)\!-\!B_{R\de_i}(z_i))\big).
\end{split}\end{equation*}
Along with~\eref{noEloss_e2}, this implies that $u(0)\!=\!v(\i)$.\\

\noindent
Suppose $v\!:\P^1\!\lra\!X$ is a constant map.
By~\eref{EnerBub_e}, $k\!\ge\!1$ and so there exists $w^*\!\in\!\C$ such that 
$|\nd_{w^*}v_i|\!\lra\!\i$ as $i\!\lra\!\i$.
By~\eref{maxdui_e} and the definition of~$\psi_i$, 
$|\nd_0v_i|\!\ge\!|\nd_wv_i|$ for all $w\!\in\!\C$ contained in the domain of~$v_i$
and so $|\nd_0v_i|\!\lra\!\i$ as $i\!\lra\!\i$.
By~\eref{fmdfn_e3} and~\eref{dejidfn_e}, 
$$\fm_0'\equiv\lim_{\de\lra0}\lim_{i\lra\i}
E_g(v_i)
\le \lim_{i\lra\i} E_g\big(v_i;B_1\big)
=\lim_{i\lra\i} E_g\big(u_i;B_{\de_i}(z_i)\big)
= \fm-\frac{\hb}{2}<\fm,$$
and so $k\!\ge\!2$, as claimed in~\ref{stab_it}.
Since the amount of energy of $v_i$ contained in $\C\!-\!B_1$ approaches $\hb_{J,g}/2$,
as illustrated in Figure~\ref{EnerDist_fig},
there must be in particular a bubble point~$w_r^*$ with $|w_r^*|\!=\!1$,
though this is not material.\\

\noindent
The above establishes Proposition~\ref{EnerPres_prp} whenever $k\!=\!0$ by
taking 
$$u_{\i}\big|_{B^1}=u \qquad\hbox{and}\qquad u_{\i}\big|_{\P^1_0}=v.$$
Since $\fm_r'\!\ge\!\hb_{J,g}$ for every~$r$, $k\!=\!0$ if $\fm\!<\!2\hb_{J,g}$.
If $k\!\ge\!1$, $\fm_r'\!\le\!\fm\!-\!\hb_{J,g}$ by~\eref{EnerBub_e}
because 
\hbox{$E_g(v)\!\ge\!\hb_{J,g}$} if $v$ is not constant by Corollary~\ref{LowEner_crl}
and $k\!\ge\!2$ otherwise by the above.
Thus, by induction on $[\fm/\hb_{J,g}]\!\in\!\Z^+$, we can assume that 
Proposition~\ref{EnerPres_prp} holds when applied to~$\{v_i\}$ on 
$B_{\de_0}(w_r^*)\!\subset\!\C$ with $r\!=\!1,\ldots,k$.
This yields a tree~$\Si_r$ of Riemann spheres~$\P^1$ with a distinguished smooth point~$\i$
and a $J$-holomorphic map $v_r\!:\Si_r\!\lra\!X$ such $v_r(\i)\!=\!v(w_r^*)$ and
$E_g(v_r)\!=\!\fm_r'$.
Combining the last equality with~\eref{EnerBub_e}, we obtain
$$E_g(v)+\!\sum_{r=1}^k\!E_g(v_r)=\fm\,.$$
Identifying $\i$ in the base sphere of each~$\Si_r$ with $w_r^*\!\in\!\P_0^1$,
which has been already attached to $0\!\in\!B_1^*$, we obtain
a $J$-holomorphic map $u_{\i}\!:\Si_{\i}\!\lra\!X$ with the desired properties;
see Figure~\ref{GrLim_fig2}.
\end{proof}

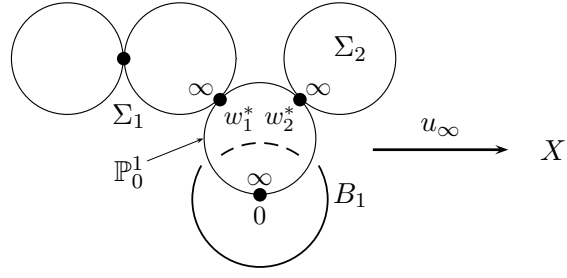
\begin{figure}
\begin{pspicture}(.5,0.5)(9,4)
\psset{unit=.3cm}
\psarc[linewidth=.08](25,3.8){3}{150}{30}
\psarc[linewidth=.08,linestyle=dashed](25,3.3){3}{55}{125}
\pscircle[linewidth=.05](25,6.5){2.5}
\pscircle[linewidth=.05](28.54,10.04){2.5}
\pscircle[linewidth=.05](21.46,10.04){2.5}
\pscircle[linewidth=.05](16.46,10.04){2.5}
\pscircle*(25,4){.3}\pscircle*(18.96,10.04){.3}
\pscircle*(26.77,8.22){.3}\pscircle*(23.23,8.22){.3}
\rput(27.6,8.7){\sm{$\i$}}\rput(22.4,8.7){\sm{$\i$}}
\rput(25,4.7){\sm{$\i$}}\rput(25,3.1){\sm{$0$}}
\rput(25.9,7.4){\sm{$w_2^*$}}\rput(24.1,7.4){\sm{$w_1^*$}}
\rput(29,4){$B_1$}\rput(29,10.5){$\Si_2$}\rput(19.2,7.5){$\Si_1$}
\rput(19.2,5){$\P^1_0$}
\psline[linewidth=.12]{->}(30,6)(36,6)\rput(38,6){$X$}
\psline[linewidth=.03]{->}(20,5.5)(22.5,6.5)
\rput(33,7){$u_{\i}$}
\end{pspicture}
\caption{Gromov's limit of a sequence of $J$-holomorphic maps $u_i\!:B_1\!\lra\!X$}
\label{GrLim_fig2}
\end{figure}

\begin{proof}[{\bf\emph{Proof of Theorem~\ref{GromConv_thm1}}}]
Fix a Riemannian metric~$g_{\Si}$ on~$\Si$.
For $z\!\in\!\Si$ and $\de\!\in\!\Si$, 
let $B_{\de}(z)\!\subset\!\Si$ denote the ball of radius~$\de$  around~$z$.\\

\noindent
By Lemma~\ref{EnerPres3_lmm}, there exist a finite collection 
$z_1^*,\ldots,z_{\ell}^*\!\in\!\Si$ of distinct points 
and a subsequence of~$\{u_i\}$, still denoted by~$\{u_i\}$, 
such that $u_i$ converges uniformly in the $C^{\i}$-topology
on compact subsets  of $\Si\!-\!\{z_1^*,\ldots,z_{\ell}^*\}$
to a $J$-holomorphic map $u\!:\Si\!\lra\!X$.
In particular, $|\nd u_i|_g$ is uniformly bounded on compact subsets 
of $\Si\!-\!\{z_1^*,\ldots,z_{\ell}^*\}$.
We can also assume that  the~limit
$$\fm_j\equiv\lim_{\de\lra0}\lim_{i\lra\i}E_g\big(u_i;B_{\de}(z_j^*)\big)$$
exists for every $j\!=\!1,\ldots,\ell$.
We note that 
\BE{Egtot_e}\begin{split} 
E_g(u)+\sum_{j=1}^{\ell}\fm_j&=
\lim_{\de\lra0}\lim_{i\lra\i}E_g\big(u;\Si\!-\!\bigcup_{j=1}^{\ell}B_{\de}(z_j^*)\big)
+\sum_{j=1}^{\ell}\lim_{\de\lra0}\lim_{i\lra\i}E_g\big(u_i;B_{\de}(z_j^*)\big)\\
&=\lim_{\de\lra0}\lim_{i\lra\i}E_g(u_i)=\lim_{i\lra\i}E_g(u_i).\\
\end{split}\EE

\vspace{.1in}

\noindent
Let $\de_0\!\in\!\R^+$ be such that the balls $B_{\de_0}(z_i^*)$ are pairwise disjoint.
If 
$$\limsup_{i\lra\i}\max_{\ov{B_{\de_0}(z_j^*)}}\big|\nd u_i\big|<\i$$
for some~$j$, then $\{u_i\}$ converges  uniformly in the $C^{\i}$-topology
on $\ov{B_{\de_0}(z_j^*)}$ to~$u$ by the ellipticity of the $\dbar$-operator.
Thus, we can assume~that 
$$\lim_{i\lra\i}\sup_{\ov{B_{\de_0}(z_i^*)}}\big|\nd u_i\big|=\i$$
for every $j\!=\!1,\ldots,\ell$.\\

\noindent
For each $j\!=\!1,\ldots,\ell$, Proposition~\ref{EnerPres_prp} provides
a tree~$\Si_j$ of Riemann spheres~$\P^1$ with a distinguished smooth point~$\i$
and a $J$-holomorphic map $v_j\!:\Si_j\!\lra\!X$ such $v_j(\i)\!=\!v(w_r^*)$ and
$E_g(v_j)\!=\!\fm_j$.
Combining the last equality with~\eref{Egtot_e}, we obtain
$$E_g(v)+\!\sum_{j=1}^{\ell}\!E_g(v_j)=\lim_{i\lra\i}E_g(u_i)\,.$$
Identifying the distinguished point~$\i$ of each~$\Si_j$ with $z_j^*\!\in\!\Si$,
we obtain a Riemann surface $(\Si_{\i},\fj_{\i})$ and 
a $J$-holomorphic map $u_{\i}\!:\Si_{\i}\!\lra\!X$ with the desired properties.\\

\noindent
If $\Si\!=\!\P^1$ and the limit map~$u$ above is constant,
then $\ell\!\ge\!1$ by~\eref{Egtot_e}.
Suppose $\ell\!\in\!\{1,2\}$. 
Let
$$M_i=\sup_{\ov{B_{\de_0}(z_1^*)}}\big|\nd u_i\big|$$
and parametrize~$\P^1$ so that $z_1^*\!=\!0$.
Define 
$$h_i\!: \P^1\lra\P^1, \qquad  h_i(z)=z_i\!+\!z/M_i,$$
and apply the preceding argument with $u_i$ replaced by~$u_i\!\circ\!h_i$.
By the proof of Corollary~\ref{EnerPres2_crl}\ref{fmLowBnd_it}, the limiting map $u|_{\Si}$
is then non-constant and $(\Si_{\i},\fj_{\i},u)$ is a stable $J$-holomorphic map.
\end{proof}

\subsection{An example}
\label{PnEg_subs}

\noindent
We now give an example illustrating Gromov's convergence 
in a classical setting.\\

\noindent
Let $n\!\in\!\Z^+$, with $n\!\ge\!2$, and $\P^{n-1}\!=\!\C\P^{n-1}$.
Denote by $\ell$ the positive generator of $H_2(\P^{n-1};\Z)\!\approx\!\Z$,
i.e.~the homology class represented by the standard $\P^1\!\subset\!\P^{n-1}$.
A \textsf{degree~$d$ map} $f\!:\P^1\!\lra\!\P^{n-1}$ is a continuous map
such that $f_*[\P^1]\!=\!d\ell$.
A holomorphic degree~$d$ map $f\!:\P^1\!\lra\!\P^{n-1}$ is given~by
$$[u,v]\lra\big[R_1(u,v),\ldots,R_n(u,v)\big]$$ 
for some degree~$d$ homogeneous polynomials $R_1,\ldots,R_d$ on~$\C^2$
without a common linear factor.
Since the tuple $(\la R_1,\ldots,\la R_n)$ determines the same map
as $(R_1,\ldots,R_n)$ for any $\la\!\in\!\C^*$, 
the space of degree~$d$ holomorphic maps $f\!:\P^1\!\lra\!\P^{n-1}$ is
a dense open subset of
$$\fX_{n,d}\equiv \big((\Sym^d\C^2)^n-\{0\}\big)\big/\C^* 
\approx\P^{(d+1)n-1}\,.$$

\vspace{.1in}

\noindent
Suppose $f_k\!: \P^1\!\lra\!\P^{n-1}$ is a sequence of holomorphic degree $d\!\ge\!1$ maps
and
$$\bR_k=\big[R_{k;1},\ldots,R_{k;n}\big] \in\fX_{n,d}$$
are the associated equivalence classes of $n$-tuples of 
homogeneous polynomials  without a common linear factor.
Passing to a subsequence, we can assume that $[\bR_k]$ converges to some 
\BE{bRdfn_e}\bR\equiv\big[(v_1u\!-\!u_1v)^{d_1}\ldots(v_mu\!-\!u_mv)^{d_m}S_1,
\ldots,(v_1u\!-\!u_1v)^{d_1}\ldots(v_mu\!-\!u_mv)^{d_m}S_n\big]\in\fX_{n,d}\,,\EE
with $d_1,\ldots,d_m\!\in\!\Z^+$ and homogeneous polynomials
$$\bS\equiv [S_1,\ldots,S_n]\in \fX_{n,d_0}$$
without a common linear factor and with $d_0\!\in\!\Z^{\ge0}$.
By~\eref{bRdfn_e},
$$d_0+d_1+\ldots+d_m=d.$$
Rescaling $(R_{k;1},\ldots,R_{k;n})$, we can assume that 
\BE{Pneg_e3}
\lim_{k\lra\i}R_{k;i}=(v_1u\!-\!u_1v)^{d_1}\ldots(v_mu\!-\!u_mv)^{d_m}S_i
\qquad\forall\,i\!=\!1,\ldots,n.\EE

\vspace{.1in}

\noindent
Suppose $z_0\!\in\!\C\!-\!\{u_1/v_1,\ldots,u_m/v_m\}$.
Since the polynomials $S_1,\ldots,S_n$ do not have a common linear factor,
$S_{i_0}(z_0,1)\!\neq\!0$ for some $i_0\!=\!1,\ldots,n$.
This implies that $R_{k;i_0}(z_0,1)\!\neq\!0$ for all $k$ large enough and so
$$\lim_{k\lra\i}\frac{R_{k;i}(z,1)}{R_{k;i_0}(z,1)}
=\frac{\lim\limits_{k\lra\i}R_{k;i}(z,1)}{\lim\limits_{k\lra\i}R_{k;i_0}(z,1)}
=\frac{(v_1z\!-\!u_1)^{d_1}\ldots(v_mz\!-\!u_m)^{d_m}S_i(z,1)}
{(v_1z\!-\!u_1)^{d_1}\ldots(v_mz\!-\!u_m)^{d_m}S_{i_0}(z,1)}
=\frac{S_i(z,1)}{S_{i_0}(z,1)}\,$$
for all $i\!=\!1,\ldots,n$ and $z$ close to~$z_0$.
Furthermore, the convergence is uniform on a neighborhood of~$z_0$.
Thus, the sequence $f_k$ $C^{\i}$-converges on compact subsets of 
$\P^1\!-\!\{[u_1,v_1],\ldots,[u_m,v_m]\}$ to 
the holomorphic degree~$d_0$ map $g\!:\P^1\!\lra\!\P^{n-1}$ determined by~$\bS$.\\

\noindent
Let $\om$ be the Fubini-Study symplectic form on~$\P^{n-1}$ normalized so that $\lr{\om,\ell}\!=\!1$
and $E(\cdot)$ be the energy of maps into~$\P^{n-1}$ with respect to the associated
Riemannian metric.
For each $\de\!>\!0$ and $j\!=\!1,\ldots,m$, 
denote by $B_{\de}([u_j,v_j])$ the ball of radius~$\de$ around $[u_j,v_j]$
in~$\P^1$ and let
$$\P^1_{\de}=\P^1-\bigcup_{j=1}^mB_{\de}([u_j,v_j])\,.$$
For each $j\!=\!1,\ldots,m$, let
$$\fm_{[u_j,v_j]}\big(\{f_k\}\big)=\lim_{\de\lra0}\lim_{k\lra\i}\!
E\big(f_k;B_{\de}([u_j,v_j])\big)\in\R^{\ge0}$$
be the energy sinking into the bubble point $[u_j,v_j]$.
By Theorem~\ref{GromConv_thm1}, the number $\fm_{[u_j,v_j]}(\{f_k\})$ is the value of~$\om$
on some element of~$H_2(\P^{n-1};\Z)$, i.e.~an integer. 
Below we show that $\fm_{[u_j,v_j]}(\{f_k\})\!=\!d_j$.\\

\noindent
Since the sequence $f_k$ $C^{\i}$-converges 
to the degree~$d_0$ map $g\!:\P^1\!\lra\!\P^{n-1}$  on compact subsets of 
$\P^1\!-\!\{[u_1,v_1],\ldots,[u_m,v_m]\}$,
$$d_0=\lr{\om,d_0\ell}=E(g)
=\lim_{\de\lra0}E_g\big(g;\P^1_{\de}\big)
=\lim_{\de\lra0}\lim_{k\lra\i}E\big(f_k;\P^1_{\de}\big).$$
Thus,
\begin{equation*}\begin{split}
\sum_{j=1}^m \fm_{[u_j,v_j]}\big(\{f_k\}\big)
&=\sum_{j=1}^m\lim_{\de\lra0}\lim_{k\lra\i}\!
E\big(f_k;B_{\de}([u_j,v_j])\big)
=\lim_{\de\lra0}\lim_{k\lra\i}
E\big(f_k;\bigcup_{j=1}^m\!B_{\de}([u_j,v_j])\big)\\
&=\lim_{\de\lra0}\lim_{k\lra\i\!}
\big(E_g(f_k)\!-\!E_g\big(f_k;\P^1_{\de}\big)\big)
=d-d_0=d_1+\ldots+d_m\,.
\end{split}\end{equation*}
In particular, $\fm_{[u_j,v_j]}(\{f_k\})\!=\!d_j$ if $m\!=\!1$, no matter what 
the ``residual" tuple of polynomials $\bS$~is.
We use this below to establish this energy identity for $m\!>\!1$ as well.\\

\noindent
By~\eref{Pneg_e3}, for all $k\!\in\!\Z^+$ sufficiently large
there exist $\la_{k;i;j;p}\!\in\!\C$ with  $i\!=\!1,\ldots,n$, $j\!=\!1,\ldots,m$, 
and $p\!=\!1,\ldots,d_j$ and tuples
$$\bS_k\equiv \big[S_{k;1},\ldots,S_{k;n}\big]\in \fX_{n;d_0}$$
of polynomials without a common linear factor such that 
\begin{gather*}
\lim_{k\lra\i}\bS_k=\bS, \qquad \lim_{k\lra\i}\la_{k;i;j;p}=1~~\forall\,i,j,p,\\
R_{k;i}(u,v)=\prod_{j=1}^m\prod_{p=1}^{d_j}(v_ju\!-\!\la_{k;i;j;p}u_jv) 
\cdot S_{k;i}(u,v)\quad\forall~k,i\,.
\end{gather*}
For each $j_0\!=\!1,\ldots,m$, let 
$$\bT_{j_0}\equiv 
\big[T_{j_0;1},\ldots,T_{j_0;n}\big]\in \fX_{n;d-d_{j_0}}$$
be a tuple of polynomials without a common linear factor.
If in addition, $i\!=\!1,\ldots,n$, $\ep\!\in\!\R$, and $k\!\in\!\Z^+$, let
\begin{alignat*}{2}
S_{i;j_0;\ep}(u,v)&\equiv \prod_{j\neq j_0}^m(v_ju\!-\!u_jv)^{d_j} 
\cdot S_i(u,v)+
\ep T_{j_0;i}(u,v)\,, &\qquad &i=1,\ldots,n,\\
R_{k;i;j_0;\ep}(u,v)&\equiv R_{k;i}(u,v)+
\ep\prod_{p=1}^{d_{j_0}}(v_{j_0}u\!-\!\la_{k;i;j_0;p}u_{j_0}v) 
\cdot T_{j_0;i}(u,v), &\qquad &i=1,\ldots,n.
\end{alignat*}

\vspace{.1in}

\noindent
The polynomials within each tuple $(S_{i;j_0;\ep})_{i=1,\ldots,n}$ and 
$(R_{k;i;j_0;\ep})_{i=1,\ldots,n}$ have no common linear factor for 
all $\ep\!\in\!\R^+$ sufficiently small and $k$ sufficiently large 
(with the conditions on~$\ep$ and~$k$ mutually independent).
We denote by 
$$f_{k;j_0;\ep}\!:\P^1\!\lra\!\P^{n-1}$$ 
the holomorphic degree~$d$ map determined by the~tuple
$$\bR_{k;j_0;\ep}\equiv\big[R_{k;1;j_0;\ep},\ldots,R_{k;n;j_0;\ep}\big]\,.$$
Since 
$$\lim_{k\lra\i}\bR_{k;j_0;\ep}
=\big[(v_1u\!-\!u_1v)^{d_{j_0}}S_{1;j_0;\ep},\ldots, (v_1u\!-\!u_1v)^{d_{j_0}}S_{n;j_0;\ep}\big]
\in\fX_{n;d}$$
and the polynomials $S_{1;j_0;\ep},\ldots,S_{n;j_0;\ep}$ have no linear factor in common,
\BE{Pneg_e5}
\lim_{\de\lra0}\lim_{k\lra\i}\!E\big(f_{k;j_0;\ep};B_{\de}([u_{j_0},v_{j_0}])\big)
\equiv \fm_{[u_{j_0},v_{j_0}]}\big(\{f_{k;j_0;\ep}\}\big) =d_{j_0}\EE
by the $m\!=\!1$ case established above.\\

\noindent
For $\de\!\in\!\R^+$ sufficiently small, 
$\ep\!\in\!\R^+$ sufficiently small, and $k$ sufficiently large,
$$\prod_{j\neq j_0}^m\prod\limits_{p=1}^{d_j}(v_ju\!-\!\la_{k;i;j;p}u_jv) 
\cdot S_{k;i}(u,v)\neq0 \qquad\forall~[u,v]\!\in\!B_{2\de}\big([u_{j_0},v_{j_0}]\big).$$
Thus, the ratios
$$\frac{R_{k;i;j_0;\ep}(u,v)}{R_{k;i}(u,v)}
=1+\ep\frac{T_{j_0;i}(u,v)}
{\prod\limits_{j\neq j_0}^m\prod\limits_{p=1}^{d_j}(v_ju\!-\!\la_{k;i;j;p}u_jv) 
\cdot S_{k;i}(u,v)}$$
converge uniformly to~1 on $B_{\de}([u_{j_0},v_{j_0}])$ as $\ep\!\lra\!0$.
Thus, there exists $k^*\!\in\!\Z^+$ such~that 
$$\lim_{\ep\lra0}\sup_{k\ge k^*}\sup_{z\in B_{\de}([u_{j_0},v_{j_0}])}
\bigg|\frac{|\nd_zf_{k;j_0;\ep}|}{|\nd_zf_k|}-1\bigg|=0.$$
It follows that
\begin{equation*}\begin{split}
\fm_{[u_{j_0},v_{j_0}]}\big(\{f_k\}\big)
&\equiv \lim_{\de\lra0}\lim_{k\lra\i}E\big(f_k;B_{\de}([u_{j_0},v_{j_0}])\big)
= \lim_{\de\lra0}\lim_{k\lra\i}\lim_{\ep\lra0}E\big(f_{k;j_0;\ep};B_{\de}([u_{j_0},v_{j_0}])\big)\\ 
&=\lim_{\ep\lra0}\lim_{\de\lra0}\lim_{k\lra\i}\!
E\big(f_{k;j_0;\ep};B_{\de}([u_{j_0},v_{j_0}])\big)
=\lim_{\ep\lra0}d_{j_0}=d_{j_0};
\end{split}\end{equation*}
the second-to-last equality above holds by~\eref{Pneg_e5}.\\

\noindent
Suppose that either $d_0\!\ge\!1$ or $m\!\ge\!3$.
Otherwise, the maps~$f_k$ can be reparametrized so that $d_0\!\neq\!0$;
see the last paragraph of the proof of Theorem~\ref{GromConv_thm1}
at the end of Section~\ref{Conver_subs2}.
By Theorem~\ref{GromConv_thm1} and the above, a subsequence of $\{f_k\}$ converges
to the equivalence class of a holomorphic degree~$d_0$ map $f\!:\Si\!\lra\!\P^{n-1}$,
where $\Si$ is a nodal Riemann surface consisting of the component~$\Si_0\!=\!\P^1$
corresponding to the original~$\P^1$ and finitely many trees of $\P^1$'s coming off 
from~$\Si_0$.
The maps on the components in the trees are defined only up reparametrization of
the domain.
By the above, $f|_{\Si_0}$ is the map~$g$ determined by the ``relatively prime part"
$\bS$ of the limit~$\bR$ of the tuples of polynomials.
The trees are attached at the roots~$[u_j,v_j]$ of the common linear factors 
$v_ju\!-\!u_jv$ of  the polynomials in~$\bR$;
the degree of the restriction of~$f$ to each tree is the power 
of the multiplicity~$d_j$ of the corresponding common linear factor.\\

\noindent
The same reasoning as above applies to the sequence of maps 
$$\big(\id_{\P^1},f_k\big)\!:\P^1\lra\P^1\!\times\!\P^{n-1}\,,$$
but the condition that either $d_0\!\ge\!1$ or $m\!\ge\!3$ is no longer
necessary for the analogue of the conclusion in the previous paragraph.
This implies~that the~map
$$\fM_{0,0}\big(\P^1\!\times\!\P^{n-1},(1,d)\big)\lra\fX_{n,d},
\qquad [f,g]\lra\big[g\!\circ\!f^{-1}\big],$$
from the subspace of $\ov\fM_{0,0}(\P^1\!\times\!\P^{n-1},(1,d))$ corresponding
to maps from~$\P^1$ extends to a continuous surjective map
\BE{LinSig_e}\ov\fM_{0,0}\big(\P^1\!\times\!\P^{n-1},(1,d)\big)\lra \fX_{n,d}\,.\EE
In particular, Gromov's moduli spaces refine classical compactifications 
of spaces of holomorphic maps $\P^1\!\lra\!\P^{n-1}$.
On the other hand, the former are defined for arbitrary almost Kahler manifolds,
which makes them naturally suited for applying topological methods.
The right-hand side of~\eref{LinSig_e} is known as the \textsf{linear sigma model}
in the Mirror Symmetry literature.
The morphism~\eref{LinSig_e} plays a prominent role in the proof of 
mirror symmetry for the genus~0 Gromov-Witten invariants in~\cite{Gi} and~\cite{LLY};
see \cite[Section~30.2]{MirSym}.

\end{document}